\documentclass[12pt]{article}
\usepackage{etex,amsmath,amsthm, amssymb,latexsym,enumitem,hyperref}
\usepackage[title,titletoc,toc]{appendix}
\usepackage[pdftex]{graphicx,color}
\usepackage{caption}
\usepackage{subcaption}

\usepackage{hyperref}

\usepackage[T1]{fontenc}
\usepackage[utf8]{inputenc}

\usepackage{multicol}
\usepackage{times}
\usepackage{hyperref}

\DeclareMathAlphabet{\mathbbold}{U}{bbold}{m}{n}

\usepackage{amsthm}
\theoremstyle{definition}

\def\R{\mathbb R}

\def\frame#1{}

\newtheorem*{definition}{Definition}
\newtheorem{remark}{Remark}

\newtheorem{cor}{Corollary}
\newtheorem{conj}{Conjecture}
\newtheorem{thm}{Theorem}
\newtheorem{lemma}{Lemma}

\def\R{\mathbb R}

\catcode`@=11 \catcode`!=11

\expandafter\ifx\csname fiverm\endcsname\relax
  \let\fiverm\fivrm
\fi
  
\let\!latexendpicture=\endpicture 
\let\!latexframe=\frame
\let\!latexlinethickness=\linethickness
\let\!latexmultiput=\multiput
\let\!latexput=\put
 
\def\@picture(#1,#2)(#3,#4){%
  \@picht #2\unitlength
  \setbox\@picbox\hbox to #1\unitlength\bgroup 
  \let\endpicture=\!latexendpicture
  \let\frame=\!latexframe
  \let\linethickness=\!latexlinethickness
  \let\multiput=\!latexmultiput
  \let\put=\!latexput
  \hskip -#3\unitlength \lower #4\unitlength \hbox\bgroup}

\catcode`@=12 \catcode`!=12
\input pictex.tex
\catcode`@=11 \catcode`!=11
  
\let\!pictexendpicture=\endpicture 
\let\!pictexframe=\frame
\let\!pictexlinethickness=\linethickness
\let\!pictexmultiput=\multiput
\let\!pictexput=\put

\def\beginpicture{%
  \setbox\!picbox=\hbox\bgroup%
  \let\endpicture=\!pictexendpicture
  \let\frame=\!pictexframe
  \let\linethickness=\!pictexlinethickness
  \let\multiput=\!pictexmultiput
  \let\put=\!pictexput
  \let\input=\@@input   % \@@input is LaTeX's saved version of TeX's primitive
  \!xleft=\maxdimen  
  \!xright=-\maxdimen
  \!ybot=\maxdimen
  \!ytop=-\maxdimen}

\let\frame=\!latexframe

\let\pictexframe=\!pictexframe

\let\linethickness=\!latexlinethickness
\let\pictexlinethickness=\!pictexlinethickness

\let\\=\@normalcr
\catcode`@=12 \catcode`!=12
\newcommand{\ie}{{\it i.e. }}
\newcommand{\eg}{{\it e.g. }}

\usepackage{changes}
\definechangesauthor[name={Oleg}, color=red]{OK}
\newcommand{\Do}{\deleted[id=]}

\renewcommand\Do[1]{}

\definechangesauthor[name={Seb}, color=red]{S}

\newcommand\Ds[1]{}

\colorlet{Changes@Color}{red}

\author{Oleg Kozlovski and Sebastian van Strien}
\title{Asymmetric unimodal maps with non-universal period-doubling scaling laws}

\date{\today}

\begin{document}

\maketitle

\begin{abstract} We consider a family of strongly-asymmetric unimodal maps $\{f_t\}_{t\in [0,1]}$ of the form 
$f_t=t\cdot f$ where $f\colon [0,1]\to [0,1]$ is unimodal, $f(0)=f(1)=0$, $f(c)=1$ is of the form and $$f(x)=\left\{ \begin{array}{ll} 1-K_-|x-c|+o(|x-c|)& \mbox{ for }x<c, \\
1-K_+|x-c|^\beta + o(|x-c|^\beta) &\mbox{ for }x>c, \end{array}\right. $$
where we assume that $\beta>1$.  We show that such a family contains a Feigenbaum-Coullet-Tresser $2^\infty$ map, and develop a renormalization theory for these maps. The scalings of the renormalization intervals  of the $2^\infty$ map turn out to be super-exponential and non-universal (i.e. to depend on the map) and the scaling-law is different for odd and even steps of the renormalization.   The conjugacy between the attracting Cantor sets of two such maps  is smooth if and only if some invariant is satisfied.  We also show that the Feigenbaum-Coullet-Tresser map does not  have wandering intervals, but surprisingly we were only able to prove this using our rather detailed scaling results. 
\end{abstract}

\setcounter{tocdepth}{1}
\tableofcontents

\section{Introduction}

The theory of one-dimensional dynamics is rather well
developed. Especially a lot is known for smooth one-dimensional
unimodal maps (\ie maps of an interval having just one critical
point): absence of wandering intervals, real bounds, convergence of
renormalizations, density of hyperbolic maps, various scaling
properties, etc... Most of these
 %aforementioned 
 results are obtained
under some conditions on the order of the critical point, typically the map is assumed
to be smooth or even analytic, 
 and the critical 
point is assumed to be non-flat and 
in many results  the order is, additionally, assumed to be an 
even integer. Moreover,  in these studies, the order of the critical
point is assumed to be the same on  both sides, \ie in a small
neighbourhood of the critical point the map behaves as
$f(x)-f(c)  \sim -K|x-c|^\alpha$, where $c$ denotes the critical point
and $\alpha$ is its order. Here $\sim$ means that the left hand side
divided by the right hand side tends to $1$ as $x\to c$.

A natural generalisation and the next step in the theory of
one-dimensional maps is to consider maps which have different
critical orders on the two sides of the critical point. Specifically, to 
study maps such that near the critical point
the map $f$ takes the form 
$$f(x)-f(c) \sim
\left\{ \begin{array}{rl}  
           - K_-|x-c|^\alpha &\mbox{ for }x<c\\ 
           - K_+|x-c|^\beta &\mbox{ for }x>c\end{array} \right.
$$ 
where $1\le \alpha\le \beta$. 
Maps for which $\alpha<\beta$ deserve to be studied on their own merit and can appear in
applications, \eg the Poincare first return maps of smooth
two-dimensional flows or semi-flows can have singularities with different critical
order. We will call these maps {\it strongly asymmetric} when $\alpha<\beta$ and  {\em weakly symmetric} when $\alpha=\beta$.

The purpose of this  project is to ask the following question:  do strongly asymmetric maps have substantially
different properties when compared with {\lq}symmetric{\rq} unimodal maps?
In some cases the answer would be NO. For  example, hyperbolic maps will have
similar properties because the order of the critical point does not
play any role for such maps. A slightly  less trivial example is the case of
Misiurewicz maps (that is maps whose critical orbit does not
accumulate on the critical point) where the standard theory of
one-dimensional maps can be applied to  strongly asymmetric maps
without any  significant alteration.

%However, in principle it is rather unclear whether the properties of strongly asymmetric maps with a recurrent
%critical point can be significantly different.

At the start of this project on strongly asymmetric maps, the authors were not sure what  to expect in non-trivial cases. For example, 
could one expect universality? Could there be wandering intervals? 

In this paper 
we will make a first step towards a general theory for such maps by considering
one of the simplest non-trivial class of such maps, namely
infinitely renormalizable maps of the Feigenbaum-Coullet-Tresser
combinatorics and will show that the scaling properties and limits of
renormalizations are quite different compared to the classical
ones. Note that the theory of such infinitely renormalizable maps is
still far from  complete even in the case of maps with a
{\lq}symmetric{\rq} critical point when the order of the critical point is
not an even integer. Though it is generally believed that the
renormalizations should converge in this case, no proof is known. 

Before we formulate our results,  let us quickly discuss some obvious
differences between symmetric and asymmetric cases in the setting
of the Feigenbaum-Coullet-Tresser maps (which we will often call
{\em $2^\infty$ maps} or  {\em maps of $2^\infty$ combinatorics}). Recall that for
such a map one can construct a shrinking sequence of intervals
$[a_n, b_n]$   around the critical point such that the restrictions
$f^{2^n}|_{[a_n,b_n]}$ are unimodal maps (also with  $2^\infty$
combinatorics) for $n=0,1,\dots$.  Let $R_n: [a_n, b_n] \to [0,1]$ be linear  surjections  and let the
$n$-th renormalizations $\tilde f_n$ of $f$ be defined 
by the formula $\tilde f_n= R \circ f^{2^n} \circ R^{-1}$.

When the order of the  critical point of $f$ is an even integer it is known   that
the sequence of the renormalizations converges to some unimodal
real-analytic map which is universal in the sense that this limit map
depends only on the order of the critical point and not on the
particular choice of the initial map $f$,   for references see below.

Now let us check what happens with renormalizations when the map is
strongly asymmetric. First, note that the renormalization intervals
$[a_n,b_n]$ can be constructed in different ways. These differences
are non essential,  and we will find if convenient to assume  that $f(a_n)=f(b_n)$. Then
asymptotically we have
$|a_n-c| \sim (K_+/K_-)^{\frac 1\alpha} |b_n-c|^{\frac \beta\alpha}$, and
since $\alpha<\beta$ we see that $|a_n-c| \ll |b_n-c|$. Thus, the
critical point is located much closer to the left end of the
renormalization intervals and in the limit after rescaling the
critical point coincides with the left boundary point of the rescaled
interval. This means that the renormalizations cannot converge to a
unimodal map!  As we  will see 
 in the case we consider,
 when $\alpha=1<\beta$,
the limit of $\tilde f_n$ exists (even though it is degenerate), and moreover is universal  in the sense that it only depends on $\beta$.  There is even an
explicit formula for it!

\medskip 
To initiate this research direction we decided to focus on
strongly
asymmetric unimodal maps with  Feigenbaum-Coullet-Tresser
combinatorics. Though the authors believe that these results must hold
in the general case $1\le \alpha < \beta$, we were only able to prove them
under the assumption that $\alpha=1$   because in that case we are able to 
use the notion of  {\em semi-extension} which is defined in 
\S\ref{subsec:semiexten} and discussed a little more in the informal summary below.
 The precise definition of the class of
considered maps is given in Section~\ref{sec:settings}.

\subsubsection*{Informal summary of the the results in this paper.}

\begin{figure}
\centering
\begin{subfigure}{.3\textwidth}\includegraphics[width=4cm, height=5cm]{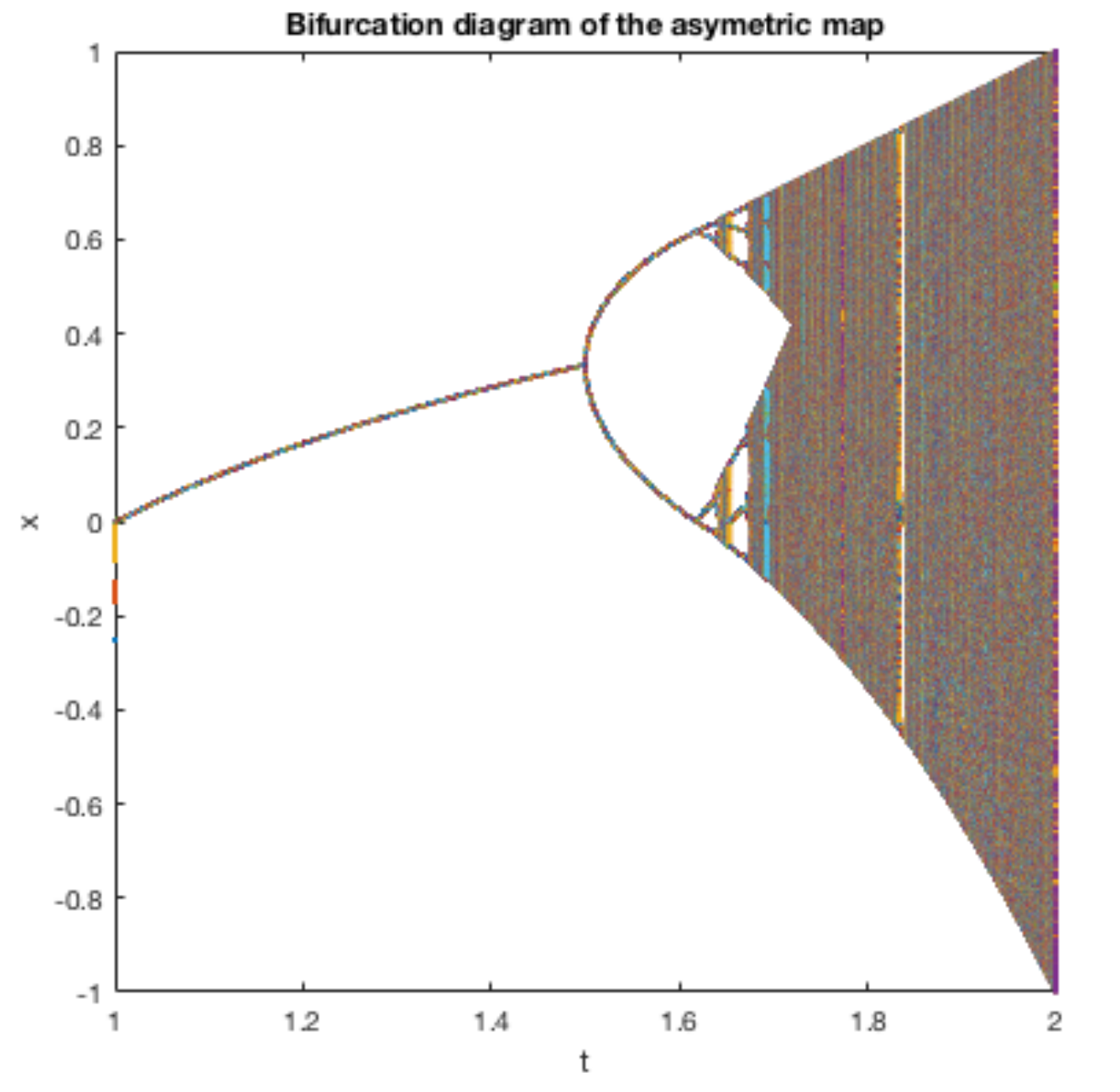}
\end{subfigure}%
 \centering
\begin{subfigure}{.3\textwidth}\includegraphics[width=4cm, height=5cm]{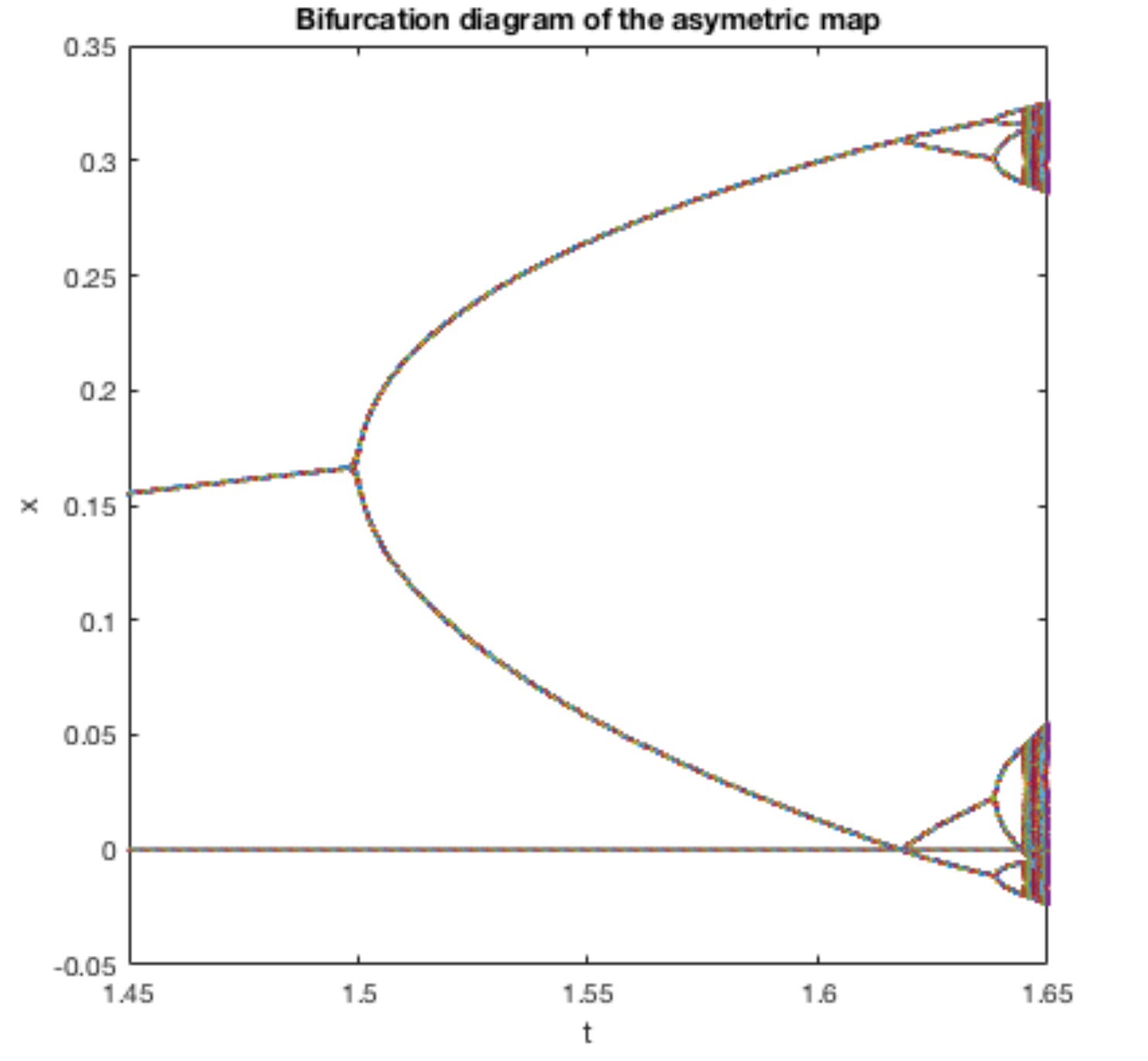}
\end{subfigure}
 \centering
\begin{subfigure}{.3\textwidth}\includegraphics[width=4cm, height=5cm]{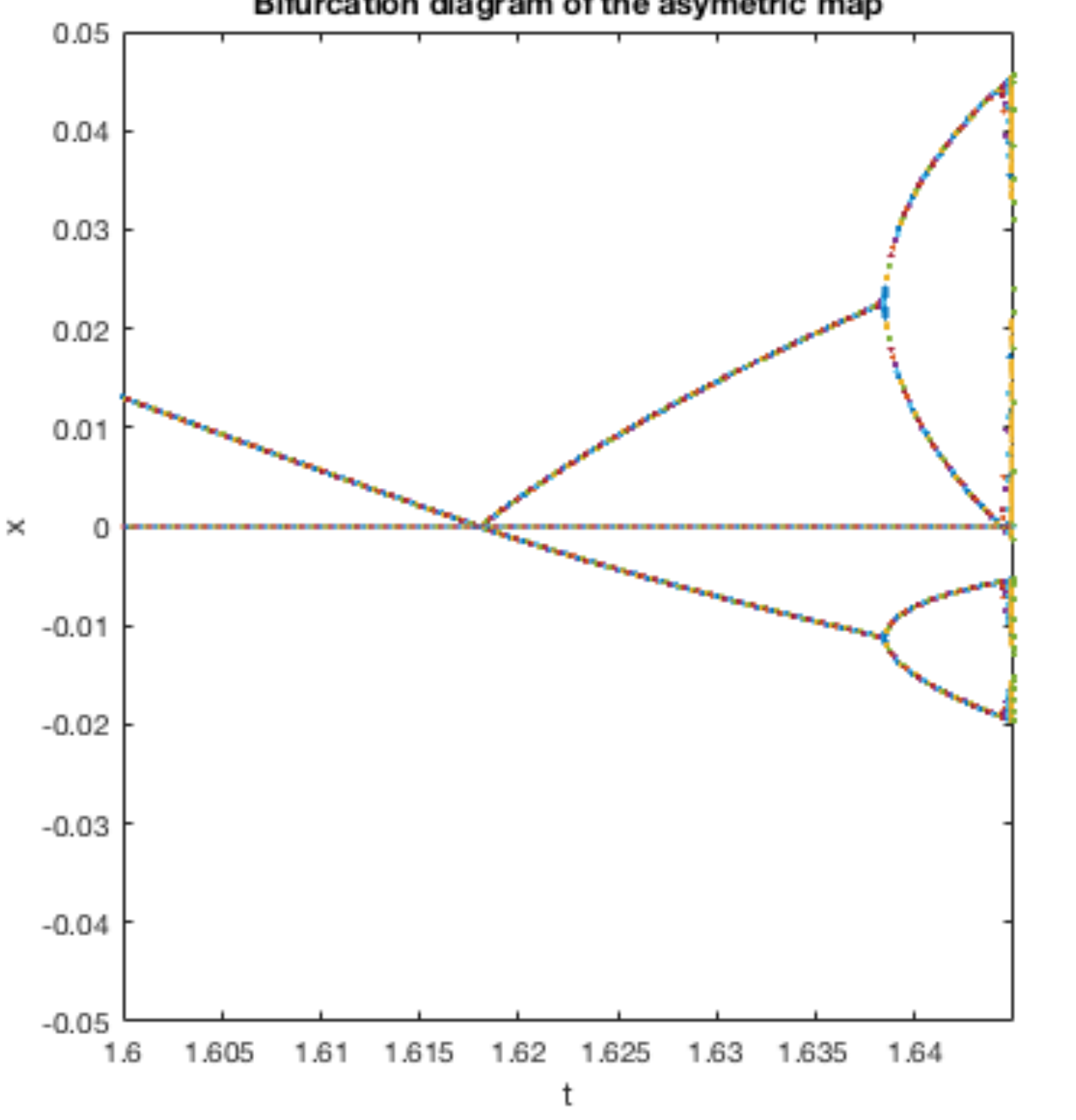}
\end{subfigure}
\caption{\label{fig1}\small The bifurcation diagram of the family of asymmetric maps $\{f_t\}_{t\in [1,2]}$, defined in (\ref{examplefamily}) together with two zoomed-in versions 
with the position of the critical point $x=0$ marked. Note that the  doubling bifurcation from period $2^n$ to period $2^{n+1}$ when $n$ is odd
is not the classical one; in the current asymmetric  case the period doubles precisely when $0$ is periodic (rather than 
when the multiplier is equal to $-1$), as is explained in Theorem~\ref{thm:existence}.
The parameter scalings also appears to be rather different than that for the quadratic family.}
\end{figure}

\begin{itemize}[leftmargin=*] % labelindent=0pt by default.
\item We study bifurcations leading to a
    Feigenbaum-Coullet-Tresser map and prove the existence of such a
    map in our class (Theorem~\ref{thm:existence2}). The argument here
    will be rather soft.  Although the period doubling diagram, see
  Figure~\ref{fig1}, looks qualitatively the same as for the quadratic
  family, there are important differences when $1=\alpha<\beta$: when $n$ is odd, the
  periodic orbit of period $2^n$ doubles its period when it contains
  the critical point rather than when its multiplier is $-1$.
\item An initial crucial step in the theory of unimodal (or, more
    generally, one-dimensional) maps is to establish the existence of
   distortion bounds. This usually relies on  {\lq}real bounds{\rq} or {\lq}Koebe space{\rq}, by which we
    will mean, in this setting,  that the first entry map $f^{2^n-1}$ from the critical value
    to $[a_n,b_n]$ has a diffeomorphic extension whose range contains a definite intervals around $[a_n,b_n]$.
    Having this property gives distortion bounds on the first entry map. Surprisingly, as we will show   
     such extensions do NOT exist for $f^{2^n-1}$ (Theorem~\ref{thm:nousualbound}, also
    \ref{thm:absencekoebe}).             As far as we know this is the first type of unimodal map
    for which such bounds are known not to exist.
        In
    spite of the absence of Koebe space, we will be able to control
    the distortion of certain branches of the iterates of $f$
    (Theorem~\ref{thm:bounds}). This is the main step in this paper, 
    and the proof is involved and interesting. For an idea of the proof see \S\ref{subsec:sketch}. 
    Here we rely heavily on the 
    fact that $f$ is almost linear  on one side of  the critical
    point.  This lead us to invent the notion of a  {\em semi-extension}, see  \S\ref{subsec:semiexten}.
    This means that we consider a maximal diffeomorphic extension of $f^n$ which is obtained by 
    taking  an appropriate      composition of the right branch of $f$ and a 
    (diffeomorphic)     extension of the left branch of $f$ beyond the critical point. 
      Using this tool, we analyse various scenarios concerning the position of certain points which 
      as such have no dynamical interpretation. Thus we obtain increasingly precise information, and thus we   eventually obtain extremely good real bounds for these  semi-extensions. 
 
\item Using distortion properties mentioned above, we are able to obtain very precise
  scaling laws, see Theorem~\ref{thm:scalings}. These scaling laws are
  rather different than for the usual {\lq}symmetric{\rq}
  Feigenbaum-Coullet-Tresser case where the scalings are geometric and
  universal (the rates only depend on the order of the critical point)
  and so      we have
  $$|b_{k+1}-a_{k+1}| \sim \kappa |b_k- a_k|$$
  for some $0<\kappa<1$ which does not depend on which unimodal map
  one takes (provided its critical point is quadratic).  In our
  setting, the scalings of their lengths are quite different for even
  and odd steps, namely
  \begin{equation*}
    \begin{array}{rll}
      |b_{2k+2}-a_{2k+2}|   &\sim 
      % \dfrac{\beta}{K_0} \left[ \dfrac {K_0^\beta} {\beta^{\beta+1}}  \right] ^{1/(\beta-1)} \lambda^{-2}   b_{k+1}^2  =
      & \beta^{\frac{-2}{\beta-1}} K_0^{\frac{1}{\beta-1}}  \lambda^{-2} 
        | b_{2k+1}-a_{2k+1}|^2  \\ 
                            & \\ 
      |b_{2k+1}-a_{2k+1}| & \sim  &\lambda |b_{2k}-a_{2k}|  
    \end{array}
  \end{equation*}
  where $\lambda$ is the root of $$\lambda^\beta+\lambda-1=0$$  and $K_0=K_+/K_-$. 
  Moreover, there exists $\Theta>0$ so that  
  % \begin{equation} \dfrac{1}{|b_{2k}-a_{2k}|} \sim  \beta^{\frac{-2}{\beta-1}} K_0^{\frac{1}{\beta-1}}  \exp(2^k \Theta).
  %   \label{eq:Thetaintro}\end{equation} 
  \begin{equation} |b_{2k}-a_{2k}| \sim  \beta^{\frac{2}{\beta-1}} K_0^{\frac{-1}{\beta-1}}  \exp(-2^k \Theta).
    \label{eq:Thetaintro}\end{equation} 

\item In the classical Feigenbaum-Coullet-Tresser $2^\infty$ case, 
  maps with quadratic critical points   
  are necessarily differentiably
  conjugate along the closure of the forward iterates of the critical
  point. This phenomenon is usually referred to as {\em
    universality}. Here this universality no longer holds: two maps
  $f,\tilde f$ are Lipschitz (and even differentiably conjugate) if
  and only if
  $$\beta=\tilde \beta , \Theta = \tilde \Theta.$$
  This means that this case is rather more similar to \cite{LM,MP}
  where there are also necessary and sufficient conditions for these
  maps to be differentiably conjugate at the turning point, see 
  Theorems~\ref{thm:thetavaries} and \ref{thm:nouniversality}.

  One of the consequences of this fact is that $f$ and its
  renormalizations are not Lipschitz conjugate even at the critical
  point $c$.

\item In the {\lq}symmetric{\rq} case the $n$-th renormalization of
  the function converges to some analytic function with unknown closed
  formula. Here we obtain a degenerate limit, but whose form is
  entirely explicit, see Theorems~\ref{thm:renormalizationlimit} and
    \ref{thm:renormalizationlimit2}).

\item The $2^\infty$ maps we consider do not have wandering
  intervals, see Theorem~\ref{thm:wandering}.
   Absence of wandering interval for our class of maps
    implies that the maps we consider are all topologically conjugate
    to the quadratic Feigenbaum-Coullet-Tresser map.  
  
\end{itemize}  

%\vspace{2mm}

\subsubsection*{History of the problem.}
Renormalisation and rigidity results were proved previously for circle
diffeomorphisms with Diophantine conditions on the rotation number
\cite{Her,Yoccoz2}. For circle maps with discontinuities of the
derivative (break type singularities) there are quite a few results,
see e.g. \cite{KK1, KKT, ANA, KKM, KK2, CS}.    For smooth homeomorphism of
the circle with a critical point, there are results by \cite{FM1,FM2,
  KY, Yam,Avila}.  For infinitely renormalizable unimodal interval
maps there is a rich history, starting with the conjectures of
Feigenbaum and Coullet-Tresser.  Rigorous proofs were finally provided
by \cite{sullivan, McM1, McM2, AL}, see also \cite{
  deFariadeMeloPinto, Smania2, Smania3, Smania4}. The weakly symmetric 
  case is considered in \cite{OM}.    Note that for
interval maps smooth rigidity is not possible, so the natural context
there is quasi-symmetric rigidity. This was proved in increasing
generality in \cite{GS2, Ly2, KSS1, CvS}, see also \cite{CdFvS, CvS,
  KSS1,LSw}.  For Lorenz maps there is another very interesting
phenomenon: in this case the renormalization operator can have several
(degenerate) fixed points even when the left and right critical
exponent at the discontinuity is the same.  This can happen even for
bounded combinatorics, and return maps can degenerate \cite{MW,W}.

For circle maps with
plateaus see for example \cite{MP, Swiatek,MMMS,  Palmisano, Palmisano2,TV1,TV2,Gra}.
%We should also mention the literature for circle maps with a plateau $[a,b]$.
Here it is also natural to explore the role of the orders of the critical points at the boundary points 
of a plateau $[a,b]$. Quite often it is assumed that these orders are the same, 
see \cite{Swiatek,MMMS,  Palmisano, Palmisano2} but not in the entire literature, 
see for example \cite{TV1,TV2,Gra}.
%because there also the issue of critical points of different orders comes up naturally, 
%see for example \cite{Swiatek,TV1,TV2, Gra, Palmisano, Palmisano2}.
For such maps, super-exponential scaling was obtained in \cite{Gra}  under the assumption that 
  $f(x)-f(a)\sim -|x-a|$ to the left of $a$ and $f(x)-f(b)\sim |x-b|^\beta$ with $\beta>1$ to the right of 
  $b$.  Here the $q_n$-th iterates of the plateau are considered, and these iterates
  converge super-exponentially in terms of $n$. 
    In \cite{Palmisano, Palmisano2} it is assumed that $f(x)-f(a)\sim  -|x-a|^\alpha$ to the left of $a$ and 
$f(x)-f(b)\sim |x-b|^\alpha$ to the right of $b$ (so the orders on both sides are the same).  The main result
in \cite{Palmisano}  is that one  has bounded geometry (so the approach rate is at most exponential) in terms of $n$ if $\alpha>2$, and a  super-exponential approach is if $\alpha\le 2$. 
In \cite{Palmisano2} it  is shown that any two such maps  with bounded geometry and with the same  rotation number, 
are quasi-symmetrically conjugate. 

The question whether two maps which are combinatorially the same, are
in fact topologically conjugate hinges on absence of wandering
intervals.  The first results in this direction were obtained for
circle diffeomorphisms in the 1920's by Denjoy \cite{De}, % \cite{Denjoy}, 
for critical circle
maps in 
%\Do{Yoccoz} 
\cite{Yoccoz1} and for circle maps with plateaus in 
\cite{MMMS}.  For interval maps  there are results, in increasing 
generality, \cite{Mis,Gu,dMvS,BL,Ly1,MMS,vSV}. 
%%\Do{ by \Do{Misiurewicz}
% \cite{Mis}, \Do{Guckenheimer} \cite{Gu}, \Do{de
%Melo-van Strien} \cite{dMvS}, 
%%\Do{Block-Lyubich} 
%\cite{BL}, \Do{Lyubich}
%\cite{Ly1}, \Do{de Melo, Martens, van Strien,} \cite{MMS} and 
%%\Do{Vargas-van Strien} 
%\cite{vSV}.  }
On the other hand, interval exchange
transformations can have wandering intervals, see e.g. \cite{MMY}.
Furthermore, it is  not known whether a 
circle homeomorphism with a strongly asymmetric critical point (which means that  
$f(x)-f(c)\sim -|x-c|^\alpha$ to the left of $c$ and $f(x)-f(c)\sim |x-c|^\beta$
to the right of $c$ where for example $1\le \alpha<\beta$) can have  wandering intervals.
%  For example the case of smooth
%circle homeomorphisms with at least two singularities with one of the
%form $\pm |x-c_1|^\alpha$, $\alpha<1$ and the other of the form
%$\pm |x-c_2|^\beta$ with $\beta>1$ is still wide open.
It was for this reason that the authors were curious to find out whether one
can have wandering intervals in the strongly asymmetric case.

\subsubsection*{Open questions.}

Before stating our results rigorously, let us discuss questions and possible directions for further research. 

\paragraph{Super-exponential scaling when  $1<\alpha<\beta$.} 
In this paper we always assumed that the left critical order $\alpha$
of our map is equal to $1$.  We believe that the super-exponential
scaling of the points $a_n$ and $b_n$ that we have shown here, also
holds when $1<\alpha<\beta$. 
Indeed,  the strong asymmetry (and the fact that
the map is unimodal) forces there to be scalings of entirely different orders of magnitude:
the scaling on the left side of the critical point
 is a power of the scaling on the right side of the critical point.
Assuming suitable {\lq}real bounds{\rq}  (and that the map 
has $2^\infty$  dynamics) 
this implies  super-exponential scalings.
%Such super-exponential scaling would most likely hold if one could
%  obtain some sort of {\lq}real bounds{\rq}. 
  However, it is very  unlikely that such real bounds hold when $\alpha<\beta$, 
  and this is one reason why our proof
  is delicate. But if what we suspect is true, then the case $\alpha=\beta$
  is completely different from when $\alpha<\beta$. 
The same phenomena should also hold for many  other combinatorics
provided, amongst other things, the critical point is accumulated from both sides under certain first return maps.  

\paragraph{Absence of wild attractors when  $1<\alpha<\beta$.} 
It is well-known that in the {\lq}symmetric{\rq} case, the so-called
Fibonacci map has a wild attractor provided the order of the critical
point is large.  Inspired by our belief that one has super-exponential
scaling, we believe that such attractors do not exist when
$1<\alpha<\beta$, even if these numbers are arbitrarily large.

\paragraph{Absence of wandering intervals.}  In this paper we only
proved absence of wandering intervals for the $2^\infty$ combinatorics
and when $1=\alpha\le \beta$. We believe one has absence of wandering
intervals without these assumptions.  In fact, we tried and failed to
prove this result in the case that $1<\alpha<\beta$.

\paragraph{Monotonicity of bifurcations.} Notice numerical simulations
suggest that the bifurcations from the family $f_t$ from equation
(\ref{examplefamily}) are monotone: no periodic orbit seems to
disappear when $t$ increases.  When instead we consider the family
\begin{equation}
f_t(x)=\begin{cases} t-1 -t|x|^\alpha &\mbox{ when }x<0, \\
t-1- t x^\beta &\mbox{ when }x\ge 0. \end{cases}
\label{examplefamilybis}
\end{equation}
with $\alpha,\beta>1$ large, then there are partial results towards
monotonicity in \cite{LSvS} see also \cite{LSvS2}. Monotonicity for this family is only 
known in full generality when $\alpha=\beta$ is an even integer. For references on the history
of results on monotonicity, see \cite{LSvS2}.
%{This was proved using complex methods by Sullivan,
%Thurston, Tsujii, Milnor, Douady, for references see \cite{LSvS}.}

\paragraph{More precise rigidity results.} Consider continuous degree
one circle maps, which are smooth local diffeomorphisms outside a
single plateau and with $x^\beta$ behaviour at the boundary points of
this plateau. In earlier papers \cite{MMMS} it was shown that such
maps have no wandering intervals, and in \cite{Palmisano} it was shown
that one has super-exponential decay of scales when $\beta\in (1,2)$
when the rotation number is golden mean. In \cite{MP}, it is 
%\Ds{Martens and Palmisano} 
shown that there exist invariants for Lipschitz,
differentiable and $C^{1+\epsilon}$ conjugacy. For related results see
\cite{CG}.  A similar obstruction to differentiable conjugacy also
appears in \cite{LM}.

\paragraph{Parameter scaling.} Consider the family $f_t$ defined in (\ref{examplefamily})
and  let  $t_n$ be the parameter where the turning point $0$ has period $2^n$ for $f_{t_n}$ and let 
$t_*$ be so that $f_{t_*}$ has $2^\infty$ dynamics. 
Computer experiments suggest that the parameters $t_n$ scale also super-exponentially. 
We are hopeful that we will be able to elaborate the methods in this paper to prove the following 

\begin{conj}[Non-universality of parameter bifurcations]
\begin{equation} |t_{n+2}-t_*| \sim  \kappa |t_{n}-t_*|^2 
\label{eq:Theta-parameter} \end{equation} 
where $\kappa$ depends non-trivially on the two parameters $\beta,\Theta$  associate to 
the family $f_t$ and so is not a universal parameter,
where $\Theta$ is defined through equation (\ref{eq:Thetaintro}).
\end{conj} 

So we conjecture that, in our setting, the parameter scaling is super-exponential and non-universal. 
This is in contrast to the universality results for generic smooth families of unimodal maps  with a quadratic critical point
 (where the genericity assumption is that the family is assumed to be transversal to the stable manifold of the renormalization operator)
where one has the parameter scaling 
$$|t_{n+2}-t_*| \sim  \lambda  |t_{n}-t_*|$$
where $\lambda$ is universal and so does not depend on the family. 

\paragraph{Renormalisation theory in the smooth setting.} 
The renormalization theory we develop here is done by obtaining large bounds.
This is quite different from the renormalization theory obtained for real analytic unimodal maps, 
\cite{sullivan,McM1, McM2,Ly3,AL,deFariadeMeloPinto}, see also \cite{Smania3,Martens,GY}. 
It would be interesting to tie these approaches together.

\subsection*{Acknowledgement} 
The authors would like to thank Bj\"orn Winckler for carefully reading this manuscript, 
in particular Section~\ref{Sec:bigbounds} and Trevor Clark and Polina Vytnova for some helpful discussions. 
SvS was supported by ERC AdG RGDD No 339523. We are grateful for the referees for their helpful  comments.

 \section{The setting of this paper} \label{sec:settings} 

Consider the class $\mathcal{A}_{\alpha,\beta}$ of  continuous unimodal maps $f : [a_0,b_0] \to  [a_0,b_0]$ 
where $a_0<0<b_0$ and 
with the following properties:
\begin{enumerate}
\item[1]  $f(a_0)=f(b_0)=a_0$ 
and outside the turning point $c:=0$ the map $f$ is $C^3$ and has Schwarzian derivative
$Sf\le 0$. The authors believe that the results in this paper also
hold without the $Sf\le 0$ assumption. 
\item[2] $c=0$ is the unique extremal value of $f$ and  $f'(x)>0$ for $x<0$ and $f'(x)<0$ for $x>0$. 
\item[3] Near the critical point $c=0$ the map $f$ behaves as $f (x) \sim - K_- |x|^\alpha + f (0)$ for $x<0$
and $|x|$ small  and  $f (x)\sim - K_+ x^\beta + f (0)$ for small positive values of
$x$. The constants should satisfy $K_- >0, K_+ > 0$ and $\beta>\alpha\ge 1$.
\end{enumerate}
Almost everywhere in the paper we shall assume that $\alpha=1$, in
this case we will denote $\mathcal{A}_{1,\beta}$  just by $\mathcal{A}$.
We say that $f\in \mathcal A_{\alpha,\beta}(2^\infty)$ if in addition 
\begin{enumerate}
\item[4] The map $f$ has $2^\infty$  combinatorics, i.e. $f$ is
an  infinitely renormalizable  {\em Feigenbaum-Coullet-Tresser} period
  doubling  map. By definition this means that there exists a shrinking sequence of intervals
  $[a_k,b_k]\ni c$ so that the restriction of $f^{2^k}$ to $[a_k,b_k]$ is again unimodal,  mapping
$\{a_k,b_k\}$ into itself and so that the intervals $f^i[a_k,b_k]$, $i=0,\dots,2^k-1$ have pairwise disjoint interiors.
\end{enumerate}

 The sequence  $[a_k , b_k ], k = 0, 1, . . .$, is constructed in the
following way. Let $b_1$ be a fixed point of $f$ with negative
multiplier and $a_1$ be its preimage. Then
$c_2 := f^2 (0) \in [a_1 , b_1 ]$.  Notice that
$a_0 < a_1 < 0 < b_1 < b_0$.  The intervals $[a_0,b_0]$ and 
$[a_1,b_1]$ are drawn in Figure~\ref{fig2}.
Since the map $f$ is assumed to be of
Feigenbaum-Coullet-Tresser $2^\infty$ type, $f^2 |[a_1 ,b_1 ]$ is
again unimodal; it decreases on $[a_1 , 0]$ and increases on
$[0, b_1 ]$. The branch $f^2 |[a_1 ,0]$ has a fixed point which we
will denote by $a_2$ and $b_2$ will denote its preimage by
$f^2 |[0,b_1]$. Using again that $f$ is a $2^\infty$ map,
$f^4 |[a_2 ,b_2 ]$ is unimodal, and we can continue this process
indefinitely and obtain a sequence of points $a_k<0<b_k$ and unimodal
maps $f^{2^k}\colon [a_k,b_k]\to [a_k,b_k]$.

\medskip

As will be shown in  Theorem~\ref{thm:existence} in Subsection~\ref{subsec:existence}, 
there exist many maps within the class  $\mathcal A(2^\infty)$.  For example,  there exists $t_*\in (1,2)$ so that 
$f_{t_*}\in \mathcal A(2^\infty)$ where   $f_t\colon [-1,1]\to [-1,1]$, $t\in [1,2]$ is defined by  
\begin{equation}
f_t(x)=\begin{cases} t(1+x)\,\, -1 &\mbox{ when }x<0, \\
t(1-x^\beta)-1 &\mbox{ when }x\ge 0. \end{cases}
\label{examplefamily}
\end{equation}
%where in this case we take $[a_0,b_0]=[-1,1]$. 
As we will see in Subsection~\ref{subsec:existence}
this family $f_t$ undergoes  unusual period doubling bifurcations, see Figure~\ref{fig1}.

%Note that the extremal value of $f_t$ is $f_t(0)=t-1$),
%as we will show in %

\begin{figure}
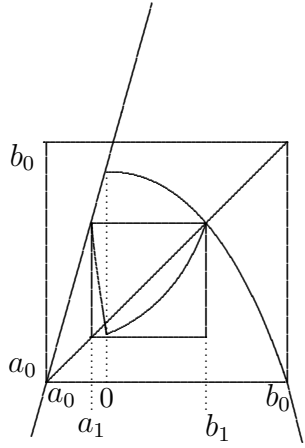
 \hfil
\beginpicture
\dimen0=0.4cm
\setcoordinatesystem units <\dimen0,\dimen0>  point at -15 0
\setplotarea x from -4 to 4, y from -6 to 2
\setlinear
\plot -4 -4 4 -4  4 4 -4 4 -4 -4 /
\plot -4 -4 4 4 / 
%\setdots <0.8mm> 
%\plot -4.5 -4 4.5 -4 / 
%\plot 4.5 -4 4.5 -6 / 
%\plot -4.5 -4 -4.5 -5.5 / 
%\plot 1.3 1.3 1.3 -4 / 
\setsolid
\put {$a_0$} at -3.5 -4.5 
\put {$b_0$} at 3.7 -4.5 
\put {$a_0$} at -4.8 -3.5 
\put {$b_0$} at -4.8 3.5 
\plot -2.5 1.3 1.3 1.3  1.3 -2.5 -2.5 -2.5 -2.5 1.3 /  
%\put {$a_1$} at -2.8 -2.8 
%\put {$b_1$} at 1.7 -2.8 
\setquadratic 
\plot -2 -2.4    0 -1      1.3 1.3  /
\plot  -2.5 1.3 -2.36  0 -2 -2.4  / 
%\setlinear 
%\plot -2.5 1.3 -2 -2.4  / 

%\put {$b_k'$} at 5.3 -4.5 
%\put {$a_k'$} at -5.3 -4.5 
%\put {$b_{k+1}$} at 1.5 -4.5 
%\put {\tiny $\bullet$} at 4.5 -4  
%\put {\tiny $\bullet$} at 4 -4
%\put {\tiny $\bullet$} at -4 -4
%\put {\tiny $\bullet$} at -4.5 -4
%\put {\small  $0$} at -2.2 -4.5 
%\put {\small $d_k$} at -1.3 -4.5 
%\put {\small $e_k$} at -0 -4.5 
%\put {\small $b_k$} at -5 4  
%\put {\small $B_k$} at -5 8.6 
%\put {\small $A_k$} at -5 -6  
\put {\small $0$} at -2 -4.5 
\setlinear 
\plot -4.5 -5.8 -4 -4 -1.5 5    -0.5 8.6 /  % 2.5 9 i.e. 1 3.6
\setquadratic 
%\plot -2 3 0 1.5 4 -4 /
\plot -2 3 1.3 1.3 4 -4 /
\setlinear 
\plot 4 -4 4.5 -6 / 
\setdots  <0.8mm>  \setlinear
%\plot -0.5 -4 -0.5 8.6 /  
%\plot -1.7 -4 -1.7 4 / 
%\plot -4 8.6 -0.2 8.6 /$Q_t(x)=tx(1-x)$
\plot -2.5 -2.5 -2.5 -5 /
\plot 1.3 -2.5 1.3 -5 /  
\put {$a_1$} at -2.5 -5.5 
\put {$b_1$} at 1.7 -5.5 
\plot -2 -4 -2 3 /
%\plot -4 -2 -2 -2  /
\endpicture
\caption{\label{fig2}  $f$ together with its renormalization and its semi-extension.}
\end{figure}

%\textcolor{red}{In this paper, 
%we will develop a renormalization theory in the setting that $1=\alpha<\beta$} 
%and show that
%the scaling laws for such maps are universal, but entirely different
%from those of smooth maps with non-flat critical points.
%\textcolor{red}{As mentioned, we hope that the results in this paper  will form
%the start of a  general theory for maps in  $\mathcal A_{\alpha,\beta}$ with $1\le \alpha<\beta$.} 
%%As a consequence we derive renormalization results.

\paragraph{Some notation.}

We say that the interval $T$ is a $\tau$-scaled neighbourhood  of $J
\subset T$
if both components of $T\setminus J$ have at least size $\tau \cdot |J|$. 
We shall also  use the notations 
$$\begin{array}{rl}
u_k\sim v_k   &\iff \frac{u_k}{v_k}\to 1 \mbox{ as }k\to \infty \\
& \\
$$u_k\approx v_k  &\iff 0<\liminf \frac{u_k}{v_k}\le \limsup  \frac{u_k}{v_k} < \infty \mbox{ as }k\to \infty.\end{array}$$
Given two intervals $U,V\subset \R$ we define $[U,V]$ to be the
smallest interval containing both.

\section{Statement of results}

%Throughout the paper we will assume that $f\in   \mathcal A(2^\infty) (=\mathcal A_{1,\beta}(2^\infty))$, except
%in the next theorem where we consider families of maps in $\mathcal A_{1,\beta}$
%in order to show that the set $\mathcal A_{1,\beta}(2^\infty)$ is non-empty. 

\paragraph{Existence of infinitely renormalizable maps.}
Our first task is to show that the class $\mathcal A(2^\infty)$ is non-empty. 
In other words, we need to establish strongly asymmetric maps
with Feigenbaum-Coullet-Tresser dynamics. For maps which are differentiable 
at the extremal point, this follows from an analysis how kneading sequences
depend on the parameter, see \cite{MT} or from some fixed point argument \cite{dMvS}. 
When $1=\alpha<\beta$ these proofs break down. In fact, if $\alpha=\beta=1$ holds 
(this corresponds to a family of tent maps) then there are no Feigenbaum-Coullet-Tresser maps.

Nevertheless we have the following theorem, showing that 
every family such as the one defined in (\ref{examplefamily}) contains
a map in $\mathcal A(2^\infty)$. 

\begin{thm}\label{thm:existence2}
For the family defined in (\ref{examplefamily}) there exists a parameter $t_*$ so that $f_{t_*}\in \mathcal A(2^\infty)$. 
\end{thm}

In fact, the proof of this theorem will show that any family similar
to (\ref{examplefamily}) (not necessarily with $\alpha=1$) is {\em
  full} in the sense that for each parameter $t$ there exists $t_*$ so
that $f_{t_*}$ has the same kneading invariant as $Q_t(x)=tx(1-x)$.

\paragraph{The issue of real bounds}

Since the power laws of $f$ at both sides of $0$ are different, most
proofs from the theory of one-dimensional dynamics do not apply. The
stumbling block appears already when trying to recover real bounds. 
   For example, for {\lq}symmetric{\rq} unimodal maps for which the 
power laws on both sides of $0$ are the same, one has the property that 
the first entry map from the critical value $f(0)$ to the interval $[a_n,b_n]$ has bounded distortion, see 
\cite{dMvS}. This kind of bound forms  the cornerstone for everything else in the theory 
of unimodal maps, and so this is the first
issue to overcome.  In the weakly symmetric unimodal case the
standard proof of such a real bound relies on the simple but powerful {\em smallest interval
argument}, see Lemma~\ref{lem1}. In the weakly symmetric case this argument
gives space on both sides of some interval, and in the strongly asymmetric case
only on one side, which prevents Koebe like distortion results. It
turns out that this is not just a technical issue as the most basic
real bounds do not hold. Indeed, the first entry map from the critical
value into a periodic renormalization interval around the critical
point does NOT have a diffeomorphic extension with Koebe space, see
for example Theorem~\ref{thm:nousualbound} below. 
 Moreover, entirely new
scaling phenomena appear as a result of this asymmetry.

The purpose of this paper is to make a step towards a theory 
for strongly asymmetric maps
%overcome this gap in the literature by
obtaining results on real bounds, scaling laws and absence of
wandering intervals in this setting.  Indeed we believe that the results described
in this paper go through for all maps in $\mathcal A_{\alpha,\beta}$ with $1\le \alpha<\beta$, 
although we were only able to do this under the assumption that $\alpha=1$.  
For the case that $1=\alpha<\beta$ we were able to exploit the almost linearity of 
the left branch near the turning point $c=0$, but when $1\le \alpha<\beta$ one should be able to 
exploit the huge asymmetry to obtain good control on the first entry maps. This is certainly
what numerical simulations seem to suggest.

\paragraph{No diffeomorphic extensions}

The main source of difficulties lies in the following theorem, which
shows the difference with the {\lq}symmetric{\rq} case:

\begin{thm}
  \label{thm:nousualbound} For every $\tau>0$ there exists $k_0\ge 0$ so
  that if $T\ni f(0)$ is the maximal interval on which $f^{2^k -1}|T$
  is diffeomorphic, then $f^{2^k-1}(T)$ does {\bf not} contain a
  $\tau$-scaled neighbourhood of $[a_k,b_k]$ for any $k\ge k_0$.
\end{thm}

\paragraph{Semi-extensions.}
To overcome this issue, we will introduce the notion of
semi-extension. Since $\alpha=1$, the derivative of $f$ near the critical point of the left branch of $f$ is non-zero
and we can extend this branch smoothly ($C^3$) and monotonically to $f_1 : [a_0,\epsilon_0] \to \R$
in such a way that $\epsilon_0>0$,  $f_1 |[a_0 ,0] = f$, the derivative of $f_1$ is strictly positive,
and the Schwarzian derivative of $f_1$ is $\le 0$.  For consistency, the right branch
of $f$  will be denoted by $f_2$, i.e. $f_2 = f |[0,b_0]$.

\begin{definition}[Semi-extensions]
 Let $J$ be an interval and $f^n |J$ be monotone. Then $F : T\to \R$ is
called {\em monotonic semi-extension} of $f^n |J$  if
\begin{itemize}
\item  $J\subset T$ and $F|J =f^n|J$;
\item  $F=f_{i_1} \circ \dots \circ f_{i_n},$ where $i_k\in \{1,2\}$ for $k=1,...,n$. 
\end{itemize}

We will call such an extension {\em maximal} if $T$ is the maximal
interval satisfying the above properties.
\end{definition}

\paragraph{Big bounds for  the first entry maps to $[a_k,b_k]$ when $k$ is even.}
It turns out that these semi-extensions are surprisingly useful
since the branch $f_1$ is essentially linear near $0$. 
%(under the assumption that $\alpha=1$). 
%, where we crucially use that  the branch $f_1$ is essentially linear near $0$. 
%Indeed, when $\alpha=1$ then iterates of $f$ have very good semi-extensions: 
Indeed, the semi-extension of the first entry map from an interval $J\ni f(0)$ 
to $[a_k,b_k]$ becomes almost linear for $k\to \infty$ and even.  
On the other hand, it turns out that as $k$ odd and $k\to \infty$ 
this first entry map does {\bf not} converge to a linear map.

\begin{thm}\label{thm:bounds}
Let $f^{2^k-1} \colon J\to [a_k,b_k]$ be the first entry map 
of $J\ni f(0)$ into $[a_k,b_k]$ and let  $F_k \colon T_k \to \R$
be the maximal monotonic semi-extension of $f^{2^k-1} \colon J\to [a_k,b_k]$.
Take $\tau_k>0$ be maximal so that  $F_k(T_k)$ is $\tau_k$-scaled neighbourhood of $[a_k,b_k]$.
Then 
\begin{itemize}
\item $\lim \tau_{2k-1}= \lambda$
where $\lambda\in (0,1)$ is the root of the equation $\lambda^\beta +\lambda=1$.
\item $\tau_{2k}\approx b_{2k}^{-1/2}$ grows super-exponentially with $k$. In fact, $\log \tau_{2k}$ grows exponentially,
see also equation  (\ref{eq:defTheta}) below. 
\end{itemize}
\end{thm}

\begin{remark}\label{remark1}  As we will show in 
Theorem~\ref{thm:absencekoebe}  and  Section~\ref{sec:nokoebe}, 
this theorem does not hold when we drop the assumption that  $J\ni f(0)$. 
This will complicate for example the proof of Theorem~\ref{thm:wandering} (on absence of wandering intervals). 
\end{remark}

\medskip

\paragraph{Scaling laws.} 
From this theorem we will obtain that the geometry of the $\omega$-limit set
is quite different from the one found in smooth unimodal maps with $2^\infty$
combinatorics. In the next theorem we describe this scaling.
%It is straightforward to show that $a_k \to  0$ and $b_k\to 0$  as $k \to \infty$
% (otherwise the critical point would be non recurrent). 
By definition $f(a_k) = f(b_k)$ and therefore 
\begin{equation} a_k \sim - K_0b_k^\beta,\mbox{ where }K_0 = K_+/K_-.
\label{akbk} 
\end{equation}
Thus the scaling properties of the renormalization intervals
can be described just by the scaling properties of $b_k$.

\begin{thm} \label{thm:scalings}
The following scaling properties hold for $b_k$:
\begin{itemize}
\item For large even values of $k$ one has
\begin{equation}
\begin{array}{rll}
%b_{k+1} & \sim & \lambda b_k, \\ 
b_{k+1} & \sim  &\lambda b_k  \\
c_{2^k} &\sim & b_k,
\end{array}
\label{eq3}
\end{equation}
where as before $\lambda\in (0,1)$ is the root of the equation $\lambda^\beta +\lambda=1$.
\item For large odd values of $k$ one has
\begin{equation}
\begin{array}{rll}
b_{k+1}  &\sim 
%\dfrac{\beta}{K_0} \left[ \dfrac {K_0^\beta} {\beta^{\beta+1}}  \right] ^{1/(\beta-1)} \lambda^{-2}   b_{k+1}^2  =
& \beta^{\frac{-2}{\beta-1}} K_0^{\frac{1}{\beta-1}}  \lambda^{-2} 
 b_{k}^2 \\ 
 c_{2^{k}} &\sim  &
 -\beta^{-\frac{\beta+1}{\beta-1}} K_0^{\frac{\beta}{\beta-1}}  \lambda^{-\beta-1} 
 b_{k}^{\beta+1} 
 \end{array}
 \label{eq3"}\end{equation}
\item The length of the renormalization intervals decays super-exponentially
fast: there exists $\Theta >0$ so that 
\begin{equation} \log \left(\frac{1}{b_{2k}}\right) \sim \log \left(\frac{1}{|b_{2k}-a_{2k} |}\right) \sim \Theta \cdot 2^k . 
\label{eq:Theta} \end{equation} 
More precisely, 
\begin{equation}  1/b_{2k} \sim  \beta^{\frac{-2}{\beta-1}} K_0^{\frac{1}{\beta-1}}  \exp(2^k \Theta). \label{eq:defTheta}
  \end{equation} 
 \end{itemize}

%
%taking $ D :=-\log ( \beta^{\frac{-2}{\beta-1}} K_0^{\frac{1}{\beta-1}}  \lambda^{-2})$ we get 
%\begin{equation} 
%\lambda /b_{2k+1} \sim  1/b_{2k} =  \exp(2^k \Theta + D + o(1) ). \label{bklog} \end{equation} 
%%$$|b_k-a_k|<C\mu^{k^{\sqrt{2}}}, k\ge 0.$$ 
%\end{itemize}
In (\ref{eq3})  the convergence is super-exponentially:   $b_{k+1}/b_k$ 
converges to $\lambda$  super-exponentially fast. 
\end{thm}

%\begin{remark} \textcolor{red}{Check remark. Is it needed?} 
% Let $k$ be a large odd integer and let $\tau_k$ be the maximal number so that 
%$[a_k,b_k]$ is a $\tau_k$-scaled neighbourhood of $[a_{k+1},b_{k+1}]$. 
%Then the previous theorem  implies that when $\alpha=1$ and $\beta<2$, then $|b_{k+1}-a_{k+1}|\sim b_{k+1}\approx b_k^2 < b_k^\beta \approx  |a_{k}| \sim |a_k-a_{k+1}|$ and therefore   $\tau_k\to 0$.
%On the other hand, when $\beta>2$ then $\tau_k \to \infty$. Note that this latter fact is not very useful, because
%the first entry map to $[a_{k+1},b_{k+1}]$ does not extend diffeomorphically to $[a_k,b_k]$. Indeed, the 
%semi-extension of the first entry to $[a_{k+1},b_{k+1}]$ extends to the left as far as $c_{2^{k}}$ 
%only if this first entry map includes visits to the interval $[b_{k+1},b_k]$, 
%as becomes clear from the right panel of Figure~\ref{fig:PD}  (where we need to replace $k+1$ by $k$). 
%\end{remark}

% the previous theorem  implies that
% the length the renormalization interval $[a_{k+1}, b_{k+1}]$ is much smaller than the length of $[a_k, a_{k+1}]$.
% On the other hand, when
%$\beta > 2$, then $|b_{k+1}-a_{k+1}| \gg |a_{k+1} - a_k|$ (because $b_k^2 \gg b_k^\beta$ ). 
%This has some implications for the way one can construct the space around the renormalization interval given by Theorem 1: if $\alpha=1$ and $\beta < 2$,  one  gets definite space to the left of the renormalization interval by extending the range of the branches to the renormalization interval of the previous level. In the case $\beta> 2$, this is not enough and one needs a different argument.

\paragraph{The parameter $\Theta$ can be arbitrarily large.}

The parameter $\Theta$ is determined by the asymptotic behaviour of $1/b_{2k}$. In the next corollary
we show that $\Theta$ indeed varies within the space $\mathcal{A}(2^\infty)$: 

\begin{cor}\label{thm:thetavaries}
For each $\Theta_0>0$ there exists a map $f\in \mathcal{A}_{1,\beta}(2^\infty)$ so that $\Theta(f)>\Theta_0$. 
\end{cor}
\begin{proof} From formula  (\ref{eq:defTheta}) it follows immediately that
 $\Theta(R^2(f))=2\cdot \Theta(f)$.
\end{proof} 

\paragraph{Renormalisation limits.} 

The above scaling laws make it possible to compute the renormalization map $R^k$ for $k$ even with  quite a lot of accuracy:

%\begin{remark} 
%Let $R^kf$ denote the $k$-th renormalization of f. In other words, let 
%$l_k \colon  [0,1] \to [a_k,b_k]$ be a linear map such that $l(0) = a_k$ and $l(1) = b_k$. Then define 
%$$R^kf := l_k^{-1} \circ f^{2^k} \circ  l_k.$$ 
%Let $\hat c_k$ denote the critical point of $R^kf$. From (\ref{akbk})  it is clear 
%that $\hat c_k \to 0$ as $k\to \infty$. Therefore, the left branch of $R^kf$ gets more and more degenerate and disappears in the limit.
%\end{remark}

\begin{thm} \label{thm:renormalizationlimit} 
For $k$ even we have 

 \begin{equation}
f^{2^k} (x) = 
\begin{cases}
c_{2^k} - s_k |x| +O(b_k^{\frac{3}{2}}) & \mbox{ when }x\in [a_k,0]\\ 
c_{2^k} - t_k x^\beta +O(b_k^{\frac{3}{2}}) & \mbox{ when }x\in [0,b_k] 
\end{cases}
\label{eq:f2k-thm}
\end{equation}
where \begin{equation}s_k\sim  \dfrac{b_k^{1-\beta}}{K_0}\mbox{ and }
t_k\sim b_k^{1-\beta}.\label{eq:sk-thm} \end{equation}
\end{thm}

%The right branch of the renormalizations of $f$ converge super exponentially fast in the $C^N$ norm to
%\begin{equation*}
%\begin{array}{rcl} 
%\lim_{k\to \infty} (R^{2k}f)|[\hat c_k,1]  &=& 1-x^\beta,  \\
%\lim_{k\to \infty}  (R^{2k+1}f)|[\hat c_k,1] &=& x^\beta.
%\end{array}
%\end{equation*} 
%Let $m_k : [-1, 0] \to  [0, \hat c_k ]$ be linear orientation preserving maps mapping the boundary to the boundary. Then in the $C^N$ norm
%\begin{equation*}
%\begin{array}{rcl} 
%\lim_{k\to \infty} (R^{2k}f)\circ m_{2k} &=& x, \\
%\lim_{k\to \infty}  (R^{2k+1}f) \circ  m_{2k+1} &=& -\lambda^{\beta^2-1 }(x + \lambda^{-\beta})^\beta + \lambda^{-1}. 
%\end{array}
%\end{equation*} 
%Here the convergence is super exponentially fast as well and $\lambda\in (0,1)$ is the root of $\lambda^\beta +\lambda=1$ as before.
%\end{thm}

As usual we can state the renormalization results by rescaling the intervals
to a fixed interval. So let $R^k f$ denote the $k$-th renormalization of $f$. In other words, let 
$l_k : [0, 1] \to [a_k , b_k ]$ be the linear map such that $l(0) = a_k$ and $l(1) = b_k$
and define  $R^kf:=l_k^{-1} \circ f ^{2^k}\circ l_k$. 
Let  $\hat c_k$ denote the the critical point of $R^k f$. 
From (\ref{akbk})  it is clear
that $\hat c_k \to 0$ as $k \to \infty$. Therefore, the left branch of $R^k f$ gets more and more
degenerate and disappears in the limit.

\begin{thm}\label{thm:renormalizationlimit2} 
 The right branch of the renormalizations of $f$ converge super
exponentially fast in the $C^1$ norm to
$$\lim_{k\to \infty}  (R^{2k} f )|[\hat c_k ,1] = 1-x^\beta$$
$$\lim_{k\to \infty}  (R^{2k+1}f )|[\hat c_k ,1] = x^\beta.$$
Let $m_k : [-1, 0] \to [0, \hat c_k ]$ be the linear orientation preserving maps mapping the
boundary to the boundary. Then in the $C^1$ norm
$$\lim_{k\to \infty}  (R^{2k} f ) \circ  m_{2k} = x+1$$
$$\lim_{k\to \infty}  (R^{2k+1} f ) \circ  m_{2k+1} = -\lambda^{\beta^2-1} (x+\lambda^{-\beta})^\beta + \lambda^{-1}.$$
 Here the convergence is super exponentially fast as well and $\lambda\in (0, 1) $ is the
root of $\lambda^\beta + \lambda = 1$ as before.
\end{thm}

It is easy to see that  $\lambda^\beta + \lambda = 1$ implies that
$ -\lambda^{\beta^2-1} (x+\lambda^{-\beta})^\beta + \lambda^{-1}$ is equal to $1$ when $x=-1$
and equal to $0$ when $x=0$.
Note that the asymptotic expression for the  left branch of $R^{2k+1}f$ is an explicit but non-trivial expression. 

\begin{remark}
One can prove also convergence in the $C^N$ norm in the above theorem
if $f$ is a smooth function outside of zero. 
If the map $f$ is only assumed to have finite smoothness this can be done as in \cite{KS}
or following the approach in \cite{CdFvS}.
If $f$ is real analytic (on each side of $0$) then this can be done by complex tools:
then $f^{2^k}=E_k\circ f$ where $E_k$ extends holomorphically to a diffeomorphism
whose range is $B(0,\tau_k|b_k|)$. Using the Koebe Lemma (in the complex case)
we then obtain that, for $k$ even, $DE_k=DE_k(c_1) + o(k)$  and 
$D^iE_k=o_i(k)$ for each $i\ge 2$. The speeds of convergence can be obtained from Koebe 
and from the speed of $\tau_k$.
\end{remark}

\paragraph{Metric invariants and universality.} 

Theorem~\ref{thm:scalings} implies that two maps  $f,\tilde f\in \mathcal{A}(2^\infty)$
are not necessarily differentiably conjugate on their postcritical sets.  In fact, 
there are necessary and sufficient conditions which are needed for universality:

\begin{thm}[Complete invariants for $C^1$ universality] \label{thm:nouniversality} 
Take two maps $f\in \mathcal{A}_{1,\beta}(2^\infty)$ and $\tilde f\in \mathcal{A}_{1,\tilde \beta}(2^\infty)$, with as before
$\beta,\tilde\beta>1$.  Then there exists a 
homeomorphism $h$ which is a conjugacy between the 
 postcritical sets of $f,\tilde f$ and 
 \begin{enumerate}
 \item $h$ is H\"older at $0$;
 \item $h$ is Lipschitz at $0$  $\iff$ $h$ is differentiable at $0$  
$ \iff$ $\Theta=\tilde \Theta$ and $\beta=\tilde \beta$. 
% $h(b_n)= \tilde \tilde b_n$ and $h(a_n)=\tilde a_n$
%$\Theta=\tilde \Theta$ and $\beta=\tilde \beta$
\end{enumerate}
Here $\Theta$ is defined 
through equation (\ref{eq:Theta}) %or  (\ref{eq:defTheta}) 
in Theorem~\ref{thm:scalings}. 
 
Moreover, let $\Lambda=\overline{\cup_n f^n(0)}$ be the attracting Cantor set and $\tilde \Lambda$ be the corresponding
set for $\tilde f$. Then $\Theta=\tilde \Theta$ and $\beta=\tilde \beta$ implies that 
the conjugacy $h\colon \Lambda\to \tilde \Lambda$ is differentiable in the sense that the following limit exists
$$\lim_{y\in \Lambda, y\to x} \dfrac{h(y)-h(x)}{y-x} \ne 0$$ 
and depends continuously on $x\in \Lambda$. 
\end{thm}

\medskip 

\begin{cor}  $f$ and $R^2(f)$ are not Lipschitz conjugate.
\end{cor}
\begin{proof} This follows from the previous theorem and Corollary 1. 
\end{proof} 
%\begin{proof}
%Take $\rho_k$ so that $b_{k+2}=\rho_k b_k^2$. From Theorem~\ref{thm:scalings}  it follows 
%that $\rho_k$ converges to some expression in $K_0$ and $\beta$. 
%Take $ \rho_-,  \rho_+$ and $k$ so that $ \rho_- \le \rho_+ \le \overline \rho$
%for $k\ge k_0$.  
%It then follows that 
%$$\dfrac{1}{\rho_-} [\rho_- b_{k_0} ]^{2^i}  \le b_{k_0+2i} \le\dfrac{1}{\rho_+} [\rho_+ b_{k_0} ]^{2^i}.$$
%
%\end{proof} 

\paragraph{Hausdorff dimension of the Attracting Cantor set. }

As in the symmetric case the closure of the orbit of the critical
point of $f\in \mathcal{A}(2^\infty)$ is a Cantor set which we denote
as $\Lambda(f)$.

\begin{thm} \label{thm:hdim}
  The Hausdorff dimension of the Cantor set $\Lambda(f)$, where $f
  \in \mathcal{A}(2^\infty)$,  is zero.
\end{thm}

\paragraph{Absence of Koebe space.}
In Theorem~\ref{thm:bounds} we showed that there is a monotonic
  semi-extension of the branch of $f^{2^k-1}$ defined around the
  critical value with nice bounds. The next theorem shows that such a
  property does not hold for all points of the interval.

\begin{thm} \label{thm:absencekoebe} 
For each $\tau>0$ there exists $x$ and $k$ 
so that the maximal semi-extension of the first entry map of $f$ from $x$ into $[a_k,b_k]$
does {\bf not} contain a $\tau$-scaled neighbourhood of $[a_k,b_k]$. 
\end{thm}

\paragraph{Absence of wandering intervals.} As usually, one says that $W$ is a wandering interval
if all iterates of $W$ are disjoint and if $W$ is not in the basin of a periodic attractor.
Existing proofs  for absence of wandering intervals do not go through. Indeed, we used an argument
which is quite different from anything we have seen in the literature 
showing that

%\subsection{Absence of wandering intervals}

\begin{thm}\label{thm:wandering}
No map $f\in \mathcal A_{1,\beta}(2^\infty)$  has  wandering intervals.
\end{thm}

\section{Some background material} 

In the proofs below we will need the well-known {\em Koebe Theorem}.

\begin{lemma}[Koebe Lemma]\label{thm:koebe} 
Let $g\colon T\to g(T)$ be a $C^3$ diffeomorphism with $Sg<0$. 
Assume that $J\subset T$ is an interval so that $g(T)$ contains a $\tau$-scaled
neighbourhood of $g(J)$, i.e. $g(T)\supset (1+\tau) g(J)$. 
Then for all $x,y\in J$, 
$$\dfrac{\tau^2}{(1+\tau)^2} \le \dfrac{Dg(x)}{Dg(y)}  \le \dfrac{(1+\tau)^2}{\tau^2}$$
and
$$\dfrac{\tau}{1+\tau} \dfrac{|g(J)|}{|J|} \le |Dg(y)|  \le \dfrac{1+\tau}{\tau} \dfrac{|g(J)|}{|J|}.$$
%$$\dfrac{\tau}{1+\tau} \le \dfrac{|g(J)|/|J|}{Dg(y)}  \le \dfrac{1+\tau}{\tau}.$$
\end{lemma}
\begin{proof} See the proof of Theorem IV.1.2 in \cite{MS}.
\end{proof}

Integrating the last inequalities immediately gives: 

\begin{lemma}[Corollary of Koebe] \label{cor:koebe} Let $g$ be as in the previous lemma and 
let $L\colon J\to g(J)$ be the affine surjective map with the same orientation as $g$. Then 
for all $x\in J$, 
$$Lx - \dfrac{1}{1+\tau} |g(J)| \le g(x)\le L x + \dfrac{1}{\tau} |g(J)|, \quad | \dfrac{Dg(x)}{DL(x)}-1| \le \dfrac{1}{\tau} .$$
\end{lemma}

%We will use the latter form (which is better). 

%\begin{thm}[One-sided Koebe Theorem]
%Let $g\colon T\to g(T)$ be a diffeomorphism with  $Sg<0$. 
%Let $T=(a,b)$ and assume that $x\in T$ is so that 
%$|g(b)-g(x)|\ge \tau |g(x)-g(a)|$. 
%Then 
%$$|Dg(x)|  \ge \dfrac{\tau^2}{(1+\tau)^2} |Dg(a)|.$$
%Moreover, 
%\end{thm}

 %The proof of this two theorems can be found in \cite[Section IV.1]{MS}. 

\section{Unusual bifurcations of families of maps with strong asymmetries}
\label{sec:bifurcation}

In this section we will consider the local bifurcation of families of maps $g_t$ with strong 
asymmetries. For simplicity, take $\beta>1,A>1$ and let us consider a concrete example: 
$$g_t(x)=\left\{ \begin{array}{rl} A|x|+t &\mbox{ for }x\le 0\\ 
x^\beta+t&\mbox{  for }x\ge 0.\end{array}\right.$$ 
For $t>0$ this maps has an attracting fixed point, whereas for {\em any} $t<0$ near $0$
this has a repelling fixed point $p(t)$ and an attracting periodic orbit $\{q_1(t),q_2(t)\}$  with period $2$
with $q_1(t)<p(t)<0<q_2(t)$, see the left panel of  Figure~\ref{fig-I_n-t}. So periodic doubling occurs 
precisely when $0$ is a fixed point of $g_t$ . We will call this an asymmetric period doubling bifurcation.

Note that  if we take a map with the opposite orientation, say $\hat g_t(x)=-g_t(x)$, then
the attracting fixed point disappears as soon as $t<0$ (so this is the analogue of the saddle-node bifurcation). 

In the next section we will consider the analogue of 
the periodic doubling phenomena for a family of maps $f_t$ in $\mathcal A_{1,\beta}$.  During this parameter window
only  period doubling occurs. The usual period doubling occurs when an attracting periodic orbit of period $2^{2n}$
becomes repelling and creates an attracting periodic orbit of period $2^{2n+1}$ (when the multiplier is equal to $-1$). On the other hand, 
the asymmetric periodic doubling occurs when an attracting periodic orbit of period $2^{2n+1}$ looses stability 
as it goes through the turning point $0$.

\section{The existence of a $2^\infty$ map within the space of one-sided linear unimodal maps and a full family result}
\label{subsec:existence}

This section is the only one in this paper where we consider maps in
$\mathcal A_{\alpha,\beta}$ where we allow $\alpha\ge
1$.  In fact, in the proof of Theorem~\ref{thm:existence} below we assume $\alpha=1$,
because when $\alpha>1$ the proof is simpler:  in that case the proofs in \cite{MT}  
(for the unimodal setting)
and in \cite{dMvS}  (for the multimodal setting) go through.

We say that a non degenerate interval $I$ is {\em restrictive} of
period $d>0$ of a unimodal map $f$ if it contains the critical point
of $f$, the interiors of $I, f(I), \ldots, f^{d-1}(I)$ are disjoint
and $f^d(I)\subset I$, $f^d (\partial I) \subset \partial I$. If a map $f$
has a restrictive interval $I$ of period $d$ is called {\em
  renormalizable} and $f^d|I$ is called a renormalization of $f$. Note
that any renormalization of a unimodal map is unimodal.

The maps in class $\mathcal{A}_{\alpha,\beta}(2^\infty)$ we defined  are all infinitely
renormalizable, moreover all the restrictive intervals $I_1 \supset
I_2 \cdots \supset I_n \cdots$ are of periods $2,2^2,\ldots,2^n,\ldots$.

The following theorem implies Theorem~\ref{thm:existence2}:

\begin{thm}\label{thm:existence} 
  Consider a family $f_t\colon [a_0,b_0]$, $t\in [0,1]$  in  $\mathcal A_{\alpha,\beta}$ with $1\le \alpha<\beta$ so that
  $t\mapsto f_t|[a_0,0] \in C^1 $ and $t\mapsto f_t|[0,b_0] \in C^1 $
  are continuous and so that $f_0$ has a unique attracting fixed point
  and so that $f_1$ is surjective.  Then there exist two sequences of
  parameters $u_1<u_2<\cdots <v_2<v_1$ such that
  \begin{itemize}
  \item for $t \in (u_n, v_n]$ the map $f_t$ is $2^n$ renormalizable,
    more precisely, there exists a non degenerate restrictive interval
    $I_{n,t}$ of period $2$ of the map $f_t^{2^{n-1}}|I_{n-1,t}$
    continuously depending on the parameter $t\in (u_n, v_n]$ (here we set
    $I_{0,t}=[a_0,b_0]$);
  \item when $n$ is even then $f_{u_n}^{2^{n-1}}(0)=0$ and $\lim_{t\downarrow u_n} I_{n,t}=\{0\}$, 
 while for $n$ is
    odd $f_{u_n}$ has a parabolic periodic orbit of period $2^{n-1}$
    with multiplier $-1$ and  and $\lim_{t\downarrow u_n} I_{n,t}$ is non-degenerate;
  \item $f_{v_n}^{2^n}(I_{n,v_n}) = I_{n,v_n}$, that is
    $f_{v_n}^{2^n}|I_{n,v_n}$ is surjective.
\end{itemize}

Clearly, $f_t \in \mathcal{A}_{\alpha,\beta}$ for any $t \in \cap_n (u_n,v_n)$.
\end{thm}

Note that $\cap_n (u_n,v_n)\ne \emptyset$ because the intervals $(u_n,v_n)$ 
are properly nested. In particular, the family (\ref{examplefamily}) (with $\beta>1$) contains a map in the class $\mathcal A_{\alpha,\beta}(2^\infty)$.

\begin{proof}
 The proof we will give of this theorem is almost the same as a proof
  based on a bifurcation analysis for smooth unimodal maps and
  will use the following two properties:

  (1) whenever $f_t$ has an attracting periodic orbit then $0$ is in the
  immediate basin of this attractor.  This holds since $f$ has negative
  Schwarzian derivative, and therefore the immediate basin of a periodic
  attractor contains a turning point of an iterate of $f$ and hence
  $0$ is also in the immediate basin of this periodic attractor.
  
  (2) whenever $0$ is a (topologically) {\em attracting} periodic point
  of $f_{t_0}$ of period $n$ then $f_t$ has a periodic attractor of
  period $n$ or period $2n$ for each $t$ near $t_0$. Note that within
  this class of maps it is no longer true that if $0$ is periodic then
  it is also attracting (it can be repelling on one side when $\alpha=1$).

  \begin{figure}
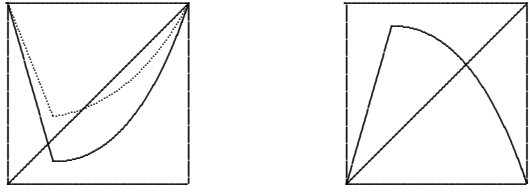
 \hfil
    \beginpicture
    \dimen0=0.3cm
    \setcoordinatesystem units <\dimen0,\dimen0>  point at -15 0
    \setplotarea x from -4 to 4, y from -6 to 6
    \setlinear
    \plot -4 -4 4 -4  4 4 -4 4 -4 -4 /
    \plot -4 -4 4 4 / 
    % \setdots <0.8mm> 
    % \plot -4.5 -4 4.5 -4 / 
    % \plot 4.5 -4 4.5 -6 / 
    % \plot -4.5 -4 -4.5 -5.5 / 
    % \plot 1.3 1.3 1.3 -4 / 
    \setsolid
    % \put {$a_0$} at -3.5 -4.5 
    % \put {$b_0$} at 3.7 -4.5 
    % \put {$a_0$} at -4.8 -3.5 
    % \put {$b_0$} at -4.8 3.5 
    % \put {\small $0$} at -2 -4.5 
    \setlinear 
    \plot  -4 -4 -2 3  /
    \setquadratic 
    \plot -2 3 1.3 1.3 4 -4 /
    % \setdots <0.3mm> 
    % \setlinear 
    % \plot  -2.6 -2.6 -2 -1.3   /
    % \setquadratic 
    % \plot -2 -1.3 -1.3 -1.6 -0.8 -2.6 /
    % \setsolid 
    % \setlinear 
    % \plot -2.6 -2.6 -.8 -2.6 -0.8 -0.8 -2.6 -0.8 -2.6 -2.6 / 
    \setcoordinatesystem units <\dimen0,\dimen0>  point at  0 0
    \setplotarea x from -4 to 4, y from -6 to 6
    \setlinear
    \plot -4 -4 4 -4  4 4 -4 4 -4 -4 /
    \plot -4 -4 4 4 / 
    % \setdots <0.8mm> 
    % \plot -4.5 -4 4.5 -4 / 
    % \plot 4.5 -4 4.5 -6 / 
    % \plot -4.5 -4 -4.5 -5.5 / 
    % \plot 1.3 1.3 1.3 -4 / 
    \setsolid
    % \put {$a_0$} at -3.5 -4.5 
    % \put {$b_0$} at 3.7 -4.5 
    % \put {$a_0$} at -4.8 -3.5 
    % \put {$b_0$} at -4.8 3.5 
    % \put {\small $0$} at -2 -4.5 
    \setlinear 
    \plot  -4 4 -2 -3  /
    \setquadratic 
    \plot -2 -3 1.3 -1.3 4 4 /
    \setdots <0.4mm> 
    \setlinear 
    \plot  -4 4 -2 -1  /
    \setquadratic 
    \plot -2 -1 1.3 0.5 4 4 /
    \endpicture
    \caption{\label{fig-I_n-t} 
   $f^{2^n}|I_{n,t}$ for $n$ odd (on the
      left) and $n$ even (on the right).  When $n\ge 2$ is even then
      $I_{n,t}\to \{0\}$ as  $t\downarrow u_n$ and for $t\in (u_n,v_n)$ the only fixed point of
      $f_t^{2^n}$ in the interior of $I_{n,t}$ lies to the right of
      $0$.}
  \end{figure}

  % Let us prove the existence of a parameter $t_*$ so that $f_{t_*}$
  % has the required properties, by analysing what bifurcations occur in
  % the family $f_t$ analogous to the period doubling bifurcations which
  % occur in the quadratic family.  To do this, we will prove
  % inductively that there exists a nested sequence of maximal parameter
  % intervals $[t_n,v_n]\subset [t_{n-1},v_{n-1}]$ so that for each for
  % $n\ge 0$ and each $t\in [t_n,v_n]$ there exists a (possibly
  % degenerate) interval $I_{n,t}\ni 0$ (depending continuously on the
  % parameter) so that $f_{t}^{2^n}\colon I_{n,t}\to I_{n,t}$ is a
  % unimodal map with
  % $f_t^{2^n}(\partial I_{n,t})\subset \partial I_{n,t}$ and so that
  % $I_{n,t},\dots,f_t^{2^n-1}(I_{n,t})$ have disjoint interiors. The
  % interval $I_{n,t}$ is allowed to degenerate to a point only when
  % $t=t_n$.  The map $f^{2^n}_{t_n}|I_{n,t_n}$ will have an attractive
  % fixed point and the map $f^{2^n}_{v_n}|I_{n,v_n}$ will be surjective
  % onto $I_{n,v_n}$.

  Analysing what bifurcations occur in
  the family $f_t$ analogous to the period doubling bifurcations which
  occur in the quadratic family, we will prove
  inductively that there exists a nested sequence of maximal parameter
  intervals described by the theorem.
  
  Slightly abusing notation we set $u_0=0$, $v_0=1$ and
  $I_{0,t}=[a_0,b_0]$. Clearly all the properties stated in the
  theorem are satisfied except one claiming that the critical point is
  fixed by  $f_0$. This does not affect the proof which follows. So
  assume by induction that such parameter interval $[u_n,v_n]$ exists
  for some integer $n$.  There are two possibilities.

  (i) $n$ is even. In this case for each $t\in [u_n,v_n]$,
  $f_t^{2^n}|I_{n,t}$ is of type $+-$ and $\alpha\beta$, i.e.,
  orientation preserving (resp. reversing) to the left (right) of $0$
  and the order of the critical point is of order $\alpha$ to the left
  of $0$ and of order $\beta$ to the right of $0$. We know that
  $f_{v_n}^{2^n}|I_{n,v_n}=I_{n,v_n}$, therefore there exists an
  orientation reversing fixed point $p_n>0$ of
  $f_{v_n}^{2^n}|I_{n,v_n}$. Note that this fixed point is repelling
  because the orbit of the critical point of $f_{v_n}^{2^n}$ belongs
  to the boundary of $I_{n,v_n}$. Since the multiplier of $p_n$ is not
  equal to one this fixed point persists when we change a parameter in
  a neighbourhood of $v_n$, that is there is a continuous function
  $p_{n,t}$ defined for $t$ in some    interval $W_n\ni v_n$ such that 
  $f^{2^n}_t(p_{n,t})=p_{n,t}$ and $p_{n,v_n}=p_n$. We will assume
  that $W_n$ is the maximal interval where such a function can be
  defined.
% Let $W_n$ be the set of parameters $t$ so that there exists $p_{n,t}>0$ in $I_n$
%  so that $f^{2^n}_t(p_{n,t})=p_{n,t}$. Note that $v_n\in W_n$.
% % 
%  Then there exists an interval $W_n$ of
%  parameters containing $v_n$ and a family of points $p_{n,t}$ such
%  that $f_t^{2^n}(p_{n,t})=p_{n,t}$ for all $t\in W_n$ and
%  if $t \in W_n \cap (u_n, v_n]$, then $p_{n,t} \in I_{n,t}$.}
%Then for each $t\in W_n\cap (u_n,v_n]$ we have $p_{n,t}\in I_{n,t}$.
%Let $u_{n+1}<v_n$ be minimal so that $Df_{u_{n+1}}^{2^n}(p_{n,u_{n+1}})\le -1$ for each $t\in [u_{n+1},v_n]$.
%   Let
%  $u_{n+1}< v_n$ be maximal such that
%  $Df_{u_{n+1}}^{2^n}(p_{n,u_{n+1}})=-1$, that is $p_{n,u_{n+1}}$
%  becomes a parabolic periodic point of $f$ with multiplier $-1$.  
  Let $u_{n+1}< v_n$ be maximal such that
  $Df_{u_{n+1}}^{2^n}(p_{n,u_{n+1}})=-1$, that is $p_{n,u_{n+1}}$
  becomes a parabolic periodic point of $f$ with multiplier $-1$. 
  Such a point $u_{n+1}$ exists and $u_{n+1}>u_n$ because the multiplier of
  $p_{n,t}$ varies continuously with the parameter 
 $t\in W_n\cap (u_n,v_n]$,  since $Df^{2^n}_t(p_{n,t})<-1$
  for $t=v_n$ and since for any $t$ we have $\lim_{x\downarrow 0} Df^{2^n}_t(x)=0$
  while $f_{u_n}^{2^{n-1}}(0)=0$.
%  \textcolor{yellow}{and at the boundary
%  point $t$ of $W_n$ one cannot have $Df_t^{2^n}(p_{n,t})<-1$ unless
%  $t$ is either $0$ or $1$;   secondly, we know that
%  $f_{u_n}^{2^n}(0)=0$, thus $p_{n,u_n}=0$ 
%  and  $\lim_{x\downarrow{}p_{n,u_n}}Df_{u_n}^{2^n}(x)=0$ if $u_n \in
%  W_n$. This also implies that $u_{n+1}>u_n$.}

  For $t\in [u_{n+1}, v_n]$ let $\hat p_{n,t}<0$ denote a preimage of
  $p_{n,t}$ under $f_t^{2^n}|I_{n,t}$ and let
  $I_{n+1, t}=[\hat p_{n,t}, p_{n,t}]$. Since $f$ has negative
  Schwarzian derivative it follows that $p_{n,u_{n+1}}$ is a parabolic
  periodic point of $f_{u_{n+1}}$ and that the critical point belongs
  to the basin of attraction of $p_{n,u_{n+1}}$. This in turn implies
  that
  $f_{u_{n+1}}^{2^{n+1}}(I_{n+1,u_{n+1}}) \subset I_{n+1,u_{n+1}}$,
  i.e.,  $I_{n+1,u_{n+1}}$ is a restrictive interval of
  $f_{u_{n+1}}^{2^n}$ of period 2. Note that if $t$ is slightly larger
  than $u_{n+1}$, the interval $I_{n+1,t}$ is still a restrictive
  interval of period 2 of the corresponding map.  We know that
  $f_{v_n}^{2^n}(0)$ belongs to the boundary of $I_{n,v_n}$ and
  therefore $f_{v_{n}}^{2^{n+1}}(0) \not \in I_{n+1, v_n}$.  Define
  $v_{n+1}$ to be infimum of all parameters $t > u_{n+1}$ such that
  $f_{v_{n+1}}^{2^{n+1}}(0) \not \in I_{n+1, v_n}$, thus
  $f_{v_{n+1}}^{2^{n+1}}(0)$ belong to the boundary of
  $I_{n+1,v_{n+1}}$. It must be the left boundary point (that is
  $f_{v_{n+1}}^{2^{n+1}}(0) = \hat p_{n,v_{n+1}}$) because otherwise
  the condition $Df_t^{2^n}(p_{n,t}) \le -1$ for $t\in [u_{n+1},v_n]$
  would be broken.

  It is easy to see that the constructed points $u_{n+1}$, $v_{n+1}$
  and the intervals $I_{n+1, t}$ satisfy all the induction
  assumptions. Note that in this case the intervals $I_{n+1, t}$ are
  non degenerate for all $t \in [u_{n+1}, v_{n+1}]$.

  (ii) $n$ is odd.  In this case $ f_{u_n}^{2^n}|I_n$ is of type $-+$
  and $\alpha\beta$.  The construction will be very similar to the
  case of even $n$ with some modifications relating to the asymmetric
  period doubling bifurcation.

  Arguments similar to the case when $n$ is even show that there exists a
  maximal $u_{n+1} < v_n$ such that $f_{u_{n+1}}^{2^n}(0)=0$. Then for
  all $t\in [u_{n+1}, v_n]$ there exists an orientation reversing
  fixed point $p_{n,t} \in I_{n,t}$ of $f_t^{2^n}$. Note that $p_{n,t}$
  is negative (i.e. it is to the left of the critical point).
  Define $\hat p_{n,t}>0$ to be a preimage of $p_{n,t}$ under
  $f_t^{2^n}|I_{n,t}$ and let $I_{n+1, t}=[\hat p_{n,t}, p_{n,t}]$ for
  all $t\in [u_{n+1}, v_n]$ as before. Note that
  $p_{n,u_{n+1}}=\hat p_{n,u_{n+1}}=0$ and the interval
  $I_{n+1, u_{n+1}}$ degenerates to the critical point. For all other
  values of the parameters the intervals $I_{n+1, t}$ are non
  degenerate.  In Section~\ref{sec:bifurcation} it was explained that
  for values of parameters $t$ slightly larger than $u_{n+1}$ the
  interval $I_{n+1,t}$ is a restrictive interval of period 2 of the
  map $f_t^{2^n}$. As before define $v_{n+1}> u_{n+1}$ to be maximal
  such that $I_{n+1,t}$ is a restrictive interval of period 2 of the
  map $f_t^{2^n}$ for all $t\in (u_{n+1}, v_{n+1})$ and note that
  $v_{n+1} < v_n$. 

\end{proof}

In fact,  we have

\begin{thm} Any family $\{f_t\}$ as in Theorem~\ref{thm:existence} is a full family in  
the following sense. Take a quadratic interval map $Q$ without periodic attractors. Then there exists a parameter $t$ so that 
$f_t$ combinatorially equivalent to $Q$. 
\end{thm}
\begin{proof} 
   In \cite{MT}, see also \cite{JR}, 
 this result is shown  for families $f_t$ of unimodal maps with $\alpha,\beta>1$. 
Let us give an outline of that proof.  The main ingredients are
the notion of the {\em kneading invariant} $\nu(f)$ of a unimodal map $f$,  
the abstract notion of an {\em admissible kneading sequence} $\nu$, the 
lexicographical ordering on the space of kneadings, and a topology on this space.  The required result
follows by showing that for each admissible kneading sequence $\nu$ there exists $t$ so that 
$\nu=\nu(f_t)$.  Proving this relies on some kind of intermediate value in the space of kneadings, 
by analysing the discontinuities of the map $t\mapsto \nu(f_t)$ and using the following two observations:\\
(1) if $t_0$ is a parameter for which the critical point of $f_{t_0}$ is non-periodic, then the kneading
invariant $t\mapsto \nu(f_t)$ is continuous at $t=t_0$;\\
(2) if $t_0$ is a parameter for which the critical point of $f_{t_0}$ is periodic, then for
$t\approx t_0$ the map $f_t$ still has a periodic attractor (here it used that $\alpha,\beta>1$). 
This then makes it possible to show that for each $s,t\approx t_0$  the kneading sequences 
$\nu(f_t)$ and $\nu(f_s)$ are 
the same up to a simple operation (related to some star product). Thus one obtains
that there are no admissible kneading sequences that get skipped. 

 In our case, when $\alpha=1<\beta$ the first step still holds, but in the 2nd step
the map $f_t$ may not have a  periodic attractor when $t\approx t_0$. However, 
as is shown in the previous theorem, the kneading sequences for nearby maps
still bifurcate the same way as they do for nearby smooth maps. Thus the proof
in \cite{MT}, see also \cite{JR},  goes through. 

%
%
%% critical point
%%of $f_t$ for $t\approx t_0$ is still attracted to a periodic orbit. 
%%
%%kneading
%%invariant $\nu(f_{t_0})$
%%kneading invariant of 
%%The idea is that the kneading invariant $\nu(f_t)$ depends continuous 
%
%The proof of this theorem can be deduced from the previous proof as this 
%shows that the kneading invariant of a family of maps $f_t$ in $\mathcal{A}_{\alpha,\beta}$ bifurcates 
%in the same way as in the quadratic family (but no assertion about monotonicity of the bifurcations
%can be deduced from this), see also \cite{MT}. 

Another way of proving this theorem is by adapting
the proof given in  \cite[Theorem II.IV.1]{dMvS}. That proof follows a
Thurston mapping approach  and, contrary to the proof from \cite{MT}, 
also  applies to multimodal families.  To apply this proof in our setting, 
one needs to show that a certain map defined on some open symplex 
is {\lq}repelling{\rq} near the boundary of this simplex. We will not give the details
for the required modifications here. 
\end{proof}

\section{The smallest interval argument}

The usual smallest interval argument in the current setting gives a weaker statement
than in the {\lq}symmetric{\rq} case: 

\begin{lemma}\label{lem1} There exists $\tau>1$ so that the following holds. 
Consider $I=[a_n,b_n]$ and choose $x\notin I$.  Assume that there exists $k>0$ (minimal) so that $f^k(x)\subset I$.
Then there exists an interval $T\ni x$ so that $f^k|T$ is a diffeomorphism and $f^k(T)\supset [\tau a_n,\tau b_n]$. 
 \end{lemma}
\begin{proof} For completeness let us include the proof of this lemma. 
Let $T$ be the maximal interval  $T\ni x$ so that $f^k|T$ is a diffeomorphism.
 By maximality of $T$ and since $f^i(x)\notin I$ for all $i=0,\dots,k-1$
there exist integers $0<i_0,i_1<2^n$  so that  $f^k(T)\supset [f^{i_0}(I),f^{i_1}(I)]$ 
where $f^{i_0}(I)$ and $f^{i_1}(I)$ are to the left respectively to the right  of $I$. So it suffices to show that 
$[f^{i_0}(I),f^{i_1}(I)]\supset  [\tau a_n,\tau b_n]$ for some
universal choice of $\tau>0$. % for any two such intervals. 

Write $I_i=f^i(I)$ and let $3\le m\le 2^n$ be so that $I_m$ is the smallest of the intervals 
$I_3,\dots,I_{2^n}$. Let $K_m$ be  the smallest interval containing the left and right neighbours of $I_m$
from the collection $I_1,\dots,I_{2^n}$ (such neighbouring intervals exist because $m\ge 3$). It follows that 
$K_m$ contains a $\tau_0$-scaled neighbourhood of $I_m$ where 
$\tau_0>0$ is independent on $n$  (here we use that $I_1,I_2$ are not much smaller than $I_3$). 
Let $K_1\supset I_1$ be the maximal interval on which $f^{i_0-1}|K_1$ is a diffeomorphism
with $f^{i_0-1}(K_1)\subset K_m$. By maximality, $f^{i_0-1}(K_1)= K_m$. By Koebe it follows that 
$K_1$ contains a $\tau_1$-scaled neighbourhood of $I_1$. Hence $K_0:=f^{-1}(K_1)$ contains 
$ [\tau_1' a_n,\tau_1'' b_n']$ where $\tau_1'=\tau_1^{1/\alpha}$ and $\tau_1''=\tau_1^{1/\beta}$. 
Note that because $|a_n|<<b_n$, this latter interval is no longer a definite interval around $[a_n,b_n]$. 
Note also that by the choice of $K_m$ the  interval $K_0$ is contained in any interval 
 of the form $[f^{i_0}(I),f^{i_1}(I)]$ where $f^{i_0}(I)$ and $f^{i_1}(I)$ are to the left respectively to the right  of $I$. 
 \end{proof}

%\section{Proof of the first part of Theorem~\ref{thm:bounds}: big bounds} 
\section{Big bounds} 
\label{Sec:bigbounds} 

Since $\alpha=1$, we can consider a semi-extension of $f$ of the  {\lq}linear{\rq} branch
and use the following strategy. First, using the standard smallest
interval argument we have already shown that there exists a definite space to the right
of the renormalization intervals. Next we will show that either there is definite
space to the left of the renormalization interval for the semi-extension or this space is at least as big as
the space on the previous level.  Considering several scenarios, this will 
imply that there is some definite space
on both sides of the renormalization intervals  (for the semi-extension). Once there is {\lq}space{\rq} on both sides
of the renormalization intervals we can repeat the argument used to obtain it
and get as much space as one may want. From this the rest follows.

\subsection{Using semi-extensions} \label{subsec:semiexten} 
Let $f^{2^{k}-1} \colon J_k \to [a_k , b_k ]$ be the branch of the first entry map to $[a_k , b_k ]$ for which
$c_1 := f (0) \in J_k$. Note that this is a surjective diffeomorphism.
Let $\hat T_k\supset J_k$ be the maximal interval around $f(0)$ so that  $f^{2^{k}-1}|\hat T_k$ is a diffeomorphism  and let
$[\hat A_k,\hat B_k]:=f^{2^{k}-1}(\hat T_k)$ where $\hat A_k<\hat B_k$. Note that $f^{2^{k}-1}|\hat T_k$  is orientation preserving (reversing)
when $k$ is even (odd). 
We also define an interval  $[A_k,B_k]\supset [\hat A_k,\hat B_k]$,
with $A_k<B_k$,
associated to 
the semi-extension as follows. 
Let  $E_k : T_k\to  [A_k , B_k ]$ be the maximal monotone surjective
{\em semi-extension} of  $f^{2^{k}-1}\colon J_k \to [a_k , b_k ]$ such that $A_k\le a_k < 0 < b_k \le B_k$. 
(In principle this extension depends on the choice of the extension
$f_1\colon [0,\epsilon)\to \R$ of $f\colon
[a_0,0]\to \R$.)

\subsection{Useful  dynamical and non-dynamical points $a'_k,b_k',d_k,e_k$.}
Let
$[a'_k , e_k ] = f_1 ^{-1}(T_k )$, $a_k'<a_k<0<e_k$, and therefore $E_k \circ f_1 : [a'_k , e_k ] \to [A_k , B_k ]$ is the maximal
monotone surjective semi-extension of $f^{2^k} : [a_k , 0] \to [a_k , b_k ]$. Also, define the point $b'_k > b_k$ as the right boundary point of the
interval $f_2^{-1}(T_k )$.  Furthermore, define $d_k \in [0, e_k ]$ such that $E_k \circ f_1 (d_k ) = b_k$
for even values of $k$. When $k$ is odd the point $d_k$ is not defined.
The properties of these points are made clear in Figure~\ref{fig:PD}
and the purpose of these points is expanded on in \S\ref{subsec:sketch} where
a  sketch of the proof of Theorem~\ref{thm:bounds} is given.

Since $E_k$ is orientation preserving (reversing) 
when $k$ is even (odd),  the following holds:
\begin{itemize}
\item  for even values of $k$
$$
\begin{array}{rl}
A_k &= E_k \circ f_1 (a'_k ) = E_k \circ f_2 (b'_k ), \\
B_k &= E_k \circ f_1 (e_k )\end{array} $$
\item and for odd $k$
$$\begin{array}{rl}
B_k &= E_k \circ f_1 (a'_k ) = E_k \circ f_2 (b'_k ),\\
A_k  &= E_k \circ f_1 (e_k ).
\end{array}$$
\end{itemize}
As we will show in Lemma~\ref{lem2},  $B_k=\hat B_k$ but  in general $A_k\ne \hat A_k$. 

\begin{figure}[h]
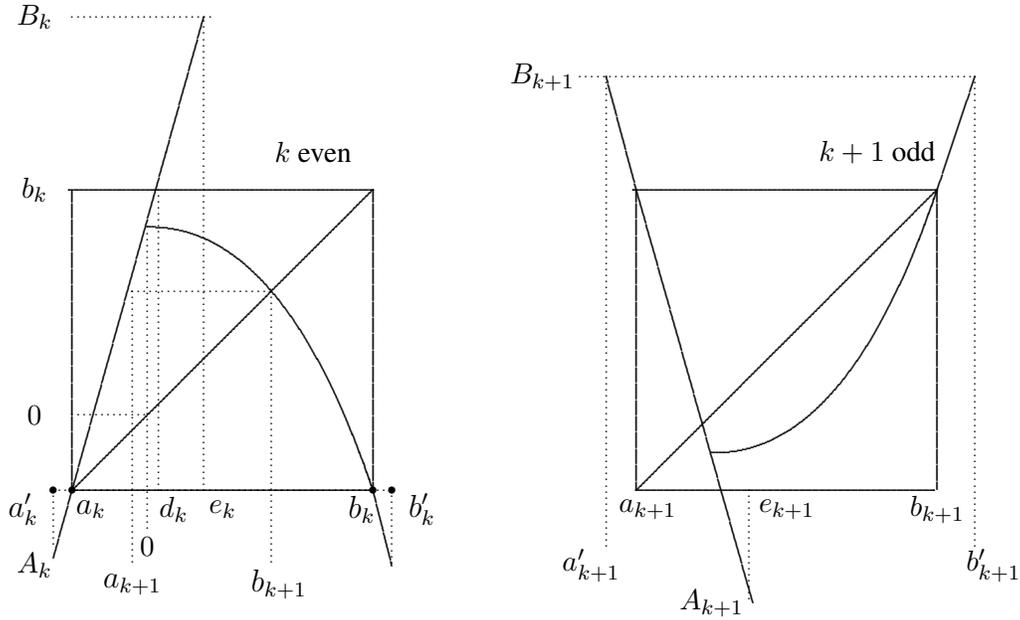
 \hfil
\beginpicture
\dimen0=0.5cm
\setcoordinatesystem units <\dimen0,\dimen0>  point at 30 0
\setplotarea x from -4 to 4, y from -3 to 3
\setlinear
\plot -4 -4 4 -4  4 4 -4 4 -4 -4 /
\plot -4 -4 4 4 / 
\setdots <0.8mm> 
\plot -4.5 -4 4.5 -4 / 
\plot 4.5 -4 4.5 -6 / 
\plot -4.5 -4 -4.5 -5.5 / 
 \plot 1.3 1.3 1.3 -6 / 
 \plot -2.4 1.3 1.3 1.3 / 
  \plot -2.4 1.3 -2.4 -6 / 
\setsolid
\put {$a_k$} at -3.5 -4.5 
\put {$b_k$} at 3.7 -4.5 
\put {$b_k'$} at 5.3 -4.5 
\put {$a_k'$} at -5.3 -4.5 
\put {$b_{k+1}$} at 1.5 -6.5 
\put {$a_{k+1}$} at -2.4 -6.5
\put {\tiny $\bullet$} at 4.5 -4  
\put {\tiny $\bullet$} at 4 -4
\put {\tiny $\bullet$} at -4 -4
\put {\tiny $\bullet$} at -4.5 -4
\put {\small  $0$} at -2 -5.5 
\put {\small $d_k$} at -1.3 -4.5 
\put {\small $e_k$} at -0 -4.5 
\put {\small $b_k$} at -5 4  
\put {\small $B_k$} at -5 8.6 
\put {\small $A_k$} at -5 -6  
\put {\small $0$} at -5 -2 
\plot -4.5 -5.8 -4 -4 -1.5 5    -0.5 8.6 /  % 2.5 9 i.e. 1 3.6
\setquadratic 
%\plot -2 3 0 1.5 4 -4 /
\plot -2 3 1.3 1.3 4 -4 /
\setlinear 
\plot 4 -4 4.5 -6 / 
\setdots  <0.8mm>  \setlinear
\plot -0.5 -4 -0.5 8.6 /  
\plot -1.7 -4 -1.7 4 / 
\plot -4 8.6 -0.2 8.6 / 
\plot -2 -5 -2 3 /
\plot -4 -2 -2 -2  /
\put {\small $k$ even } at  2.5 5 
\setcoordinatesystem units <\dimen0,\dimen0>  point at 15 0
\setplotarea x from -4 to 4, y from -6 to 6
\setlinear \setsolid
\plot -4 -4 4 -4  4 4 -4 4 -4 -4 /
\plot -4 -4 4 4 / 
\plot -4 4 -0.9 -7 / 
\plot -4.8 7 -4 4  / 
\plot 4 4 5 7 / 
\setquadratic 
%\plot -2 3 0 1.5 4 -4 /
\plot -2 -3 1.3 -1.3 4 4 /
\setdots <0.8mm>  
\setlinear 
\plot -4.8 7 -4.8 -5.5  /
\plot 5 7 5 -5.5 /
\plot -5.5 7 5 7 / 
\plot -1 -4 -1 -7 / 
\put {\small $a_{k+1}'$} at -5.2 -6 
\put {\small $b_{k+1}'$} at 5.5 -6 
\put {\small $a_{k+1}$} at -3.7 -4.5 
\put {\small $b_{k+1}$} at 4 -4.5 
\put {\small $B_{k+1}$} at -6.5 7 
\put {\small $A_{k+1}$} at -2 -7 
\put {\small $e_{k+1}$} at 0 -4.5 
\put {\small $k+1$ odd } at  2.5 5 
\endpicture
\caption{\label{fig:PD} $f^{2^k}|I_k$ and $f^{2^{k+1}}|I_{k+1}$ when $k$ is even and their semi-extensions.
Note that the points $d_k,e_k,a_k',b_k'$ are defined using the semi-extension rather than dynamically.}
\end{figure}
%\label{figL:PDpage} 

\begin{figure}[h]
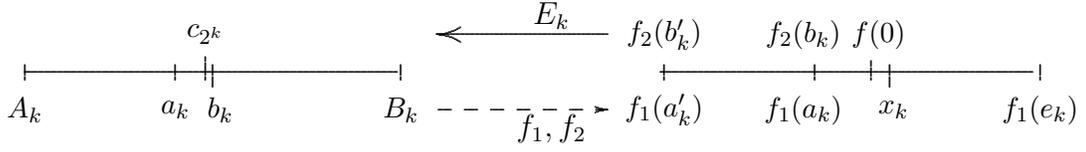
 \hfil
\beginpicture
\dimen0=0.5cm
\setcoordinatesystem units <\dimen0,\dimen0>  point at 0 0
\setplotarea x from -10 to 10, y from -1 to 1
\setlinear
\plot -12 0  -2 0 / 
\plot 5 0 15 0 / 
\put {\small $A_k$} at -12 -1 
\put {\small $a_k$} at -8 -1 
\put {\small $c_{2^k}$} at  -7.2 1 
\put {\small $b_k$} at -6.8 -1 
\put {\small $B_k$} at  -2 -1 
\plot -12 -0.2 -12 0.2 / 
\plot -8 -0.2 -8 0.2 / 
\plot -7 -0.4 -7 0.2 / 
\plot -7.2 -0.2 -7.2 0.4 / 
\plot -2 -0.2 -2 0.2 / 
\put {\small $f_1(a_k')$} at 5 -1 
\put {\small $f_2(b_k')$} at 5 1 
\put {\small $f_1(a_k)$} at 8.7 -1 
\put {\small $f_2(b_k)$} at 8.7 1 
\put {\small $f(0)$} at 10.7 1 
\put {\small $x_k$} at 11.1 -1 
\put {\small $f_1(e_k)$} at  15 -1 
\plot 5 -0.2 5 0.2 / 
\plot 9 -0.2 9 0.2 / 
\plot 10.5 -0.2 10.5 0.3 / 
\plot 11 -0.4 11 0.2 / 
\plot 15 -0.2 15 0.2 / 
\put {$E_k$} at 2 1.5 
\arrow <3mm> [0.2,0.67] from 3.5 1  to -1 1
\setdots <0.5mm>
\setdashes 
\arrow <3mm> [0.2,0.67] from -1 -1 to 3.5 -1  
\put {$f_1,f_2$} at 2 -1.5 
\endpicture
\caption{\label{fig:Jk}  When $k$ is even, $E_{k+1}=E_k\circ f_2 \circ E_k$ and $E_k$ is orientation preserving. 
Here $E_k(x_k)=b_k$. It is not clear where $b_k'$ and $a_k'$ are in relation to $B_k$ and $A_k$.} \end{figure}

\subsection{Sketch of the proof of Theorem~\ref{thm:bounds}\label{subsec:sketch}} 
Note that the interval $[A_k,B_k]$ is the range of the {\em semi-extension} of the first entry map 
$E_k$ (rather than  its diffeomorphic extension).
Therefore none of the points $A_k,B_k,a_k',b_k',e_k$  have a priori any dynamical interpretation. As it turns out $B_k=\hat B_k$, see Lemma~\ref{lem2} and therefore $B_k$ has a dynamical interpretation, but
 none of these other points do.

Our aim in this section is to show $[A_k,B_k]$ is much bigger than $[a_k,b_k]$ (for $k$ even and large). 
To do this, we will consider all the various positions of $a_k',b_k',e_k$ and show that 
each of these give some recursive information. Let us outline the argument. 

{\bf Step 1 (\S\ref{subsec:topproper})} consists in obtaining %(in Lemma~~\ref{lem2}, \S)
various topological properties, including 
that if $e_{k+1}<b_{k+1}$ then one can propagate the semi-extension of level $k+1$ to level $k+2$. 
More precisely,  
 for $k$ even \begin{equation} e_{k+1}<b_{k+1} \implies A_{k+2}=A_{k+1}.\label{eq:step1} \end{equation} \newline \indent 
{\bf Step 2 (\S\ref{subsec:firstrecursive})} consists in using some cross-ratio inequality and the strong-asymmetry 
of $f$ to  show that there exists a 
$C>0$ so that the following recursive inequality holds for all
  even $k$
 \begin{equation}d_k \le  C b_{k+1}^{\beta-1} b_k. \label{eq:step2} \end{equation}
\newline \indent 
{\bf Step 3  (\S\ref{subsec:dichotomy})} gives the following  dichotomy, see  Lemma~\ref{lem5},
%To do this we use (\ref{eq:step2}) and  a sequence of furtxer arguments
%to conclude, in Lemma~\ref{lem5}, that  
\begin{equation}\mbox{ either } |A_k | > C b_{k+1}\mbox{ or }e_k < b_{k+1}. \label{eq:step3} \end{equation} 
\newline \indent 
{\bf Step 4  (\S\ref{subsec:condbound})}
shows that the assumption $|A_k|>C b_{k+1}$ implies some distortion control of the restriction of $f^{2^k}$
to $[b_{k+1},b_k]$, see Lemma~\ref{lem6}. 
\newline \indent 
{\bf Step 5   (\S\ref{subsec:somespace})} consists in showing that 
one has infinitely often space. 
 This means that we need to show that there exists $\tau>1$ so that 
$[A_k,B_k] \supset \tau  [a_k,b_k]$ for infinitely many $k$ even. 
From the smallest interval argument in Lemma~\ref{lem1} and the strong asymmetry 
we have that  $B_k>\tau b_k>>a_k$ for some $\tau>1$. So   it
suffices to show that there exists $C>0$ so that  $|A_k|>C b_{k}$ holds for infinitely many $k$.
From the dichotomy (\ref{eq:step3})  it follows that either from time to time  the  inequality $|A_k|>C b_{k+1}$  holds or 
 $e_k<b_{k+1}$ holds for all $k$ even and large. 
If the latter holds, then (\ref{eq:step1}) implies that $A_{k+2}=A_{k+1}$ for {\em all} $k$ even and large.
Using a further argument  using equation (\ref{eq:step2}), using Step 4,  we can then {\lq}replace{\rq} the inequality 
$|A_k|>Cb_{k+1}$ by the inequality $|A_k|>Cb_k$, and 
obtain in Lemma~\ref{lem7} that  
$$|A_k|>Cb_k \mbox{  holds only finitely often }  \implies   \exists k_0 \mbox{ with } A_{k_0}=A_{k_0+1}=A_{k_0+2}=\dots.$$ 
Of course the latter also  implies $|A_k|>Cb_k$ for $k$ large, thus concluding Step 5. 
\newline \indent 
{\bf Step 6  (\S\ref{subsec:improvingspace})}
consists in  showing that if  $|A_k|>C  b_k$  for some even $k$
 (or in other words if the space condition $[A_k,B_k] \supset \tau  [a_k,b_k]$ holds)
 then one gets {\em large} space in the next step. Thus we obtain an increasingly growing space. 
%This is done in Lemmas~\ref{lem8} and \ref{lem9}.
\newline \indent 
{\bf Step 7 (\S\ref{subsec:superexpspace})} In this final step we show that the space is growing 
superexponentially fast. This is done in Lemma~\ref{lem10}, and this then concludes the proof of Theorem~\ref{thm:bounds}.

 \subsection{Some topological properties of $a'_k,b_k',d_k,e_k$ \label{subsec:topproper}}

Let us list a number of more or less obvious relations between the points we defined.
For example, assertion (4) and (5) show that if some metric properties
hold for the non-dynamically defined points $b_k'$ and $e_k$
then the semi-extension from one level can be used to obtain a semi-extension
of the next level.

\begin{lemma}\label{lem2}
Let $k\ge 2$ be an even integer. Then
\begin{enumerate}
\item $B_{k+1} =  B_{k+2}=\hat B_{k+1}=\hat B_{k+2}=c_{2^k}$;
\item  $e_{k+2 }< d_k$;
\item $\hat A_{k}=\hat A_{k+1}=c_{2^{k-1}}$;
\item if $b'_k < B_k$, then $e_{k+1} < e_k$ and $A_{k+1}= A_k$.
% and 
%$$mult(f^{2^k}|[a_k',e_k])=mult(f^{2^{k+1}}|[a_{k+1}',e_{k+1}]).$$ 
%$f^{2^{k+1}}|[a_{k+1},e_{k+1}]$ has critical point at $0$ QQQ not sure about this, see proof QQQ;
\item if $e_{k+1} < b_{k+1}$, then $b'_{k+2} < b_{k+1}$ and $A_{k+2} = A_{k+1}$.
%and $$mult(f^{2^{k+1}}|[a_{k+1}',e_{k+1}])=mult(f^{2^{k+2}}|[a_{k+2}',e_{k+2}])+1.$$ 
\end{enumerate}
%The situations described in (4) and (5) are to deal with a situation which  numerically does not happen.  
\end{lemma}
\begin{proof}
Since $f^{2^k}[a_{k+1},b_{k+1}]\subset [0,b_k]$, we have 
$E_{k+1}=E_k \circ f_2\circ E_k|T_{k+1}$, where $E_k$ is orientation preserving and $f_2$ is orientation reversing.  Since the diffeomorphic range of $E_k$
is $[\hat A_k,\hat B_k]\supset [a_k,b_k]\ni 0$ and $E_k\circ f_2$ maps $(0,b_k]$ diffeomorphically onto $[a_k,c_{2^k})$,
it follows that  $B_{k+1}=\hat B_{k+1}=E_k\circ f_2(0)=c_{2^k}$ and $A_{k+1}\le \hat A_{k+1}\le a_k$. 
Taking $a_{k+1}'$ to be the point in $(a_k,a_{k+1})$ for which $f^{2^k}(a_{k+1}')=E_k\circ f_1(a_{k+1}')=0$ one has $f^{2^{k+1}}(a'_{k+1})=E_{k+1}\circ f_1(a'_{k+1})=
E_k\circ f_2 \circ E_k\circ f_1(a_{k+1}')=E_k\circ f_2 \circ E_k(0)= B_{k+1}$.

%As  $f^{2^k}([a_{k+1},0])\subset [b_{k+1},b_k] \subset (0,b_k)$, we have that $f^{2^{k+1}} |[a_{k+1},0)$  is a composition of a restriction of the diffeomorphisms 
%$f^{2^k} |[a_k ,0)$ and
%$f^{2^k} |(0,b_k ]$. Since the first map is orientation preserving and its range contains $(0,f^{2^k}(b_{k+1})]$ and the second map is orientation reversing, this implies  that
%$B_{k+1}= c_{2^k}$ and $a'_{k+1}\in [a_k , a_{k+1}]$ is a preimage of
%the critical point under the map  $f^{2^k} |[a_k ,0]$ and $f^{2^{k+1}}(a'_{k+1})=B_k$. 
%It also follows that 
%$E_{k+1}\colon T_{k+1}\to [A_{k+1},B_{k+1}]$ has no critical values in $[b_{k+1},B_{k+1})$.
Similarly,  since $f^{2^{k+1}}[a_{k+2},b_{k+2}]\subset [a_{k+1},0]$, $E_{k+2}=E_{k+1}\circ f_1\circ E_{k+1}|T_{k+2}$ where $E_{k+1}$ is orientation reversing and $f_1$ is orientation preserving. Since $a_k<a'_{k+1}<a_{k+1}<c_{2^{k+1}}=E_{k+1}(c_1)<0$, $E_{k+1}\circ f_1(a'_{k+1})=B_{k+1}=\hat B_{k+1}$
and since the diffeomorphic range of $E_{k+1}$ is $[\hat A_{k+1},\hat B_{k+1})\supset [a_k,c_{2^k}) \supset (a_{k+1}',0)$ it follows that $B_{k+2}=\hat B_{k+2}=B_{k+1}=\hat B_{k+1}=c_{2^k}$
and $\hat A_{k+2}=c_{2^{k+1}}$,  proving in particular statement (1). 

By definition $E_{k+2}\circ f_1(e_{k+2})=B_{k+2}$. Since $E_{k+1}\circ f_1(a'_{k+1})=B_{k+1}=B_{k+2}$ 
and  $E_{k+2}=E_{k+1}\circ f_1\circ E_{k+1}|T_{k+2}$
we have that  $E_{k+1}\circ f_1(e_{k+2})=a'_{k+1}$. Since $a'_{k+1}\in (a_k,a_{k+1})$,
$E_{k+1}\circ f_1(d_k)=E_k\circ f_2\circ E_k\circ f_1(d_k)=E_k\circ f_2 (b_k)=a_k$
and $E_{k+1}$ is orientation reversing, it follows that $e_{k+2}<d_k$, proving statement (2). 
%and $E_{k+1}\circ f_1|[a_{k+1}',0]$ is 

Statement (3) follows as in statement (1).

% The equality $B_{k+2}= c_{2^k}$ follows from the fact that $E_{k+1}\circ f_1 ([a_{k+2}, d_k ]) =
%f^{2^k}\circ (E_k \circ f_1 ) ([a_{k+2}, d_k ]) = f^{2^k}([f^{2^k}(a_{k+2}),b_k])=[a_k , a_{k+2}] \ni  a'_{k+1}$, that both $E_{k+1}\circ f_1| [a_{k+2}, d_k ]$ and
%$f^{2^{k+1}}|[a_k,a_{k+2}]$ are orientation reversing  while $f^{2^{k+1}}(a'_{k+1})=B_k$.   The same argument also implies
%$e_{k+2}< d_k$ and statement (3).

To prove statement (4), assume  $b'_k < B_k$. Then $E_k$ has range  $[A_k,B_k]\supset [A_k,b'_k]$. Note that the left endpoint of the domain of $E_k$
is $f_2(b'_k)$ and $E_k\circ f_2(b'_k)=A_k$. Since 
$E_{k+1}=E_k\circ f_2 \circ E_k$ it follows that the range of $E_{k+1}$ is equal to $[A_k,B_{k+1}]$ and so $A_{k+1}=A_k$.
Moreover, $A_k=A_{k+1}=E_{k+1}\circ f_1(e_{k+1})=E_k\circ f_2 \circ E_k \circ f_1(e_{k+1})$ and $E_k\circ f_2(b'_k)=A_k$.
Since $E_{k+1}$ and $f_1,f_2$ are all injective,  $b'_k= E_k \circ f_1(e_{k+1})$.  Therefore, and since $B_k=E_k\circ f_1(e_k)$ and
$f_1,E_k$ are increasing,  $b'_k<B_k$ implies that $e_{k+1}<e_k$. 

% $E_k \circ f_1$ maps $[a_{k+1} , e_k ]$ homeomorphically onto  $[b_{k+1}, B_k]\ni b'_k$ 
%and since $f^{2^k}=E_k\circ f_2$ maps $[b_{k+1},b'_k]$ homeomorphically onto $[A_k,b_{k+1}]$, statement (3) follows.

Finally, to prove statement (5), note that $E_{k+1}|[f(a_{k+1}),f(0))$ maps diffeomorphically onto $(c_{2^{k+1}}, b_{k+1}]$ and if 
  $e_{k+1} < b_{k+1}$ then this last interval contains $(c_{2^{k+1}},e_{k+1}]$. 
Since  $E_{k+1}\circ f_1$ maps the latter interval diffeomorphically onto $[A_{k+1},c_{2^{k+2}})$ and since
$E_{k+2}=E_{k+1}\circ f_1 \circ E_{k+1}|T_{k+2}$ it follows that 
$A_{k+2}= A_{k+1}$ and  $b'_{k+2} = f^{2^{k+1}} |[0,b_{k+1}](e_{k+1}) < b_{k+1}$.
\end{proof}

\subsection{A first recursive inequality \label{subsec:firstrecursive}}

\begin{lemma}\label{lem3} There exists $C>0$ so that for all $k$ even
 \begin{equation}d_k \le  C b_{k+1}^{\beta-1} b_k. \label{eq5} \end{equation}
\end{lemma}
\begin{proof}
For $k$ even,  $b_{k+1}$ is a repelling fixed point of $f^{2^k}$, so $|Df^{2^k}(b_{k+1})|>1$.
When $k$ is large this implies that 
$$b_{k+1}^{\beta-1}|DE_k(f(b_{k+1}))|\approx  |Df^{2^k}(b_{k+1})|>1.$$ Since  
 $|Df^{2^k}(a_{k+1})|\approx  |a_{k+1}|^{\alpha-1} |DE_k(f(a_{k+1}))|$ and $f(a_{k+1})=f(b_{k+1})$ it follows that
\begin{equation}
 Df^{2^k}(a_{k+1}) > C_. |a_{k+1}|^{\alpha-1} b_{k+1}^{1-\beta}\mbox{ and }|DE_k(f(b_{k+1}))|> C_. b_{k+1}^{1-\beta}.
\label{eq5a}\end{equation} 

Diffeomorphic branches of maps with negative Schwarzian derivative expand cross-ratios, see 
\cite[Chapter IV]{dMvS}.
 Applying this fact to the diffeomorphism 
$E_k\circ f_1 \colon [a_{k+1},e_k] \to [b_{k+1},B_k]$ and the four points
$a_{k+1},a_{k+1}^+,d_k,e_k$ (which map to $b_{k+1},b_{k+1}^+,b_k,B_k$) 
(where we take $a_{k+1}^+=a_{k+1}+h$ with $h>0$ close to $0$ and $b_{k+1}^+$ the image of this point)
we obtain the inequality 
$$\dfrac{(e_k-a_{k+1}^+)(d_k-a_{k+1})}{(a_{k+1}-a_{k+1}^+)(e_k-d_k)} \le
\dfrac{(B_k-b_{k+1}^+)(b_k-b_{k+1})}{(b_{k+1}-b_{k+1}^+)(B_k-b_k)} .
$$
Taking $h\downarrow 0$, we get
 \begin{equation}
 \begin{array}{rl}d_k & <d_k-a_{k+1}\le  \dfrac{(B_k-b_{k+1})}{(B_k-b_k)}(b_k-b_{k+1})\dfrac{(e_k-d_k)}{(e_k-a_{k+1})}  \dfrac{1}{Df^{2^{k}}(a_{k+1})} \\
&\le C b_{k+1}^{\beta-1} b_k.\end{array}
\label{eq5b} \end{equation}
Here we use that  the first factor in the long expression is bounded from above by Lemma~\ref{lem1},   the second by $b_k$,
the third factor by $1$ and in the final factor we use the bound from (\ref{eq5a}). 
\end{proof}
%More generally, if $\alpha\ge 1$ then we use the cross-ratio expansion of  the map $E_k$ and the four points
%$f_1(a_{k+1}),f_1(a_{k+1}^+),f_1(d_k),f_1(e_k)\in \hat J_k$ (which are mapped by $E_k|\hat J_k$ to $b_{k+1},b_{k+1}^+,b_k,B_k$). As before we get 
%$$\dfrac{(e_k^\alpha+|a_{k+1}|^\alpha)(d_k^\alpha+|a_{k+1}|^\alpha)}{(|a_{k+1}^+|^\alpha-|a_{k+1}|^\alpha)(e_k^\alpha-d_k^\alpha)} \le
%\dfrac{(B_k-b_{k+1})(b_k-b_{k+1})}{(b_{k+1}^+-b_{k+1})(B_k-b_k)}
%$$
%So, as before,  
%$$\begin{array}{rl}
% d_k^\alpha &<d_k^\alpha+|a_{k+1}|^\alpha \le  \dfrac{(B_k-b_{k+1})}{(B_k-b_k)}(b_k-b_{k+1})\dfrac{(e_k^\alpha-d_k^\alpha)}{(e_k^\alpha+|a_{k+1}|^\alpha)}  \dfrac{1}{DE_k(f(b_{k+1}))}\\
%&\le C b_{k+1}^{\beta-1} b_k.\end{array}$$
%%Here we use Lemma 1 to get an upper bound for the first factor; the third factor is bounded from above by $1$. 
%Hence
%\begin{equation} 
%d_k\le C b_{k+1}^{\frac{\beta-1}{\alpha}}b_k^{\frac{1}{\alpha}}. %\tag{5'}
%\label{eq5'}\end{equation}
%Note that when $\alpha=1$ this reduces to inequality (\ref{eq5}). 

%Let $A^2_k$ be so that $[A^2_k,B_k]$ is the homeomorphic image of a semi-extension of $f^{2^k}|[a_k,0)$
%with the property that its {\em order} is at most $2$ (i.e. the number of times iterates of the neighbourhood of 
%$c_1$ mapping onto $[A^2_k,B_k]$ hits the critical point
%is at most $2$). 

\begin{figure}[htp]
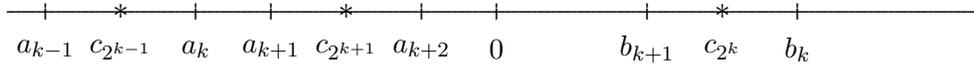
 \hfil
\beginpicture
\dimen0=0.5cm
\setcoordinatesystem units <\dimen0,\dimen0>  point at 0 0
\setplotarea x from -10 to 10, y from -1 to 1
\setlinear
\plot -13 0 13 0 / 
\put {\small $a_{k-1}$} at -12 -1
  \plot -12 -0.2 -12 0.2  / 
\put {\small $c_{2^{k-1}}$} at -10 -1
  \put {$*$} at -10 0 
\put {\small $a_k$} at -8 -1
  \plot -8 -0.2 -8 0.2  / 
 \put {\small $a_{k+1}$} at -6 -1 
  \plot -6 -0.2 -6 0.2  /  
  \put {\small $c_{2^{k+1}}$} at -4 -1 
  \put {$*$} at -4 0 
   \put {\small $a_{k+2}$} at -2 -1 
     \plot -2 -0.2 -2 0.2  /  
   \put {\small $0$} at -0 -1 
        \plot 0 -0.2 0 0.2  / 
 \put {\small $b_k$} at 8 -1
      \plot 8 -0.2 8 0.2  / 
 \put {\small $c_{2^k}$} at 6 -1 
   \put {$*$} at 6 0 
  \put {\small $b_{k+1}$} at 4 -1
       \plot 4 -0.2 4 0.2  /  

   \endpicture
\caption{\label{fig:order} The ordering of several dynamically relevant point; here $k$ is even.} 
\end{figure}

\subsection{Some dichotomies \label{subsec:dichotomy}}

\begin{lemma} \label{lem4} 
There exists a constant $C > 0$ such that for large even values of $k$,
$$|A_{k+2}| >   \min(C b_{k+1}, \frac{1}{2} |a_k |).$$
\end{lemma}
\begin{proof}
 %First, let us suppose that $d_k\le c_{2^k}$. 
Note that $E_{k+2}=E_{k+1}\circ f_1\circ E_{k+1}|T_{k+2}$ 
and that  $E_{k+1}$ maps  $\hat J_{k+1}\ni f(0)$ diffeomorphically onto $[\hat A_{k+1},\hat B_{k+1}]=
[c_{2^{k-1}},c_{2^k}] \supset [a_k,c_{2^k}]$.   

If $d_k\le c_{2^k}$ then the last interval contains $[a_k,d_k]$.
Moreover,  $E_k \circ f_1$ maps $[a_k,d_k]$ diffeomorphically  to $[a_k,b_k]\supset [0,b_k]$
and the latter interval is mapped diffeomorphically by $f^{2^k}$ to $[a_k,c_{2^k}]$. 
Since $E_{k+1}=f^{2^k}\circ E_k \circ f_1|T_{k+1}$,  it follows that $A_{k+2} \le a_k$ and since both numbers are negative we get $|A_{k+2}|\ge |a_k|$.

If  $d_k>c_{2^k}$ then the same consideration shows that $A_{k+2}=E_{k+1}\circ f_1(c_{2^k})$. If 
$|A_{k+2}|>\frac{1}{2}|a_k|$ or $|A_{k+2}|>\frac{1}{2}b_{k+1}$ there is nothing to prove.
So in the remainder of the proof of this lemma assume that $|A_{k+2}|\le \frac{1}{2}|a_k|$ and $|A_{k+2}|\le \frac{1}{2}b_{k+1}$.  
%Let $z=c_{2^k}$. Then $z\in [a_{k+2},d_k]$ and $E_{k+1}\circ f_1(z)=A_{k+2}$. 
The interval $[A_{k+2},a_{k+2}]$ is well-inside the interval
$[a_k,c_{2^k}]$ as $c_{2^k}>b_{k+1}>2|A_{k+2}|$ and $|a_k|\ge 2|A_{k+2}|$. Moreover,  $[\hat A_{k+1},\hat B_{k+1}]=[c_{2^{k-1}},c_{2^k}]$
 is the diffeomorphic range of $E_{k+1}|\hat J_{k+1}$, 
$[c_{2^{k-1}},c_{2^k}]\supset [a_k,c_{2^k}]$ and $[f(a_{k+2}),f_1(c_{2^k})]\subset \hat J_{k+1}$. 
So $[A_{k+2},a_{k+2}]=E_{k+1} [f(a_{k+2}),f_1(c_{2^k})]$ is well-inside the diffeomorphic range
of $E_{k+1}|\hat J_{k+1}$ and so the distortion of $E_{k+1}$ restricted to $[f(a_{k+2}),f_1(c_{2^k})]$ is bounded. 

It follows that  the distortion of $E_{k+1}\circ f_1 |_{[a_{k+2},c_{2^k}]}$
is bounded. Since the derivative of $f^{2^{k+1}}$ at its fixed point $a_{k+2}$ is larger than one, 
this implies that  $|D(E_{k+1}\circ f_1)(x)|>C_5$ for all $x\in [a_{k+2},c_{2^k}]$. Since 
$a_{k+2}<0<b_{k+1}<c_{2^k}$, $E_{k+1}$ is orientation reversing and $E_{k+1}\circ f_1(0)=c_{2^{k+1}}<0$, 
\begin{align*} &|A_{k+2}|=|E_{k+1}\circ f_1(c_{2^k})|>|E_{k+1}\circ f_1(b_{k+1})|>C_5 b_{k+1}.\qedhere \end{align*}
\end{proof}
%If $\alpha>1$ then,  since the derivative of $f^{2^{k+1}}$ at its fixed point $a_{k+2}$ is larger than one,
%$|DE_{k+1}(f(a_{k+2}))|\ge (1/|a_{k+2}|)^{\alpha-1}$ and since $E_{k+1}$ has bounded distortion on $ [f(a_{k+2}),f_1(c_{2^k})]$,
%it follows that $|DE_{k+1}(x)|>C_5 (1/|a_{k+2}|)^{\alpha-1}$ for all $x\in  [f(a_{k+2}),f_1(c_{2^k})]$ and as before and since
%$|f_1(b_{k+1})-f_1(0)|=b_{k+1}^\alpha$, 
%$$\begin{array}{rl}
%|A_{k+2}| &=|E_{k+1}\circ f_1(c_{2^k})|\ge |E_{k+1}\circ f_1(b_{k+1})|>C_5'  b_{k+1}^\alpha / |a_{k+2}|^{\alpha-1} = \\ 
%& \\ 
% &  = C_5' b_{k+1} \cdot  (b_{k+1}/|a_{k+2}| )^{\alpha-1} >  C_5' b_{k+1}  \cdot (b_{k+1}/|a_{k+1}| )^{\alpha-1} > C_5' b_{k+1} . 
%\end{array}$$
%\end{proof} 

%In the next lemma we show that either for infinitely many $k$
%the interval $[A_k,B_k]$ is a scaled neighbourhood of $[a_k,b_k]$ or some
%inequality holds which will allow us to conclude  property holds which as we will show in the proof of Lemma~\ref{lem7}  will imply $A_k=A_{k+1}=A_{k+2}=\dots$. 

\begin{lemma} \label{lem5} 
There exists $C > 0$ such that the following holds. Let $k$ be a sufficiently large even integer. Then either
\begin{itemize}
\item  $|A_k | > C b_{k+1}$ or
\item $e_k < b_{k+1}$. 
\end{itemize}
\end{lemma}
\begin{proof}
Suppose $e_k\ge b_{k+1}$. Then due to Lemma~\ref{lem2}(2) and  inequality  (\ref{eq5}) from Lemma~\ref{lem3},    we know that for $k$ large and even, 
\begin{equation}
b_{k+1}\le e_k < d_{k-2} < C_4 b_{k-1}^{\beta-1}b_{k-2}<b_{k-2}^\beta.
\label{eqlem5}
\end{equation}
From Lemma~\ref{lem4} we know that either $|A_k|>Cb_{k-1}$ or $|A_k|>\frac{1}{2} |a_{k-2}|$. 
In the first case we have nothing to do because $b_{k+1}<b_{k-1}$.
In the second case it follows from (\ref{eqlem5})  that $|A_k|>\frac{1}{2} |a_{k-2}|> C b_{k-2}^\beta>C_6 b_{k+1}$.
\end{proof}

\subsection{Conditional first universal bounds \label{subsec:condbound}}

\begin{lemma}\label{lem6}  For any $C > 0$ there exist $0 < \lambda_1 < \lambda_2 < 1$ such that the following
holds. Let $k$ be large even integer and $|A_k | > Cb_{k+1}$. Then
\begin{align}
 |Df^{2^k} |[b_{k+1},b_k ] |  &> \lambda_1 \,\,\, ,  \label{eq7} \\
 \lambda_1 b_k  < b_{k+1}  & < \lambda_2 b_k . \label{eq8}
\end{align} 
\end{lemma}
\begin{proof} Consider two cases.

Case 1: $|a_k | < \frac12 Cb_{k+1}$.  Then 
$|b_{k+1} - a_k | < (1 + \frac12 C)b_{k+1}$. At the same time
$|A_k - a_k | > \frac{1}{2} C b_{k+1}$ and we see that $|A_k - a_k | > C_7 |b_{k+1} - a_k |$ for some $C_7 > 0$
which depends only on $C$.

Case 2: $|a_k | \ge \frac{1}{2} Cb_{k+1}$. Then $|b_{k+1} -a_k | \le  (1+ \frac{2}{C} )a_k$. According to Lemma~\ref{lem1}, 
$|A_k | > K|a_k |$ for some universal $K > 1$, therefore $|A_k - a_k | > (K - 1)|a_k |$ and
we again get $|A_k - a_k | > C_8 |b_{k+1} - a_k |$ for some $C_8 > 0$ which depends only on
$C$ and $K$.

From this and Lemma~\ref{lem1},  we get 
that the range of the map $E_k \colon [f (b_{k+1}), f (b_k )] \to [a_k , b_{k+1}]$ can be diffeomorphically 
semi-extended to a $C_9$-scaled neighbourhood
of the interval $[a_k , b_{k+1}]$, and therefore the distortion of the map $E_k |[f (b_{k+1}),f (b_k )]$
is bounded.

On the  interval $[b_{k+1}, b_k ]$ the absolute  value of $Df$  is increasing, hence 
$$|Df^{2^k} (x)| = |DE_k (f (x))| |Df (x)| > C_{10}  |Df^{2^k} (b_{k+1})|$$ 
 for all $x \in [b_{k+1} , b_k ]$ and some constant $C_{10} > 0$ which depends only on $C$.  Since
$b_{k+1}$ is a repelling fixed point of $f^{2^k}$, we get
$|Df^{2^k} (b_{k+1})| > 1$ and  $|Df^{2^k}|>C_{10}$ on $[b_{k+1} ,b_k ]$. This implies the existence of $\lambda_1>0$
as in  equations (\ref{eq7}) and (\ref{eq8}). 

To prove the existence of $\lambda_2<1$ in (\ref{eq8}), note that
by  Lemma~\ref{lem1} and Koebe that $E_k$ has bounded distortion on the range
$[b_k/2,b_k]$. Moreover, $f_2$ has bounded distortion on $[b_k/2,b_k]$.
By contradiction assume that $b_{k+1}/b_k\approx 1$.  Then  there exists a point $x\in [b_{k+1},b_k]$ for which $(E_k\circ f_2)(x)\in [b_k/2,b_{k+1}]$
and $|D(E_k\circ f_2)(x)|$ is large.  But since $(E_k\circ f_2)(y)\in [b_{k+1},b_k]$ for all $y\in  [b_{k}/2,b_{k+1}]$,
it follows that $|D(E_k\circ f_2)(y)|$ is also large  for all such $y$.  But this contradicts that 
$(E_k\circ f_2)$ maps $[b_k/2,b_{k+1}]$ into $[b_{k+1},b_k]$.
Thus the existence of $\lambda_2<1$ follows. 
\end{proof}

\subsection{Getting space some of the time}\label{subsec:somespace}

Now we are ready to combine the results from the previous subsection.

\begin{lemma} \label{lem7} There exists a constant $C > 0$ and an infinite sequence of even
integers $k_1 < k_2 < \dots$ such that
$$|A_{k_i} | > Cb_{k_i} ,$$
and therefore, the distortion of the maps $E_{k_i} |J_{k_i}$ is universally bounded.
\end{lemma}
\begin{proof} It follows from Lemma~\ref{lem5} that either there exist infinitely many even
integers $k_i$ such that $|A_{k_i} | > Cb_{k_i +1}$ or there exists an even integer $k_0$ such that
$e_k < b_{k+1}$ for all even $k \ge k_0$.

In the first case we are done because of Lemmas~\ref{lem1} and \ref{lem6}, so suppose that we are in the second case. Since $0 < e_{k+1} \le  e_k$, Lemma~\ref{lem2}(5)
implies $b'_{k+2}< b_{k+1}$ and $A_{k+2}= A_{k+1}$ for all even $k \ge  k_0$ . Notice that
$b_{k+1} < c_{2^k}=B_{k+2}$, and  therefore Lemma~\ref{lem2}(1) gives $b'_{k+2} < B_{k+2}$ . Then from Lemma~\ref{lem2}(4)
it follows that  $A_{k+3} = A_{k+2}$ . So, we see that  $A_k = A_{k_0 +1}$ for all $k > k_0$ and
since $b_k \to 0$ we get $|A_k | > b_k$ for all $k$ large enough.

The boundedness of the distortion of the maps $E_{k_i} |J_{k_i}$ follows from Lemma~\ref{lem1}
and  from $|A_{k_i} | > Cb_{k_i} $. 
\end{proof}

\subsection{Space for some $k$ gives improved  space for the next $k$\label{subsec:improvingspace}}

\begin{lemma}\label{lem8}  For every constant $C > 0$ there exists a constant $\tau_* > 0$ such that the following holds. Let $k$ be a large even integer and $|A_k| > Cb_k$. 
Then
\begin{align}
b_{k+2}  &< \tau_* b_k^{2-1/\beta} , \label{eq9} \\
b_k - c_{2^k}  &< \tau_* b_k^\beta \quad \,\,\,  , \label{eq10} \\
d_k & <  \tau_* b_k^{2\beta-1} \,\, . \label{eq11}
\end{align}
\end{lemma}
\begin{proof} Due to Lemma~\ref{lem1} we always have some space to the right of the renormalization interval, and since 
we assumed that $|A_k| > Cb_k$, therefore the distortion of the map $E_k|J_k$ 
is bounded by a constant depending only on $C$. The map $E_{k+1}|J_{k+1}$ can be decomposed as
$E_{k+1}|J_{k+1} = E_k|J_k \circ f|[b_{k+1},b_k] \circ E_k|J_{k+1}$.
Due to Lemma~\ref{lem6} we know that $b_{k+1} > \lambda_1b_k$, and hence, the distortion of the map $f|[b_{k+1},b_k]$ is bounded. Thus, the distortion of 
$E_{k+1}|J_{k+1}$ is bounded as a composition of three maps of bounded distortion. Then the distortion of the map 
$f^{2^{k+1}} |[a_{k+1},0]$ is bounded again. Combining this with $f^{2^{k+1}}(a_{k+1}) = b_{k+1}$
and $f^{2^{k+1}}(0) = c_{2^{k+1}} \in  [a_{k+1},a_{k+2}]$ we get   
\begin{equation} 
Df^{2^{k+1}} |[a_{k+1},0]  > C_{11}b_{k+1}/|a_{k+1}|. 
\label{eq6} \end{equation} This implies the following estimate on the position of $a_{k+2}$ and, therefore, of
$b_{k+2}$:
\begin{equation}
\begin{array}{rllrll}
|a_{k+2}| & < &  \frac{|a_{k+1}|^2}{C_{11}b_{k+1}} & < & C_{12}b_k^{2\beta-1}, \\ 
 |b_{k+2}| & & < & &   C_{13}b_k^{2-1/\beta} ,
\end{array}
\end{equation}
for some universal constants $C_{12} > 0$ and $C_{13} > 0$.

Since $k$ is even we know that $c_{2^k} \in  [b_{k+1},b_k]$ and $c_{2^{k+1}} \in  [a_{k+1},a_{k+2}]$ and so in particular 
$f^{2^k}[c_{2^k},b_k]\subset [a_k,0]$. Due
to Lemma~\ref{lem6} the derivative of $f^{2^k}|[b_{k+1},b_k]$ is bounded away from zero, hence
\begin{equation}|b_k - c_{2^k}|< \lambda_1^{-1} |a_k|<C_{14} b_k^\beta \ll  b_k \end{equation}
for some universal constant $C_{14}$. Combining this with equation (\ref{eq6}), and since
$f^{2^k}[0,d_k]=[c_{2^k},b_k]$,   this  gives us a much better estimate for $d_k$
(compared to inequality (5)): 
\begin{equation}d_k < C_{11}^{-1}|b_k - c_{2^k}| \cdot |a_{k+1}| / b_{k+1}  < C_{15}b_k^\beta |a_{k+1}| / b_{k+1}   < C_{15} b_k ^{2\beta-1}\end{equation}
for some  $C_{15} > 0$.
\end{proof}

\begin{lemma}\label{lem9} For every constant $C_0 > 0$ there exists a constant $\tau_* > 0$ such that the following holds. Let $k$ be a large even integer, $C$ be a constant greater that $C_0$, and $|A_k| > Cb_k$, $B_k > (1 + C)b_k$. Then
$$|A_{k+2}| > \tau_* \min(C, b_k^{1-\beta})b_k.$$
\end{lemma}
\begin{proof} 
 Set
 \begin{equation}
\begin{array}{rcl}
\tilde A_k &=& -\frac{1}{2}Cb_k \\
\tilde  B_k &=& (1+\frac{1}{2}C) b_k.
\end{array}\end{equation}
  Let $\tilde e_k, \tilde b_k$ be points such that $E_k \circ f_1(\tilde e_k) = \tilde B_k$ and 
$E_k \circ f_2(\tilde b_k) = \tilde A_k$. Arguing as before we see that the distortions of maps $E_k \circ  f_1|[a_k,\tilde e_k]$ 
and $E_k \circ f_2|[b_{k+1},\tilde b_k]$ are bounded by some constant depending on $C_0$. Therefore, for all $x\in  [a_k,\tilde e_k]$, 
 \begin{equation}
\begin{array}{rcl}
D(E_k \circ f_1)(x) & >&  C \dfrac{b_k-a_k}{ d_k - a_k} \\
&>& C_{17} b_k^{1-\beta}.
\end{array}\end{equation}
In the same way we get the estimate on the derivative of the other branch:
$$D(E_k \circ  f_2)(x) > C_{18} $$
for all $ x\in [b_{k+1},\tilde b_k]$.
Now consider the following cases. 

{\bf Case 1.a.}  Assume that $\tilde e_k < b_{k+1}$ and $\tilde B_k > \tilde b_k$. Then, arguing as in Lemma~\ref{lem2}(4,5)
we obtain that $|A_{k+2}| > |\tilde A_k|$ and we are done in this case.

{\bf Case 1.b.} Now suppose $\tilde e_k < b_{k+1}$ and $\tilde B_k \le  \tilde b_k$. Then
 \begin{equation}
\begin{array}{rcl}
 |E_{k+1} \circ  f_1([d_k, \tilde e_k])| &>& C_{18}|\tilde B_k - b_k|\\ 
&=& \dfrac{1}{2}C_{18}Cb_k.
\end{array}\end{equation}
Using an argument similar to prove Lemma 2(4) we get $|A_{k+2}| > \frac{1}{2} C_{18}C b_k$ and
this case is also done.

{\bf Case 2:} $\tilde e_k > b_{k+1}$. From the derivative estimate we know
 \begin{equation}
\begin{array}{rcl}
E_k \circ  f_1([d_k, b_{k+1}]) &>& C_{17}b_k^{1-\beta}|b_{k+1} - d_k| \\
&>& C_{19} b_k^{2-\beta} .
\end{array}\end{equation}
Here we used inequalities (\ref{eq8}) and (\ref{eq11}).

We finish by considering two subcases as in Case 1. If $E_k \circ f_1(b_{k+1}) > \tilde b_k$,
then as before $|A_{k+2}| > |\tilde A_k|$. Otherwise,
   \begin{align*}&|A_{k+2}| > C_{18}C_{19}b_k^{2-\beta}. \qquad \qquad \qedhere
      \end{align*}
\end{proof}

\subsection{The proof of the first part of Theorem~\ref{thm:bounds}: getting huge space all the time}\label{subsec:superexpspace}

The following lemma completes the proof of  the first part of the {\lq}Big Bounds{\rq} Theorem~\ref{thm:bounds}.
The actual bounds for the space that are claimed in that theorem 
will be only obtained in the improved bounds from  Lemma~\ref{lem10'}.

\begin{lemma}[Koebe Space for the semi-extension] \label{lem10}  There exists $\hat \lambda>0$ so that 
as $k$ even  and $k\to \infty$,
 \begin{equation}
\dfrac{|b_{k+2}-a_{k+2}|}{|a_{k+2} -A_{k+2}|}  = O(b_{k}^{1-1/\beta}), \dfrac{|b_{k+2}-a_{k+2}|}{|B_{k+2} -b_{k+2}|}  = O(b_{k}^{1-1/\beta}) 
\label{eq:22}\end{equation}
and 
\begin{equation}\dfrac{|b_{k+1}-a_{k+1}|}{|a_{k+1} -A_{k+1}|} = O(b_{k-2}^{1-1/\beta}) , \dfrac{|b_{k+1}-a_{k+1}|}{ |B_{k+1} - b_{k+1}|}  \ge \hat \lambda .
\label{eq:23}
\end{equation}
In particular, 
the range of the map $E_k|J_k$ can be monotonically semi-extended to a $\tau_k$ scaled neighbourhood of
 $[a_k,b_k]$ where $\tau_k \approx O(b_{k-2}^{1-1/\beta})$ for $k$ even 
 and $\tau_k \approx 1$ for $k$ odd. 
 
 Moreover, $O(b_{k}^{1-1/\beta})$ converges super-exponentially to zero: 
$\log(b_{k})$ converges exponentially to zero. 
%\end{array}\label{eq:22}\end{equation}
\end{lemma}

\begin{proof} This lemma is a consequence of the previous two lemmas. Let $k$ be a large (even) 
integer from the sequence given by Lemma~\ref{lem7}. Then, from Lemmas~\ref{lem8}  and \ref{lem9}  it follows that
 \begin{equation}
\begin{array}{rcl}
 |A_{k+2}| &>& C_{20}b_k^{\frac{1}{\beta}-1} b_{k+2}, \\
|B_{k+2}| &>& C_{20}b_k^{\frac{1}{\beta}-1} b_{k+2},
\end{array}
\label{eq:24}
\end{equation}
for some universal constant $C_{20} > 0$. Since $\beta > 1$ we see that if k is large 
enough, we get huge improvement on the relative size of extension interval $[A_{k+2},B_{k+2}]$ compared to the renormalization interval $[a_{k+2},b_{k+2}]$. 
From this point the argument can be applied inductively and  (\ref{eq:22}) follows. 

%Koebe  (\ref{eq:22}) gives  
%$$\dfrac{|x-f_2(a_{k})|}{|f_2(a_{k})-f_2(b_{k}')|}   = O(b_{k-2}^{1-1/\beta}),\dfrac{|b_{k+2}-a_{k+2}|}{|a_{k+2} -A_{k+2}|}  = O(b_{k} QQQ^{1-1/\beta}),$$
%where $x$ is so that $E_{k}(x)=b_{k}$, see Figure~\ref{fig:Jk}. Moreover, since $c_{2^k}\sim b_k$ we have
%$|x-f_2(a_{k}|\approx |f_2(a_k)|$. 

Lemma~\ref{lem6} gives  $|a_{k+1}-b_{k+1}|\approx |a_{k}-b_{k}|$. 
By the proof of Lemma~\ref{lem2}(4) either $A_{k+1}=A_k$ (if $b_k'<B_k$) or $A_{k+1}=E_k\circ  f_2(B_k)$
(if $B_k\le b_k'$). In the former case we use (\ref{eq:22}) and get $\dfrac{|b_{k+1}-a_{k+1}|}{|a_{k+1} -A_{k+1}|}\approx \dfrac{|b_{k}-a_{k}|}{|a_{k} -A_{k}|} =O(b_{k-2}^{1-1/\beta})$.  So let us check what happens when $B_k\le b_k'$. Using (\ref{eq:24}) we obtain
(*) $\dfrac{|f(0)-f_2(b_k)|}{|f(0)-f(B_k)|}\approx b_{k}^\beta / B_{k}^\beta = O(b_{k-2}^{\beta-1})$. 
On the other hand, the expression in (\ref{eq:22}) and Koebe imply 
$\dfrac{|x-f_2(b_{k})|}{|f_2(b_{k})-f_2(b_{k}')|}   = O(b_{k-2}^{1-1/\beta})$
where $x$ is so that $E_{k}(x)=b_{k}$, see Figure~\ref{fig:Jk}. Since $c_{2^k}\sim b_k$ we have
$|x-f(a_{k}|\approx |f(a_k)-f(0)|$ this implies (**) $\dfrac{|f(0)-f_2(b_{k})|}{|f_2(0)-f_2(b_{k}')|}   = O(b_{k-2}^{1-1/\beta})$. 
Since $b_{k-2}^{1-1/\beta} >> b_{k-2}^{\beta-1}$ and comparing (*) and (**) we can conclude that 
either $B_k>b_k'$ or (by Koebe) $E_k\circ f_2(B_k)|\ge (1/2)|A_k|$. In either case (\ref{eq:23}) holds. 

Since $B_{k+1}=c_{2^k} \sim b_k$, we have by  (\ref{eq8}) that there exist universal constants $0<\lambda_1'<\lambda_2'<1$ so that 
$\dfrac{|b_{k+1}-a_{k+1}|}{ |B_{k+1} - b_{k+1}|} \sim \dfrac{|b_{k+1}|}{ |b_k - b_{k+1}|}\in (\lambda_1',\lambda_2')$.
Which proves the second expression in (\ref{eq:23}) and that this expression cannot be improved. 

The final statement follows from inequality (\ref{eq9}). 
\end{proof}

\section{Scaling laws, renormalization limits and universality} \label{sec:scaling} 
%\section{Proof of Theorems~\ref{thm:bounds}-\ref{thm:nouniversality}: scaling laws, renormalization limits and universality} \label{sec:scaling} 

\paragraph{A first error bound for the map $f^{2^k}$ on $[a_k,b_k]$ when $k$ is even.}  Let $k$ be even and $x_k$ be so that $E_k(x_k)=b_k$, see Figure~\ref{fig:Jk}. Then  $E_k\colon [f(a_k),x_k]\to [a_k,b_k]$ is the first entry map 
and  $\tau_k$ be the Koebe space of $E_k|[f(a_k),x_k]$. Let $L_k$ be the affine map which agrees with $E_k$ on the boundary points of 
 $ [f(a_k),f(0)]$. 
By the Corollary of  Koebe, Lemma~\ref{cor:koebe}, we obtain for all $x\in  [f(a_k),f(0)]$
\begin{equation} 
E_k(x)  =L_kx+O(b_k/\tau_k)\mbox{ and }  DE_k(x)=DL_k(1+O(1/\tau_k)).\label{eq:distortionE_k} \end{equation} 
%\begin{equation} DE_k(x)=DL_k+O(DL/\tau_k)\mbox{ for all }x\in  [f(a_k),x_k]\label{eq:distortionE_k} \end{equation} 
By Lemma~\ref{lem10}
 $\tau_k\approx b_{k-2}^{1/\beta-1}\to \infty$.  
In particular it follows that $O(b_k/\tau_k)=o(b_k)$.  Obviously $DL_k\approx b_k/|a_k|\approx b_k^{1-\beta}$.
Hence
\begin{equation} E_k(x)= L_kx+o(b_k) \mbox{ and }
DE_k(x) \sim DL_k, \label{eq:distortionE_k2} \end{equation} 
for all $x\in  [f(a_k),f(0)]$.   Later on,    we will improve the error bound in this  expression. 
Hence
  \begin{equation}
f^{2^k} (x) = 
\begin{cases}
c_{2^k} - s_k |x| +o(b_k) & \mbox{ when }x\in [a_k,0],\\ 
c_{2^k} - t_k x^\beta +o(b_k) & \mbox{ when }x\in [0,b_k], 
\end{cases}
\label{eq:f2k}
\end{equation}
where $s_k>0$ is so that $c_{2^k}-s_k|a_k| + o(|b_k|)=-|a_k|$ and $t_k>0$ is so that $c_{2^k}-t_k b_k^{\beta}+o(|b_k|)=-|a_k|$.
By (\ref{eq10}) we have  $c_{2^k}=b_k+O(b_k^\beta)\sim b_k$ and since $a_k\sim -K_0 b_k^{\beta}$, this implies 
\begin{equation}s_k\sim  \dfrac{b_k^{1-\beta}}{K_0}\mbox{ and }
t_k\sim b_k^{1-\beta}.\label{eq:sk} \end{equation}
Equation (\ref{eq:distortionE_k2}) also gives
  \begin{equation}
Df^{2^k} (x) \sim 
\begin{cases} s_k&\mbox{ when } x\in [a_k,0),\\ 
-t_k \beta x^{\beta-1}  & \mbox{ when }x\in (0,b_k].
\end{cases}
\label{eq:f2kder} 
\end{equation} 
For simplicity we will write 
$$f_{l,k}:=f^{2^k}|[a_k,0]\mbox{ and }f_{r,k}:=f^{2^k}|[0,b_k].$$
To avoid an overload of notation we usually write
$$f_l=f_{l,k} \mbox{ and } f_r = f_{r,k}$$
if it clear from the context which $k$ is used.

\paragraph{The scaling law from $b_k$ to $b_{k+1}$ when $k$ is even.} 
Write $b_{k+1}=\lambda_{k} b_k$. 
Then   (\ref{eq:f2k}) implies  
\begin{equation}
c_{2^k}-t_k \lambda_{k}^\beta b_k^\beta +o(b_k) =f^{2^k}(b_{k+1}) = b_{k+1}=\lambda_k b_k . 
\label{eqabc} 
\end{equation}
By (\ref{eq10}) $$c_{2^k}=b_k+O(b_k^\beta)$$ 
and combining this with (\ref{eq:sk}) and (\ref{eqabc}) implies 
$$1-\lambda_{k}^\beta+o(1)=\lambda_k .$$
So taking $\lambda\in (0,1)$ be the root of $1-\lambda^\beta=\lambda$ this gives
$\lambda_k=\lambda+o(1)$ and 
$$b_{k+1}=\lambda b_k + o(b_k).$$ 
Later on we will improve on this statement, see  (\ref{eq:lkimp}).

\paragraph{The approximate scaling law from $b_k$ to $b_{k+2}$ when $k$ is even.}
Fix some $\delta>0$ and let $C_k$ be so that   $c_{2^{k+1}}= - C_k b_k^{\delta}$.
Below we will determine $\delta$ and $C_k$. 
Note that 
$$a_{k+1}<c_{2^{k+1}}<0<c_{2^{k+2}}<b_{k+2}<b_{k+1}<c_{3\cdot 2^k}<c_{2^k}<b_k.$$ 
Then using (\ref{eq:sk}) and  (\ref{eq:f2kder}) 
\begin{equation} 
c_{2^k}-c_{3\cdot 2^{k}}=f^{2^k}(0)-f^{2^k}(c_{2^{k+1}})=f_l(0)- f_l(c_{2^{k+1}})\sim \dfrac{C_k}{K_0} b_k^{\delta} b_k^{1-\beta}.
\label{k3k} \end{equation} 
Since $f_r$ has bounded distortion and bounded derivative on $[b_{k+1},b_k]$ this implies 
\begin{equation}
c_{2^{k+2}}-c_{2^{k+1}} = f_r\circ f_l (c_{2^{k+1}}) - f_r(c_{2^k})  =
f_r(c_{3\cdot 2^k})-f_r(c_{2^k}) \approx C_k b_k^{\delta} b_k^{1-\beta}
.\label{eq:33} \end{equation} 
In fact, 
\begin{equation} |c_{2^k}-c_{3\cdot 2^k}| \approx   |c_{2^{k+2}}-c_{2^{k+1}}| <   |b_{k+2}-a_{k+1}|  <o(b_k)\label{eq:c32k} \end{equation} 
where $\approx$ follows from the fact that $Df_r$ is bounded from above and below on $[b_{k+1},b_k]$, where the first 
$<$ follows from the ordering of the points and where $<o(b_k)$ follows from equation (\ref{eq9}) and $|a_{k+1}|\approx b_{k+1}^\beta$. 
Combining this with 
%In fact,   if $\delta>\beta$ then $|c_{2^k}-c_{3\cdot 2^{k}}|\approx o(b_k)$ and since 
$c_{2^k}\sim b_k$,  equations (\ref{eq:f2kder}) and  (\ref{eq:sk}) give 
 $f_r'(b_k)\sim -\beta$ and $f_r'(x)\sim -\beta$ for all $x\in [c_{3\cdot 2^{k}},c_{2^k}]$. 
 Hence (\ref{eq:33}) in fact improves to 
 %if $\delta>\beta$
%the above expression improves to 
\begin{equation} c_{2^{k+2}}-c_{2^{k+1}}  \sim \dfrac{\beta C_k}{K_0} b_k^{\delta} b_k^{1-\beta} .\label{eq:33b} 
\end{equation} 
Since $|c_{2^{k+1}}|=C_kb_k^\delta <<  \dfrac{\beta C_k}{K_0} b_k^{\delta} b_k^{1-\beta}$ and using that $b_{k+2}\sim c_{2^{k+2}}$, 
equation (\ref{eq:33b}) gives 
%\begin{equation} 
%b_{k+2}\sim c_{2^{k+2}}\approx b_k^{\delta} b_k^{1-\beta} \mbox{ and }
%a_{k+2} \approx [b_k^{\delta} b_k^{1-\beta}]^\beta.\label{eq:deltabeta} \end{equation}
%Moreover,  if $\delta>\beta$ we obtain the improved estimates
\begin{equation} 
b_{k+2}\sim c_{2^{k+2}}\sim \dfrac{\beta C_k}{K_0} b_k^{\delta} b_k^{1-\beta}  \mbox{ and }
a_{k+2} \sim -K_0 [\dfrac{\beta C_k}{K_0} b_k^{\delta}
b_k^{1-\beta}]^\beta.
\label{eq:deltabeta'} 
\end{equation}

Next note that $f^{2^{k+1}}(a_{k+2})=f_r\circ f_l(a_{k+2})$. 
Using that $f_l|[a_k,0]$ 
%and $f_r|[b_{k+1},b_k]$ 
has derivative everywhere $\sim \frac{1}{K_0} b_k^{1-\beta}$ and
% resp. $\approx 1$, 
equation (\ref{eq9}) we have that $|a_{k+2}|\le K_0
|b_{k+2}|^\beta < C |b_k|^{2\beta-1}$ and therefore 
equation (\ref{eq:deltabeta'}) implies 
$$f_l(a_{k+2})-f_l(0) \le  C b_k^{2\beta-1} b_k ^{1-\beta}= C
b_k^\beta.$$
Therefore $f_l(a_{k+2})\sim b_k$ and so  equation  (\ref{eq:f2kder}) implies
%(\ref{eq:f2k}) and (\ref{eq:sk}) we have that  
\begin{equation} f_r'(x)\sim -\beta\mbox{ for all }x\in [f_l(a_{k+2}),b_k].
\label{eq:derboundary}
\end{equation}
%$$f_l(a_{k+2})-f_l(0) \approx  b_k^{1-\beta} [b_k^{\delta} b_k^{1-\beta}]^\beta  \approx 
%b_k ^{\beta \delta+1-\beta^2}$$
%and
%\begin{equation} 
%f^{2^{k+1}}(a_{k+2})-c_{2^{k+1}}=f_r\circ f_l(a_{k+2})-f_r(f_l(0))\approx  b_k ^{\beta \delta+1-\beta^2}.\label{eq:compare} \end{equation} 
%Again, if  $\delta>\beta$ then because of  (\ref{eq:deltabeta'}) the latter expressions improve to 
%
%which implies 
%$$f_l(a_{k+2})-f_l(0) \sim   b_k^{1-\beta} [b_k^{\delta}
%b_k^{1-\beta}]^\beta = b_k ^{\beta \delta+1-\beta^2}.$$
Since, by (\ref{eq:deltabeta'}),  
$$f_l(a_{k+2})-f_l(0) \sim  \dfrac{b_k^{1-\beta}}{K_0} K_0 [\dfrac{\beta C_k}{K_0} b_k^{\delta} b_k^{1-\beta} ]^\beta = 
\left[ \dfrac{\beta C_k}{K_0}\right]^{\beta}
b_k ^{\beta \delta+1-\beta^2}.$$ Hence (\ref{eq:derboundary}) implies
%Moreover, equation  (\ref{eq:f2kder}) implies
%(\ref{eq:f2k}) and (\ref{eq:sk}) we have that  
%\begin{equation} f_r'(x)\sim -\beta\mbox{ for all }x\in [f_l(a_{k+2}),c_{2^k}].
%\label{eq:derboundary}
%\end{equation} 
%Therefore 
\begin{equation}
f^{2^{k+1}}(a_{k+2})-c_{2^{k+1}}=f_r\circ f_l(a_{k+2})-f_r(f_l(0))\sim  \beta \left[ \dfrac{\beta C_k}{K_0}\right] ^{\beta} 
b_k ^{\beta \delta+1-\beta^2}.\label{eq:compare'} \end{equation} 
By (\ref{eq:deltabeta'}),
 $f^{2^{k+1}}(a_{k+2})=a_{k+2}\approx  -C_k^\beta [b_k^{\delta} b_k^{1-\beta}]^\beta = -C_k^\beta b_k^{\beta\delta+\beta-\beta^2}$
is orders smaller than the right hand side of (\ref{eq:compare'}),  and thus it follows that
%\begin{equation} c_{2^{k+1}}\approx - b_k ^{\beta \delta+1-\beta^2}. \label{eq:orders} \end{equation}
%Moreover, provided $\delta>\beta$ we obtain from (\ref{eq:compare'}) 
\begin{equation} c_{2^{k+1}} \sim -
\beta \left[\dfrac{\beta C_k}{K_0}\right]^{\beta} b_k ^{\beta \delta+1-\beta^2}. \label{eq:orders'} 
\end{equation} 
Using $c_{2^{k+1}}= - C_k b_k^{\delta}$ we obtain as a natural choice
\begin{equation} \delta=\beta\delta+1-\beta^2\mbox{  which gives } \delta=\beta+1 \label{eq:delta}\end{equation} 
and 
%Therefore we can apply the previous improved bounds. 
%Combining $c_{2^{k+1}}= - C_k b_k^{\delta}$ and (\ref{eq:orders'}) we obtain 
 \begin{equation} C_k\sim \beta \left[ \dfrac{\beta C_k}{K_0}\right] ^{\beta} \mbox{ and therefore }  
 C_k\sim \left[ \dfrac {K_0^\beta} {\beta^{\beta+1}}  \right] ^{1/(\beta-1)}  .\label{eq:Ck} \end{equation} 
Hence from (\ref{eq:deltabeta'}), $b_{k+2}\sim c_{2^{k+2}}$ and $c_{2^{k+1}}= - C_k b_k^{\delta}$ we obtain 
\begin{equation} 
b_{k+2} \sim \dfrac{\beta}{K_0} \left[ \dfrac {K_0^\beta} {\beta^{\beta+1}}  \right] ^{1/(\beta-1)} 
 b_k^2 =  \beta^{-2/(\beta-1)}K_0^{1/(\beta-1)} b_k^2 
 \label{bkquadratic} \end{equation} 
and 
\begin{equation} c_{2^{k+1}}\sim  -\left[ \dfrac {K_0^\beta} {\beta^{\beta+1}}  \right] ^{1/(\beta-1)} b_k^{\beta+1}.  
\end{equation} 
 Since $b_{k+1}\sim \lambda b_k$ this gives 
\begin{equation} 
b_{k+2} \sim 
%\dfrac{\beta}{K_0} \left[ \dfrac {K_0^\beta} {\beta^{\beta+1}}  \right] ^{1/(\beta-1)} \lambda^{-2}   b_{k+1}^2  =
 \beta^{\frac{-2}{\beta-1}} K_0^{\frac{1}{\beta-1}}  \lambda^{-2} 
 b_{k+1}^2 \label{eq:recurrenceeq} 
 \end{equation}
 and 
 \begin{equation} 
c_{2^{k+1}}\sim 
 -\beta^{-\frac{\beta+1}{\beta-1}} K_0^{\frac{\beta}{\beta-1}}  \lambda^{-\beta-1} 
 b_{k+1}^{\beta+1} 
 \label{bkquadratic''}\end{equation}

\paragraph{The usual Koebe space does not hold and the proof of Theorem~\ref{thm:nousualbound}}

Let $T\ni f(0)$ be the maximal interval on which $f^{2^k -1}|T$  is diffeomorphic.
Then by Lemma~\ref{lem2} we have that 
$f^{2^k-1}=[\hat A_k,\hat B_k]\supset [a_k,b_k]$  where 
$$\hat A_k=c_{2^{k-1}}, \hat B_k=c_{2^{k-2}}\mbox{ when }k \mbox{ is even}$$
$$\hat A_k=c_{2^{k-2}},  \hat B_k=c_{2^{k-1}}\mbox{ when }k\mbox{  is odd}.$$ 
When $k$ is even then 
$$\hat A_k=c_{2^{k-1}}\approx b_{k-1}^{\beta+1} \approx b_k^{(\beta+1)/2}=o(b_k)$$  
and when $k$ is odd then 
$$\hat A_k=c_{2^{k-2}}\approx b_{k-2}^{\beta+1} \approx b_k^{(\beta+1)/2}=o(b_k).$$
So in either case there exists no $\tau>0$ so that $[\hat A_k,\hat B_k]$ is a $\tau$-scaled neighbourhood
of $[a_k,b_k]$ for $k$ large. In other words,  there is no Koebe space (on the left) for the diffeomorphic extension of the first entry map
into $[a_k,b_k]$. 

\paragraph{Improved Koebe Space for the semi-extension and the proof of Theorem 3 (Big Bounds).}

We can now prove Theorem 3 and an  improved version of Lemma~\ref{lem10}: 

\begin{lemma}[Improved Koebe Space] \label{lem10'}  The range of the map $E_k|J_k$ can be monotonically semi-extended to a $\tau_k$ scaled neighbourhood of
 $[a_k,b_k]$ where $\tau_k \approx b_{k-2}/b_{k} \approx b_k^{-1/2}$ when $k$ is even and  $\tau_k \approx 1$ for $k$ odd. 
 \end{lemma}
 \begin{proof} The  map $E_k|J_k$ can be monotonically semi-extended onto  $[A_k,B_k]$. 
 As we saw in Lemmas~\ref{lem9} and \ref{lem10} we have $|A_k|\ge b_{k-2}$ for $k$ even.
By Lemma~\ref{lem2} and the previous bounds, we have for $k$ even 
  $B_k=c_{2^{k-2}}\approx b_{k-2}$. 
 It follows from this and  (\ref{bkquadratic})  that $\tau_k\approx b_{k-2}/b_k \approx b_k^{-1/2}$. 
 Note that for $k$ odd, $B_k=b_{k-1}$ and so  $\tau_k = b_k/B_k = b_k/b_{k-1}\to \lambda$
 as $k\to \infty$ and $k$ odd.
  \end{proof}

\paragraph{Proof of Theorems~\ref{thm:renormalizationlimit} and \ref{thm:renormalizationlimit2} (Renormalization limits of $R^k$):}  
Given the previous lemma, we obtain that the Koebe space is of the order 
 $\tau_k\approx b_k^{-\frac{1}{2}}$. It follows that $O(b_k/\tau_k)=O(b_k^{\frac{3}{2}})$  and so  (\ref{eq:distortionE_k}) gives
 \begin{equation}
f^{2^k} (x) = 
\begin{cases}
c_{2^k} - s_k |x| +O(b_k^{\frac{3}{2}}) & \mbox{ when }x\in [a_k,0]\\ 
c_{2^k} - t_k x^\beta +O(b_k^{\frac{3}{2}}) & \mbox{ when }x\in [0,b_k] 
\end{cases}
\label{eq:f2kbis}
\end{equation}
with 
\begin{equation}s_k\sim  \dfrac{b_k^{1-\beta}}{K_0}\mbox{ and }
t_k\sim b_k^{1-\beta}.\label{eq:sk''} \end{equation}

\medskip 

The proof of Theorem~\ref{thm:renormalizationlimit2} follows the above and  an explicit calculation. For example, 
$$\lim_{k\to \infty}  (R^{2k+1} f ) \circ  m_{2k+1}$$ is  composition of the asymptotically linear left branch 
of $R^{2k}f$ and of the part of the right branch of $R^{2k}f$ corresponding to  $[b_{k+1},c_{2^k}]$ where 
$c_{2^k}\sim b_k$.

\paragraph{Improved scaling law from $b_k$ to $b_{k+1}$ when $k$ is even.} 
Arguing as in (\ref{eqabc}) and below we have 
\begin{equation}
c_{2^k}-t_k \lambda_{k}^\beta b_k^\beta = \lambda_k b_k + O(b_k^{\frac{3}{2}})
\label{eqabc-bis} 
\end{equation}
and therefore 
$$
b_k- \lambda_{k}^\beta b_k + O(b_k^\beta) = \lambda_k b_k + O(b_k^{\frac{3}{2}})
$$
This means 
 $$
b_k- \lambda_{k}^\beta b_k  = \lambda_k b_k + O(b_k^{\frac{3}{2}}) + O(b_k^{\beta})
$$
and so 
\begin{equation} \lambda_k = \lambda + O(b_k^{\frac{1}{2}})+O(b_k^{\beta-1})\label{eq:lkimp} \end{equation} 
where as before $\lambda\in (0,1)$ is the root of $1-\lambda^\beta=\lambda$. 
In the same way, we obtain that the $\sim$ expressions in this Section~\ref{sec:scaling}  
 are in fact equalities with a multiplicative error of the form  $1+O(b_k^\epsilon)$ for some $\epsilon>0$. 

 One can similarly also obtain exponential convergence for the constants in the scaling  for $b_{k+1}$ to $b_{k+2}$.

\paragraph{The growth rate of $\log b_k$ and the completion of  the proof of Theorem~\ref{thm:scalings}.}  
Let $\mu_{k}=\log (1/ b_{2k})$.
As we saw $\mu_{k}\to \infty$. Let us give a sharper estimate here.  
According to (\ref{bkquadratic}) 
$\mu_{k+1}=2\mu_{k} + D_{k}$ for all $k\ge 0$ where 
\begin{equation}
D_{k}\sim D :=\log ( \beta^{\frac{2}{\beta-1}} K_0^{\frac{-1}{\beta-1}}) .
\label{eq:Dk}\end{equation} 
It follows that $\mu_{k}/2^k=(\mu_0 + D_{k-1}/2^{k} + \dots + D_0/2)$
and therefore there exists $\Theta>0$ so that $\dfrac{\mu_{k}}{2^k}\to \Theta$.
Moreover, 
$$\Theta - \mu_{k}/2^k = \sum_{i\ge k}  D_{i}/2^{i+1}= \sum_{i\ge k}  D/2^{i+1} + 
\sum_{i\ge k}  (D_{i}-D) /2^{i+1}=D /2^{k} +o(1)/2^k.$$
Hence
\begin{equation} 
 \log(1/b_{2k+1}) \sim \log(1/b_{2k})  = \mu_k = 2^k \Theta - D +  o(1)  \label{bklog} \end{equation} 
and  so using (\ref{eq:Dk}) 
 \begin{equation}  1/b_{2k} =  \beta^{-\frac{2}{\beta-1}} K_0^{\frac{1}{\beta-1}}   \exp(2^k \Theta + o(1)).
\label{bkasumpt}  \end{equation}

\paragraph{Necessary and sufficient invariants for $h\colon \{c_{2^k}\}_{k\ge 0} \to \{\tilde c_{2^k}\}_{k\ge 0}$ to be
 Lipschitz.}
 Assume that $h\colon \{c_{2^k}\}_{k\ge 0} \to \{\tilde c_{2^k}\}_{k\ge 0}$ is a conjugacy between $f$ and $\tilde f$
 and  is Lipschitz at $0$.  % In the following we will use the asymptotic expression from
%Theorem  \ref{thm:scalings}.
This implies %Since $f,\tilde f$ are Lipschitz conjugate, 
\begin{equation} \tilde c_{2^{2k}} \approx  c_{2^{2k}}, \tilde c_{2^{2k+1}} \approx c_{2^{2k+1}}.\label{eq:Lip} \end{equation} 
Since  $b_{2k+1}  \sim  \lambda b_{2k}$ , $ c_{2^{2k}} \sim  b_{2k}$
 where $\lambda\in (0,1)$ is the root of the equation $\lambda^\beta +\lambda=1$,
(\ref{eq:Lip})  implies 
\begin{equation}
 \tilde b_{2k}\approx  b_{2k} \mbox{ and }   \tilde \lambda^{-1} \tilde b_{2k+1}\approx  \lambda^{-1} b_{2k+1}\label{eq-btbLip} \end{equation} 
%Since $b_{2k+2}  \sim   \beta^{\frac{-2}{\beta-1}} K_0^{\frac{1}{\beta-1}}  \lambda^{-2}  b_{2k+1}^2 $ this implies 
%$$ \tilde b_{2k+2}  \sim   \tilde \beta^{\frac{-2}{\tilde \beta-1}} \tilde K_0^{\frac{1}{\tilde \beta-1}} \tilde  \lambda^{-2}  \tilde b_{2k+1}^2 \sim \rho^2   \tilde \beta^{\frac{-2}{\tilde \beta-1}} \tilde K_0^{\frac{1}{\tilde \beta-1}} \ \lambda^{-2}   b_{2k+1}^2 \sim \rho^2   \dfrac{\tilde \beta^{\frac{-2}{\tilde \beta-1}} \tilde K_0^{\frac{1}{\tilde \beta-1}}}{
% \beta^{\frac{-2}{ \beta-1}}  K_0^{\frac{1}{\beta-1}}} b_{2k+2} 
%$$
%and since $\tilde b_{2k+2}\sim \rho b_{2k+2}$ it follows that 
%\begin{equation} 
% \beta^{\frac{-2}{ \beta-1}}  K_0^{\frac{1}{\beta-1}} =  \rho \tilde \beta^{\frac{-2}{\tilde \beta-1}} \tilde K_0^{\frac{1}{\tilde \beta-1}}.\end{equation} 
By Theorem  \ref{thm:scalings} and (\ref{eq:Lip})  we also have  
\begin{equation}
%\begin{array}{rl}  
  -\tilde \beta^{-\frac{\tilde \beta+1}{\tilde \beta-1}} \tilde K_0^{\frac{\tilde \beta}{\tilde \beta-1}} \tilde  \lambda^{-\tilde \beta-1} 
\tilde  b_{2k+1}^{\tilde \beta+1}  \sim \tilde   c_{2^{2k+1}} \approx c_{2^{2k+1}} \approx
  - \beta^{-\frac{ \beta+1}{ \beta-1}} K_0^{\frac{ \beta}{\beta-1}}   \lambda^{- \beta-1} 
 b_{2k+1}^{ \beta+1} . % \end{array} 
 \label{eq3bisLip}\end{equation}
This,  the 2nd expression in (\ref{eq-btbLip}) and $b_{2k+1}\to 0$ imply that
\begin{equation} 
\beta = \tilde \beta\mbox{ and therefore  } \lambda=\tilde \lambda . 
\end{equation}
  Finally   (\ref{bkasumpt}) and (\ref{eq:Lip})  imply that 
\begin{equation} 
1\approx  \tilde c_{2^k} / c_{2^k} \sim 
\tilde b_{2k} / b_{2k} =   %\left[\dfrac{\beta}{\tilde \beta}\right]^{\frac{2}{\beta-1}} 
\left[\dfrac{K_0}{\tilde K_0}\right]^{\frac{-1}{\beta-1}} 
 \exp(2^k (\Theta - \tilde \Theta)  + o(1) \label{finalinvar} )\end{equation} 
Hence 
\begin{equation} \Theta = \tilde \Theta .\label{Thetainvar} \end{equation}   
Thus we have shown that the existence of a Lipschitz conjugacy implies 
\begin{equation} \beta=\tilde \beta\mbox{ and } \Theta=\tilde \Theta  .\label{eq:bothinvariants} \end{equation} 
%\Ds{\textcolor{red}{That $h$ is Lipschitz conjugate when (\ref{eq:bothinvariants}) follows from the above equations.}} 

\paragraph{Necessary and sufficient invariants for  $h\colon \{c_{2^k}\}_{k\ge 0}\to  \{\tilde c_{2^k}\}_{k\ge 0}$ to be
differentiable at $0$.} By the previous paragraph,  (\ref{eq:bothinvariants}) are necessary conditions 
for $h$ to be differentiable at $0$. Let us show that these conditions are also sufficient. So 
assume that (\ref{eq:bothinvariants}) holds. %implies that 
%$h\colon \{c_{2^k}\}_{k\ge 0} \to \{\tilde c_{2^k}\}_{k\ge 0}$ is 
%differentiable at $0$. 
%Let  
%$$\rho:=\left[\dfrac{K_0 }{ \tilde K_0}\right]^{1/(\beta-1)}\mbox{ i.e. } \dfrac{\tilde K_0}{K_0}=\rho^{1-\beta}.$$ 
This and (\ref{bkasumpt}) imply
\begin{equation} \dfrac{\tilde c_{2^{2k}}}{c_{2^{2k}}} \sim 
\dfrac{\tilde b_{2k}}{b_{2k}} \sim 
\dfrac{\beta^{\frac{-2}{\beta-1}} K_0^{\frac{1}{\beta-1}}}{\tilde \beta^{\frac{-2}{\tilde \beta-1}} \tilde K_0^{\frac{1}{\tilde \beta-1}}} 
   \exp(2^k (\Theta-\tilde \Theta) + o(1)) \sim \left(\dfrac{K_0}{\tilde K_0}\right)^{\frac{1}{\beta-1}}:=\rho.
   \label{b2ktwice}
\end{equation} 
By  Theorem  \ref{thm:scalings},  $\tilde \beta=\beta$, $\tilde \lambda=\lambda$ and $b_{2k+1}\sim \lambda b_{2k}$,  $\tilde b_{2k+1}\sim \tilde \lambda b_{2k}$ and the previous expression (and $\rho:=[K_0/\tilde K_0]^{\frac{1}{\beta-1}}$) 
we get %(\ref{bkasumpt}) we have  
\begin{equation}
\begin{array}{rl}  
\dfrac{\tilde   c_{2^{2k+1}}}{c_{2^{2k+1}} } &\sim \dfrac{  -\tilde \beta^{-\frac{\tilde \beta+1}{\tilde \beta-1}} \tilde K_0^{\frac{\tilde \beta}{\tilde \beta-1}} \tilde  \lambda^{-\tilde \beta-1} 
\tilde  b_{2k+1}^{\tilde \beta+1} }{-  \beta^{-\frac{ \beta+1}{ \beta-1}} K_0^{\frac{ \beta}{\beta-1}}   \lambda^{- \beta-1} 
 b_{2k+1}^{ \beta+1}}  =\left[\dfrac{\tilde K_0}{K_0}\right]^{\frac{\beta}{\beta-1}}  \left[\dfrac{\tilde b_{2k+1}}{b_{2k+1}}\right]^{\beta+1}\sim  \\
  &  \sim \left[\dfrac{\tilde K_0}{K_0}\right]^{\frac{\beta}{\beta-1}}  \left[\dfrac{\tilde b_{2k}}{b_{2k}}\right]^{\beta+1} 
  \sim  \left[\dfrac{\tilde K_0}{K_0}\right]^{\frac{\beta}{\beta-1}}
   \rho^{\beta+1} = \rho^{-\beta}  \rho^{\beta+1} = \rho.
%  \sim  \left[\dfrac{\tilde K_0}{K_0}\right]^{\frac{ \beta}{\beta-1}} \left[\dfrac{K_0}{\tilde K_0}\right]^{\frac{-1}{\beta-1}} 
% \exp(2^k (\Theta - \tilde \Theta)) \\
%&  = \left[\dfrac{\tilde K_0}{K_0}\right]^{\frac{ \beta}{\beta-1}} \left[\dfrac{K_0}{\tilde K_0}\right]^{\frac{-1}{\beta-1}} =
%\rho^{-\beta} \cdot \rho  
%\\
%
% & -\tilde \beta^{-\frac{\tilde \beta+1}{\tilde \beta-1}} \tilde K_0^{\frac{\tilde \beta}{\tilde \beta-1}} \tilde  \lambda^{-\tilde \beta-1} 
%\tilde  b_{2k+1}^{\tilde \beta+1}  \sim \tilde   c_{2^{2k+1}} \sim \rho c_{2^{2k+1}} \sim
%\\ 
%& 
%\quad \quad  - \rho  \beta^{-\frac{ \beta+1}{ \beta-1}} K_0^{\frac{ \beta}{\beta-1}}   \lambda^{- \beta-1} 
% b_{2k+1}^{ \beta+1} . 
% 
 \end{array} 
 \label{eq3bis}\end{equation}
 
\paragraph{Another ratio.}  Even though we shall not use this, let
us calculate another ratio. 
Writing as before $c_{2^{2k+1}}= - C_{2k} b_{2k}^{\delta}$ we have 
according to (\ref{eq:delta}) and (\ref{eq:Ck}) we have 
$\delta=\beta+1$ and
 $C_{2k}\sim \left[ \dfrac {K_0^\beta} {\beta^{\beta+1}}  \right] ^{1/(\beta-1)}$. 
 
Hence,  using (\ref{k3k}), we obtain
 \begin{equation} 
c_{2^{2k}}-c_{3\cdot 2^{2k}}\sim \dfrac{C_{2k}}{K_0} b_{2k}^2 \sim  \dfrac{K_0^{1/(\beta-1)}}
 {\beta^{(\beta+1)/(\beta-1)}}
%\left[ \dfrac {K_0^\beta} {\beta^{\beta+1}}  \right] ^{1/(\beta-1)} 
b^2_{2k} 
\label{k3k'} \end{equation} 
So assuming that (\ref{eq:bothinvariants}) holds we have using    (\ref{b2ktwice})
$$\dfrac{\tilde c_{2^{2k}}-\tilde c_{3\cdot 2^{2k}}}{c_{2^{2k}}-c_{3\cdot 2^{2k}}}\sim 
 \dfrac{\tilde K_0^{1/(\beta-1)}}{K_0^{1/(\beta-1)}} \dfrac{\tilde b_{2k}^2}{b_{2k}^2} 
 \sim  \dfrac{\tilde K_0^{1/(\beta-1)}}{K_0^{1/(\beta-1)}} \rho^2 = \rho
$$

%This and the 2nd expression in (\ref{eq-btb}) imply that
%\begin{equation} 
%\beta = \tilde \beta\mbox{ , } \lambda=\tilde \lambda . 
%\end{equation}  Inserting this in equation  (\ref{eq3bis}) gives 
%\begin{equation}
% \rho^{\beta+1} \tilde K_0^{\frac{ \beta}{ \beta-1}}  = \rho
%  K_0^{\frac{ \beta}{\beta-1}}  \mbox{ and so } 
%   \tilde K_0 / K_0 = \rho ^{1- \beta}
%  \label{eq3bisbis}\end{equation}
%  Finally   (\ref{bkasumpt}) implies that 
%\begin{equation} 
%\rho\sim \tilde c_{2^k} / c_{2^k} \sim 
%\tilde b_{2k} / b_{2k} =   \left[\dfrac{\beta}{\tilde \beta}\right]^{\frac{2}{\beta-1}} 
%\left[\dfrac{K_0}{\tilde K_0}\right]^{\frac{-1}{\beta-1}}  \left[\dfrac{\lambda}{\tilde \lambda}\right]^{2} 
% \exp(2^k (\Theta - \tilde \Theta)  + o(1) \label{finalinvar} )\end{equation} 
%Hence 
%\begin{equation} \Theta = \tilde \Theta .\label{Thetainvar} \end{equation}   
%Thus we have shown that the existence of differentiable conjugacy implies 
%\begin{equation} \beta=\tilde \beta\mbox{ and } \Theta=\tilde \Theta  .\label{eq:bothinvariants} \end{equation} 
%

%Note that when $\beta=\tilde \beta$ then $\lambda=\tilde \lambda$ 
%and (\ref{eq:Ks}) together with the definition of $D$ in (\ref{eq:Dk}),
%imply 
%$$D- \tilde D=\dfrac{1}{1-\beta} \log\dfrac{\tilde K_0}{K_0}  = \log \rho, \mbox{ i.e.} \exp(D-\tilde D)=\rho $$
%and so (\ref{finalinvar}) reduces to 
%$0 \sim  \exp(2^k (\Theta - \tilde \Theta)  + o(1))$. 

\paragraph{The invariants (\ref{eq:bothinvariants}) are sufficient for the  conjugacy $h\colon \Lambda\to \tilde \Lambda$  to be differentiable at $0$, where  $ \Lambda$ is the attracting Cantor set $\overline{\cup_{n\ge 0} f^n(0)}$.} 
Regardless whether or not  (\ref{eq:bothinvariants}) holds, there exists a topological conjugacy 
$h\colon \Lambda\to \tilde \Lambda$ between $f$ and $\tilde f$; in fact, in the next section we will show that 
$f,\tilde f$ do not have wandering intervals, and then we will also know that
there exists a topological conjugacy $h$ on the entire space. Let us show now
that the conjugacy $h\colon \Lambda\to \tilde \Lambda$ is necessarily 
differentiable on $\Lambda$ when (\ref{eq:bothinvariants}) is satisfied. 
%Rigidity and the Theorem~\ref{thm:nouniversality}.
%Let us now show that if the invariants 
%\begin{equation} \beta=\tilde \beta\mbox{ and } \Theta=\tilde \Theta  \label{eq:bothinvariants} \end{equation} 
%are satisfied 
%that then the conjugacy between $f$ and $\tilde f$ is differentiable at $0$. 

To do this, note that when $k$ is even that $\Lambda\cap [a_k,b_k]$ is contained in 
the union of following intervals $U_k,V_k,W_k,X_k$
where  $U_k=[x_k,c_{4\cdot 2^{k}}]$ where $x_k<0$ is chosen so that $f(x_k)=f(c_{4\cdot 2^{k}})$ and let
$U_k^-=[x_k,0]$, $U_k^+=[0,c_{4\cdot 2^{k}}]$,  $V_k=f_l(U_k^-)$, $W_k=f_r(V_k)$ and $X_k=f_l(W_k)$. 
For simplicity also define $R_k:=[X_k,V_k]$, $L_k=[W_k,U_k]$ and  $(U_k,X_k):=[c_{4\cdot 2^k},c_{3\cdot 2^k}]$. 

\begin{figure}[htp]
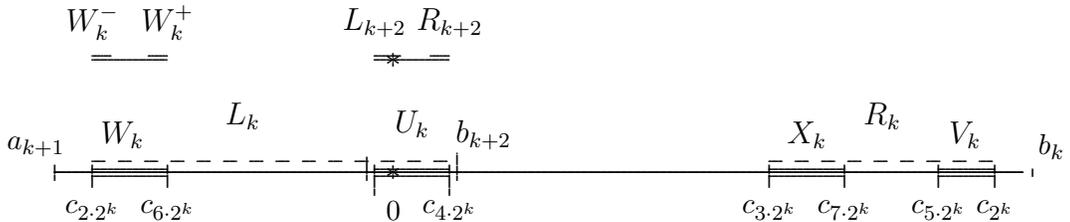
 \hfil
\beginpicture
\dimen0=0.5cm
\setcoordinatesystem units <\dimen0,\dimen0>  point at 0 0
\setplotarea x from -10 to 10, y from -1 to 1
\setlinear
\plot -13 0 13 0 / 
\put {\small $a_{k+1}$} at -13.5 0.7
  \plot -13 -0.2 -13 0.2  / 
\put {\small $c_{2\cdot 2^{k}}$} at -12 -1
  \plot -12 -0.4 -12 0.2  / 
\put {\small $c_{6\cdot 2^{k}}$} at -10 -1
  \plot -10 -0.4 -10 0.2  / 
  \plot -12 0.1 -10 0.1 / 
    \plot -12 -0.1 -10 -0.1 / 
  \put {$W_k$} at -11.2 1 
  \plot -12 3 -10 3 /  %%% Top level picture 
  \plot -12 3.1 -11.5 3.1 / 
  \plot -10.5 3.1 -10 3.1 / 
    \put {$W_k^-$} at -12 4 
       \put {$W_{k}^+$} at -10 4 
  \setdashes    \plot -12 0.3 -2.5 0.3 / 
  \setsolid 
  \put {$L_k$} at -8 1.5 
 
\plot -4.7 -0.2 -4.7 0.4  / 
  \plot -4.5 -0.6 -4.5 0.2  /  
  \put {\small $0$} at -4 -1 
  \put {$*$} at -4 0 
  \put {$U_k$} at -3.5 1.3 
   \put {\small $c_{4\cdot 2^k}$} at -2.5 -1 
\plot -2.3 -0.2 -2.3 0.5  /  
\put {$b_{k+2}$} at -1.6 1
   \plot -4.5 3 -2.5 3 /  %%% Top level picture 
   \plot -4.5 3.1 -3.8 3.1 / 
   \plot -3 3.1 -2.5 3.1 / 
   \put {$*$} at -4 3  
    \put {$L_{k+2}$} at -4.5 4 
       \put {$R_{k+2}$} at -2.5 4 
       %%% End top level picture 
     \plot -2.5 -0.6 -2.5 0.2  /  
     \plot -4.5 0.1 -2.5 0.1 / 
         \plot -4.5 -0.1 -2.5 -0.1 / 
%   \put {\small $0$} at -0 -1 
%        \plot 0 -0.2 0 0.2  / 
 \put {\small $c_{7\cdot 2^k}$} at 8 -1
      \plot 8 -0.4 8 0.2  / 
 \put {\small $c_{3\cdot 2^k}$} at 6 -1 
       \plot 6 -0.4 6 0.2  / 
         \put {$X_k$} at 7 1 
              \plot 6 0.1 8 0.1 / 
                 \plot 6 -0.1 8 -0.1 / 
        \put {\small $c_{5\cdot 2^k}$} at 10.5 -1
      \plot 10.5 -0.4 10.5 0.2  / 
 \put {\small $c_{2^k}$} at 12 -1 
       \plot 12 -0.4 12 0.2  / 
       \plot 10.5 0.1 12 0.1 / 
            \plot 10.5 -0.1 12 -0.1 / 
         \put {$V_k$} at 11.2 1 
         \plot -13 0 13 0 / 
         \plot 13 -0.1 13 0.1 / 
          \setdashes    \plot 6 0.3 12 0.3 / 
  \setsolid 
    \put {$R_k$} at 9 1.5 
\put {\small $b_{k}$} at 13.5 0.7
   \endpicture
\caption{\label{fig:order2} These four intervals contain the postcritical set in $[a_k,b_k]$. We will pull back the analogue of the 
dashed intervals for level $k+2$ inside $W_k$.}
\end{figure}

\begin{lemma} 
\begin{equation} \liminf \dfrac{|W_k|}{|L_k|} >0. \label{eq:wklk} \end{equation} 
and 
\begin{equation}
\dfrac{|R_k|}{|(U_k,X_k)|} \to 0 \mbox{ and } \dfrac{|L_k|}{|(U_k,X_k)|}\to 0\mbox{ as }k\to \infty .
\label{eq:rklk} 
\end{equation} 
\end{lemma} 
\begin{proof} 
Note that $|U_k^-|=|x_k|\approx |c_{4\cdot 2^k}|^\beta \sim b_{k+2}^\beta \approx b_{k}^{2\beta}$,
$$|V_k|=|c_{2^k} - c_{5\cdot 2^k} | = |f_l(U_k^-)| \approx s_k |U_k^-| \approx b_k^{1-\beta} b_k^{2\beta}  =b_k^{1+\beta}$$
and by (\ref{eq:derboundary}), 
$$|W_k|=|f_r(V_k)| \approx \beta b_k^{1+\beta} \approx |c_{2^{k+1}}-0|$$
where in the last $\approx$ we used  (\ref{bkquadratic''}). This implies that the size of $W_k$ is comparable
to its distance to $0$; in other words  for any two points $u_k,v_k\in W_k$ we merely have $u_k\approx v_k$, 
showing (\ref{eq:wklk}). 
To prove (\ref{eq:rklk}), note that 
$$|U_k| \sim |U_k^+| = |c_{2^{k+2}}|\sim b_{k+2} \approx b_k^2$$
and therefore 
$$|L_k|=|[W_k,U_k]| \approx b_k^{1+\beta} + b_k^2 \approx b_k^2.$$
Similarly, by (\ref{k3k}) and $\delta=1+\beta$ we have 
\begin{equation} |R_k|=|[X_k,V_k]| = |c_{2^k}-c_{3\cdot 2^k}| \approx b_k^2.
\label{eq:rkbk} \end{equation} 
These two statements imply $|(U_k,X_k)| \sim |[0,c_{2^k}]|\sim b_k$ and therefore (\ref{eq:rklk}). 
\end{proof} 

It follows from (\ref{eq:rkbk}) that when $u_k\in R_k$ arbitrarily then $u_k\sim b_k$ as $k\to \infty$ and therefore we will  be able to 
use $R_k$ instead of the intervals $X_k$ and $V_k$. 
Equation (\ref{eq:wklk})  will require us to 
choose  much smaller intervals inside $W_k$.

\begin{lemma} Let $W_k^-$ and $W_k^+$  in $W_k$ which are mapped  by $f_r \circ f_l$ 
onto $R_{k+2}$ resp. $L_{k+2}$, where we take $W_k^-$ is to the left of $W_k^+$.
Then  
 \begin{equation} \dfrac{|W_k^-|}{|W_k|},   \dfrac{|W_k^+|}{|W_k|} \to 0. \label{eq:wkpm} 
 \end{equation}
 \end{lemma}
Note that 
\begin{equation} \Lambda \cap [a_k,b_k] \subset  W_k^- \cup W_k^+ \cup U_k \cup X_k \cup V_k. \end{equation} 
 
\begin{proof} 
Since (\ref{eq:rklk}) also holds for $k+2$ replaced by $k$, there exists four intervals
in $U_k$ (with two in $L_{k+2}$ and two in $R_{k+2}$)  so that the gap between $L_{k+2}$ and $R_{k+2}$ 
is huge compared to the size of these two intervals. 
Now consider the orientation reversing map $f_r \circ f_l\colon W_k\to U_k$. 
Since this map has bounded distortion (\ref{eq:wkpm})  holds. 
 \end{proof} 

%Thus we found two intervals inside $W_k$
%which are very small compared to $W_k$ and so that the union of these two intervals 
%contain  $\Lambda\cap W_k$. 

Note that for each $x\in \Lambda  \cap  [a_k,b_k]$
either $x\in [a_{k+2},b_{k+2}]$  or $x$ is contained in one of the sets $X_k$, $V_k$, $W_k^+$ or $W_k^-$. 
Moreover, as we have shown, if $u_k,v_k\in Q_k$ and $u_k\to 0$ where $Q_k$ is either $R_k=[X_k,V_k]$, $W_k^+$ or $W_k^-$ then  $u_k \sim v_k$.

It remains to obtain asymptotic expressions for at least one point in each these intervals. 
Let us start with $W_k^+$. This interval  contains a point $z_k$ so that $f_r\circ f_l(z_k)=0$.
It follows that 
 $$|c_{2^{k+1}}-0|=|f_r(f_l(0))-f_r(f_l(z_k))| \sim \beta |f_l(0)-f_l(z_k)| \sim \beta |z_k|s_k$$
Since $s_k\sim  \dfrac{b_k^{1-\beta}}{K_0}$ and 
$c_{2^{k+1}}\sim  -\left[ \dfrac {K_0^\beta} {\beta^{\beta+1}}  \right] ^{1/(\beta-1)} b_k^{\beta+1}$
%$
%|c_{2^{k+1}}|\sim 
% \beta^{-\frac{\beta+1}{\beta-1}} K_0^{\frac{\beta}{\beta-1}}  \lambda^{-\beta-1} 
% b_{k+1}^{\beta+1} $
it follows that 
\begin{equation}
   z_k\sim - \dfrac{1}{\beta} \left[ \dfrac {K_0^\beta} {\beta^{\beta+1}}  \right] ^{1/(\beta-1)} b_k^{\beta+1}  \dfrac{K_0}{b_k^{1-\beta}}
=  - \left[ \dfrac {K_0^{2\beta-1}} {\beta^{2\beta}}  \right] ^{1/(\beta-1)} b_k^{2\beta} .
%\mbox{ as } u_k\in W_k^+, u_k\to 0.
\label{eq:eq:wk+} 
 \end{equation} 
Similarly,  $c_{2^{k+1}} \in W_k^-$ and 
%we have the asymptotic expressions for $c_{2\cdot 2^k}=c_{2^{k+1}}\in W_k^-$, namely 
according to  (\ref{bkquadratic''})
 \begin{equation} 
 c_{2^{k+1}}\sim  -\left[ \dfrac {K_0^\beta} {\beta^{\beta+1}}  \right] ^{1/(\beta-1)} b_k^{\beta+1} . \label{eq:wk-}\end{equation}
 Finally, $c_{3\cdot 2^k}, c_{2^k}\in R_k$, 
 by (\ref{eq:c32k}) 
\begin{equation}  c_{3\cdot 2^k}\sim c_{2^k} \sim b_k  . 
\label{eq:xvk} \end{equation} 

Let us now take the homeomorphism $h$ between $\Lambda$ and $\tilde \Lambda$
defined so that $h(f^n(0))=\tilde f^n(0)$ and   show that $h$ is differentiable at $0$, provided
that  $\beta=\tilde \beta$, $\Theta=\tilde \Theta$ and $K_0=\tilde K_0$. 
Because of these assumptions, equation    (\ref{bkasumpt})  gives that for $k\to \infty$ even, 
\begin{equation} \dfrac{\tilde b_k}{b_k} \to \rho := \left[ \dfrac{K_0}{\tilde K_0}\right]  ^{\frac{1}{\beta-1}} =1. \end{equation} 
Let $u_k\in \Lambda$ and take $\tilde u_k=h(u_k)$. 
By renumbering if necessary we may assume that  $u_k\in W_k^-\cup W_k^+\cup X_k\cup V_k$. 
From (\ref{eq:wk-})  follows that for $u_k\in W_k^-$, $\tilde u_k\in \tilde W_k^-$,
$$\tilde u_k/u_k \to [\tilde K_0/K_0 ] ^{(2\beta-1)/(\beta-1)} (\tilde b_k/b_k)^{2\beta} \sim \rho^{1-2\beta} \rho^{2\beta}=\rho. 
 $$
 From  (\ref{eq:eq:wk+}), $u_k\in W_k^+$, $\tilde u_k\in \tilde W_k^+$,  
 $$\tilde u_k/u_k \to [\tilde K_0/K_0]^{\beta/(\beta-1)}  (\tilde b_k/b_k)^{\beta+1} \sim \rho^{-\beta} \rho^{1+\beta}=\rho.$$ 
Finally from (\ref{eq:xvk})  we have $\tilde u_k/u_k\to \rho$ for $u_k\in X_k\cup V_k$ and $\tilde u_k\in \tilde X_k\cup \tilde V_k$.
It follows that $h\colon \Lambda\to \tilde \Lambda$ is differentiable at $0$. 

\paragraph{The invariants (\ref{eq:bothinvariants}) are sufficient for the  conjugacy $h\colon \Lambda\to \tilde \Lambda$  to be differentiable along $\Lambda$, where  $ \Lambda=\overline{\cup_{n\ge 0} f^n(0)}$.} 
\def\N{\mathbb N} 
Let $\Delta_{k,0}=[a_k,b_k]$,  $\Delta_{k,i}=f^i(\Delta_k^0)$, $i=1,\dots,2^k-1$ and $\Delta_k=\cup_{0\le i\le 2^k-1} \Delta_{k,i}$.
Note that $\Lambda=\cap_k \Delta_k$. 
Moreover,  let $\tilde \Delta_{k,i},\tilde \Delta_k$ be the corresponding the sets for $\tilde f$. 
As in \cite[Section VI.9]{MS}, define $\Omega=\{0,1\}^\N$ and a continuous map 
$\phi\colon \Omega \to \Lambda$ defined by
associating to $\omega \in  \Omega=\{0,1\}^\N$ the point  $\cap_k \Delta^{j(k,\omega)}$ where
$j(k,\omega)=\sum_{i=0}^{k-1} \omega(i) 2^j$. Denote the interval $\Delta_{k,j(k,\omega)}$
by $[\omega(0),\dots,\omega(k-1)]_k$ and let  $ [\omega(0),\dots,\omega(k-1)]_{k,\sim}$
be the corresponding interval for $\tilde f$.  Because $f$ has the period doubling
combinatorics,   $$[\omega(0),\dots,\omega(k-1)]_k\subset  [\omega(0),\dots,\omega(k-2)]_{k-1}.$$
Let $\Omega^*$  be the dual Cantor set consisting of all left infinite words 
$$\left\{\omega=\left(\dots,\omega(k),\dots,\omega(1),\omega(0)\right), \omega(i)\in \{0,1\}\right\}$$
with the product topology. 
From the scaling law (\ref{bkasumpt}) we obtain that 
 $$\dfrac{[0,\dots,0,0,0]_{k+2}}{[0,\dots,0,0]_{k}}= (1+\epsilon_k) \exp(2^k(\Theta-4\Theta)).$$
From the calculation in (\ref{eq:Dk})- (\ref{bkasumpt}) it follows that $\prod_{n\ge k}(1+\epsilon_n)$ goes to one as $k\to \infty$.  (In fact, one can show that $\epsilon_n$ tends exponentially fast to zero.)
From the above consideration  we also have that for $j_1,j_2\in \{0,1\}$ 
$$\dfrac{[0,\dots,0,j_1,j_2]_{k+2}}{[0,\dots,0,0]_{k}}
= (1+\epsilon_k) \kappa(\beta,j_1,j_2) \exp(-2^k\Psi(\Theta,\beta,j_1,j_2))$$ 
where $\kappa(\beta, j_1,j_2)>0$ and $\Psi(\Theta,\beta,j_1,j_2)$ are constants which can be computed explicitly as above (and which only depend on 
$\beta,\Theta,j_1,j_2$). Using the fact that the Koebe space 
of the semi-extension of the first entry map from  $\Delta_k^i$ into $\Delta_{k,2^k}\subset \Delta_{k,0}$
tends exponentially fast to infinity, and therefore the non-linearity of the first entry map tends exponentially fast to 
zero, we obtain 
$$\dfrac{[\omega(k+1),\dots,\omega(2),j_1,j_2]_{k+2}}{[\omega(k+1),\dots,\omega(2)]_{k}}  =  (1+\epsilon_k) \kappa(\beta,j_1,j_2) \exp(-2^k\Psi(\Theta,\beta,j_1,j_2)).$$
Hence, as in \cite[Proof of Theorems VI.9.3 and VI.9.1]{MS},  using the property that  $\prod_{n\ge k}(1+\epsilon_n)$ converges to $1$ as $k\to \infty$
and assuming that  (\ref{eq:bothinvariants})  holds we obtain that for each sequence $\omega\in \Omega^*$
$$\dfrac{[\omega(k-1),\dots,\omega(0)]_{k,\sim}}{[\omega(k-1),\dots,\omega(0)]_{k}}$$ converges and the value of the limit depends continuously on $\omega\in \Omega^*$. From this it follows that the conjugacy is differentiable along $\Lambda$.

\section{The Hausdorff dimension of the attracting  Cantor set is zero}
\label{sec:estim-hausd-dimens}
%\section{The proof of Theorem~\ref{thm:hdim}: the Hausdorff dimension of the attracting  Cantor set is zero}
%\label{sec:estim-hausd-dimens}

Recall that for every $k>0$ and $i=0,\ldots, 2^k-1$ we have defined
$\Delta_{k,i}:= f^i([a_k,b_k])$.

Let us make a few observations on locations of certain intervals
$\Delta$ inside their parents. In what follows $k$ is assumed to be
even. First, observe that the both intervals $\Delta_{k+2, 2^k}$ and
$\Delta_{k+2, 3\cdot 2^k}$ belong to $[c_{3\cdot
  2^k},c_{2^k}]$. Secondly,
$\Delta_{k+2, 2\cdot 2^k} \subset [c_{2\cdot 2^k}, c_{4\cdot
  2^k}]$. Also note that all 4 mentioned intervals belong to
$\Delta_{k,0}$.

Using formulas (\ref{k3k}), (\ref{eq:33}) and (\ref{bkquadratic}) we
see that $|\Delta| < C|\Delta_{k,2^k}|^2$ for
$\Delta=\Delta_{k+2, 2^k}$, $\Delta_{k+2, 2\cdot 2^k}$,
$\Delta_{k+2, 3\cdot 2^k}$, $\Delta_{k+2, 4\cdot 2^k}$, where $C$ is
some universal constant.

Fix some integer $1\le i \le 2^k-1$. The distortion of the map
$f^{2^k-i}: \Delta_{k,i} \to \Delta_{k,0}$ is asymptotically small due to
Theorem~\ref{thm:bounds} and Lemma~\ref{thm:koebe} ($k$ is still
assumed even). We know that
$f^{2^k-i}(\Delta_{k,i}) = [a_k, c_{2^k}]$  and this interval is very
close to $\Delta_{k,0}:=[a_k,b_k]$ due to formula (\ref{eq3}).
Hence, if $\Delta \subset \Delta_{k,i}$ is one of four intervals
of the form $\Delta_{k+2, m}$, then $|\Delta| < C|\Delta_{k, 0}||\Delta_{k,i}|$, where
$C$ is another universal constant.  This estimate implies that for any
$\gamma>0$ there exists $k_0$ (depending on $f$) such that if $k>k_0$
and $k$ is even, $|\Delta|^\gamma < \frac 14 |\Delta_{k, i}|^\gamma$. Therefore,
$$
\sum_{i=0}^{4\cdot 2^k-1} |\Delta_{k+2, i}|^\gamma < \sum_{i=0}^{2^k-1}
|\Delta_{k, i}|^\gamma.
$$
Thus we have shown that the Hausdorff dimension of $\Lambda$ is zero.

\section{Absence of any Koebe space for general first entry maps}\label{sec:nokoebe} 

Define $R_k$ to be the first return map to $[a_k,b_k]$. 

\begin{thm}[Theorem \ref{thm:absencekoebe}  - Absence of Koebe space]
For each $\tau>0$ there exists $x$ and $k$ 
so that the maximal semi-extension of the first entry map from $x$ into $[a_k,b_k]$
does {\bf not} contain a $\tau$-scaled neighbourhood of $[a_k,b_k]$. 
\end{thm} 
\begin{proof} 
Assume that $x\in I$ and $n$ is so that $y=f^n(x)$ is a first entry to $[a_{2i-1},b_{2i-1}]$  and 
that in fact $y\in [b_{2i},b_{2i-1}]$.  Moreover, assume that $y'=R_{2i-1}(y)\in [a_{2i},b_{2i}]$. Write $y'=f^m(x)$ so $y'$ is a first entry of $x$ into $[a_{2i},b_{2i}]$  under $f^m$. 
Since $f^m=R_{2i-1}\circ f^n$, the maximal diffeomorphic extension (or even semi-extension) of $f^m$
is at most that of  $R_{2i-1}$. The  diffeomorphic range of the latter map is $[c_{2^{2i-1}},B_{2i-1}]$.
%where $B_{2i-1}=c_{2^{2i-2}}\sim b_{2i-2}\approx b_{2i-1}$  according to Lemma~\ref{lem2} and   Theorem~\ref{thm:scalings}. 
By  Theorem~\ref{thm:scalings} we have $c_{2^{2i-1}}\approx -b_{2i-1}^{\beta+1}$. 

The length of $[a_{2i},b_{2i}]$ is   $\sim b_{2i}\approx b_{2i+1} \approx b_{2i-1}^2$, and 
since  $\beta>1$, therefore the space $[c_{2^{2i-1}},a_{2i}]$ is minute compared to the
size of the interval $[a_{2i},b_{2i}]$ when $i$ large. 
It follows that when $i$ is large, there exists no $\tau>1$ so that the range of the extension 
$[c_{2^{2i-1}},B_{2i-1}]$  contains a $\tau$-scaled neighbourhood of $[a_{2i},b_{2i}]$.
In fact, the range of the extension is also not a $\tau$-scaled neighbourhood of 
 $[a_{2i+1},b_{2i+1}]$ for the same reason. 
\end{proof} 

\begin{figure}[h]
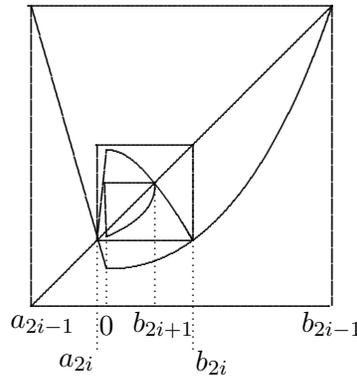
 \hfil
\beginpicture
\dimen0=0.5cm
%\setcoordinatesystem units <\dimen0,\dimen0>  point at 30 0
%\setplotarea x from -4 to 4, y from -3 to 3
%\setlinear
%\plot -4 -4 4 -4  4 4 -4 4 -4 -4 /
%\plot -4 -4 4 4 / 
%\setdots <0.8mm> 
%\plot -4.5 -4 4.5 -4 / 
%\plot 4.5 -4 4.5 -6 / 
%\plot -4.5 -4 -4.5 -5.5 / 
% \plot 1.3 1.3 1.3 -6 / 
% \plot -2.4 1.3 1.3 1.3 / 
%  \plot -2.4 1.3 -2.4 -6 / 
%\setsolid
%\put {$a_k$} at -3.5 -4.5 
%\put {$b_k$} at 3.7 -4.5 
%\put {$b_k'$} at 5.3 -4.5 
%\put {$a_k'$} at -5.3 -4.5 
%\put {$b_{k+1}$} at 1.5 -6.5 
%\put {$a_{k+1}$} at -2.4 -6.5
%\put {\tiny $\bullet$} at 4.5 -4  
%\put {\tiny $\bullet$} at 4 -4
%\put {\tiny $\bullet$} at -4 -4
%\put {\tiny $\bullet$} at -4.5 -4
%\put {\small  $0$} at -2 -5.5 
%\put {\small $d_k$} at -1.3 -4.5 
%\put {\small $e_k$} at -0 -4.5 
%\put {\small $b_k$} at -5 4  
%\put {\small $B_k$} at -5 8.6 
%\put {\small $A_k$} at -5 -6  
%\put {\small $0$} at -5 -2 
%\plot -4.5 -5.8 -4 -4 -1.5 5    -0.5 8.6 /  % 2.5 9 i.e. 1 3.6
%\setquadratic 
%%\plot -2 3 0 1.5 4 -4 /
%\plot -2 3 1.3 1.3 4 -4 /
%\setlinear 
%\plot 4 -4 4.5 -6 / 
%\setdots  <0.8mm>  \setlinear
%\plot -0.5 -4 -0.5 8.6 /  
%\plot -1.7 -4 -1.7 4 / 
%\plot -4 8.6 -0.2 8.6 / 
%\plot -2 -5 -2 3 /
%\plot -4 -2 -2 -2  /
\setcoordinatesystem units <\dimen0,\dimen0>  point at 15 0
\setplotarea x from -4 to 4, y from -4 to 3
\setlinear \setsolid
\plot -4 -4 4 -4  4 4 -4 4 -4 -4 /
\plot -4 -4 4 4 / 
\plot -4 4 -2 -3 / 
\plot -2.25 -2.25 0.3 -2.25 /   % min -2 -3 
\plot 0.3 -2.25 0.3 0.3 -2.25 0.3 -2.25 -2.25 / 
\plot -2.25 -2.25 -2 0.15 / 
\plot -0.71  -0.71 -2.05 -0.71 / 
\plot -2.05 -0.71 -2 -2.15 / 
%\plot -4.8 7 -4 4  / 
%\plot 4 4 5 7 / 
\setquadratic 
%\plot -2 3 0 1.5 4 -4 /
\plot -2 -3 1.3 -1.3 4 4 /
\plot -2 0.15 -1 -0.4 0.3 -2.25 / 
\plot -2 -2.15  -1 -1.5  -0.71 -0.71 /
\setdots <0.8mm>  
\setlinear 
\plot -0.71 -0.71 -0.71 -4 / 
\plot -2 -2.15  -2 -4 / 
\plot  0.3 0.3 0.3 -5.2 / 
\plot -2.25 -2.25 -2.25 -5.2 / 
%\plot -4.8 7 -4.8 -5.5  /
%\plot 5 7 5 -5.5 /
%\plot -5.5 7 5 7 / 
%\plot -1 -4 -1 -7 / 
%\put {\small $a_{k+1}'$} at -5.2 -6 
%\put {\small $b_{k+1}'$} at 5.5 -6 
\put {\small $a_{2i-1}$} at -3.7 -4.5 
\put {\small $b_{2i-1}$} at 4 -4.5 
%\put {\small $a_{k+3}$} at -2.25 -5
\put {\small $a_{2i}$} at -2.8 -5.5  
\put {\small $b_{2i}$} at 0.8 -5.5  
\put {\small $b_{2i+1}$} at -0.49 -4.5 
\put {\small $0$} at -2 -4.5 
%\put {\small $B_{k+1}$} at -6.5 7 
%\put {\small $A_{k+1}$} at -2 -7 
%\put {\small $e_{k+1}$} at 0 -4 .5 
\endpicture
\caption{\label{fig:PD2} The return maps $R_j$ to $[a_j,b_j]$ for $j=2i-1,2i,2i+1$. } 
\end{figure}

\section{Absence of wandering intervals}\label{sec:absencewandering}
%\section{Proof of Theorem~\ref{thm:wandering}: absence of wandering intervals}\label{sec:absencewandering}

\begin{lemma}[The orbit of a potential wandering interval] \label{lem:wand1}
If $f$ has a wandering interval $W$, then 
\begin{enumerate}
\item $W_k:=f^k(W)$ accumulates onto $0$, so for some sequence of $k_j$'s tending to infinity
$W_{k_j}\to 0$;
\item there exists $i_0$ so that if $W_k\subset [a_{2i_0-1},b_{2i_0-1}]$ for some $k$ then 
$W_k \subset  \bigcup_{i\ge i_0}  [b_{2i},b_{2i-1}]$;
\item if $W_k\subset [b_{2i},b_{2i-1}]$ then $W_k\subset [b_{2i},\eta_ib_{2i-1}]$ 
where $\eta_i \to 0$ as $i\to \infty$. 
\end{enumerate}
%there exists a sequence $n(k)$ so that
%$W_{n(k)}\subset  \bigcup_i [b_{2i},b_{2i-1}]$.  Moreover, $$\dfrac{\mbox{dist}(W_{n(k)},0)}{|W_{n(k)}|}\to 0\mbox{ as }k\to \infty.$$ 
\end{lemma}
\begin{proof}
%\Ds{Assume by contradiction that $W$ is a maximal wandering interval for $f$.} 
%We may assume that $W_i\not\ni 0$ for all $i\ge 0$ (otherwise replace $W$ by $W_{i+1}$). 
The sequence of intervals 
$W_i:=f^i(W)$ must accumulate to $0$ for some subsequence $i_j\to \infty$.
Indeed, otherwise there exists a small neighbourhood $U_0$ of $0$  and 
$n_0\ge 0$ so that $f^n(W)\cap U_0= \emptyset$ for all $n\ge n_0$. 
But a theorem of Ma\~n\'e, see \cite{dMvS}[Theorem III.5.1] 
implies that there exists $K>0, \lambda>1$ so that $|Df^n(x)|\ge K\lambda^n$ for all $x\in [a_0,b_0]$
so that $f^i(x)\notin U_0$ for $i=0,\dots,n-1$. Hence the length of the disjoint intervals $f^n(W)$
is growing exponentially with $n$, which of course is  a contradiction.
It follows that $W_i\not\ni 0$ for all $i\ge 0$. 
So for any $k$ there exists  a {\em minimal} $n(k)\ge 0$ so that $W_{n(k)}\subset I_k=[a_k,b_k]$
where $n(k)\to \infty$ as $k\to \infty$.
%So $n(k)$ is non-decreasing with $k$.
 Since all iterates of $W$ are disjoint, $W_i\cap \{a_k,b_k\}=\emptyset$ for all $i\ge 0, k\ge 0$. 
% So if  $n(k)<n(k+1)$ then $W_{n(k)}\subset [a_k,b_k]\setminus [a_{k+1},b_{k+1}]$.
% Let $X$ be the set of integers $k$ so that $n(k)<n(k+1)$ then $X$ 
% is precisely  the set of  integers $k$ so that $W_{n(k)}\subset [a_k,b_k]\setminus [a_{k+1},b_{k+1}]$.

%Take $k\in \N$ and let $k'\in X$ be maximal so that $k'<k$. 

By minimality of $n(k)$,  $W_i\cap [a_k,b_k]=\emptyset$ for all $i<n(k)$. Hence if we take 
$T_k\supset W$ to be the maximal  interval so that  $f^{n(k)}|T_k$ is a diffeomorphism then by Lemma~\ref{lem1} there exists $\tau>1$
so that  $f^{n(k)}(T_k)$ contains $[\tau a_k,\tau b_k]$.

(1) Let us first show that $W_{n(k)}$ lies to the right of $0$ for all $k$ large.  Indeed, assume by 
contradiction that there exists infinitely many $k$'s so that  $W_{n(k)}\subset [a_k,0]$. 
For each such $k$, $f^{n(k)}(T_k)\supset [\tau a_k,\tau b_k]$ is a scaled-neighbourhood of $W_{n(k)}$. By Koebe it follows that $T_k$ 
also contains a  $\tau'$-scaled neighbourhood of $W$ where $\tau'>0$ is the same for infinitely many $k$'s. This 
shows that there exists an interval $W'\supset W$
which strictly contains $W$ on which all iterates of $f$ are diffeomorphic, contradicting the maximality of $W$. 

(2)  Let us now show that there exists $k_0$ so that if $k\ge k_0$ is even then $W_{n(k)}$ cannot be contained in $[b_{k+1},b_k]$. Indeed, when $k$
is even  then by Theorem~\ref{thm:scalings}, $[\tau a_k,\tau b_k]$ is a scaled neighbourhood of $[b_{k+1},b_k]$ and so as in the previous case
we obtain a contradiction. 

From (1) and (2) it follows that for all $k$ large, $W_{n(k)}$ is contained in $\bigcup_i [b_{2i},b_{2i-1}]$.
Similarly to (2), we have that if $W_{n(k)}$ is contained in $[b_{2i},b_{2i-1}]$ then 
in fact it is contained in $[b_{2i},\eta b_{2i-1}]$ where $\eta\in (0,1)$ is small when  $i$ is large. 
Here we use that $W_{n(k)}$ must be contained in a fundamental domain of the fixed point
$b_{2i-1}$ of $R_{2i-1}$. 
\end{proof}

As above let $n(k)\ge 0$ be minimal so that $W_{n(k)}\subset I_k=[a_k,b_k]$. 
From the previous lemma it follows that   $W_{n(k)}$ is contained in $[b_{2i},b_{2i-1}]$ for some $2i-1\ge k$
and therefore $n(2i-1)=n(k)$. 
The first return map $R_{2i-1}$ to $[a_{2i-1},b_{2i-1}]$
is drawn in Figure~\ref{fig:PD2}  on page~\pageref{fig:PD2} and satisfies $R_{2i-1}(x)<x$ for $x\in [0,b_{2i-1}]$.
It follows that there exists $m_k\ge 1$ so that   
\begin{equation}
R_{2i-1}^j(W_{n(k)}) \subset [b_{2i},b_{2i-1}]\mbox{ for all }0\le j < m_k \label{eq:iterw1} \end{equation}
and then for some $i'>i$, 
\begin{equation}
R_{2i-1}^{m_k}(W_{n(k)})\subset [b_{2i'},b_{2i'-1}]. \label{eq:iterw2} \end{equation}
In other words, the next first entry into $[a_{2i},b_{2i}]$ is in fact into $ [b_{2i'},b_{2i'-1}]$
and in particular 
$n(2i-1)<n(2i)=\dots=n(2i'-1)$. 

\bigskip

%and since $R_{2i-1}=f^{2^{2i-1}}$ we get $n(2i-1)+m_k\cdot 2^{2i-1}=n(2i'-1)$.
\begin{lemma} $f$ does not wandering intervals.
\end{lemma}
\begin{proof} 
Let us write $R_{2i-1}=\phi_{2i-1}(x^\beta)$ on $[0,b_{2i-1}]$ where $\phi_{2i-1}$ is an orientation preserving 
diffeomorphism. 
For convenience we will write
$\phi$ rather than $\phi_{2i-1}$.  Let us first  obtain an estimate for $\phi$. 
It follows from Lemma~\ref{lem10} and part (3) of Lemma~\ref{lem:wand1}
%that there exists an orientation preserving diffeomorphism $\phi_{2i-1}\colon \R\to \R$ so that 
%$R_{2i-1}(x)=\phi_{2i-1}(x^\beta)$ for all $x\in [b_{2i},b_{2i-1}]$ and so that 
$|\phi'(x)/\phi'(\hat x)-1|\le \epsilon$
for all $x,\hat x\in  [b_{2i}^\beta,\eta b_{2i-1}^\beta]$ 
where $\epsilon>0$ is small when $\eta$ is small and $i$ is large.  
It follows that there exists $\gamma>0$ so that % if we take an arbitrary $\hat x \in [ b_{2i}, \eta b_{2i-1}]$ and define  $\gamma=\phi'(\hat x)$ then 
\begin{equation} - \gamma \epsilon\le \phi'(x)-\gamma \le \gamma \epsilon.
\label{ineq:phi'} \end{equation} Since $\phi(0) =c_{2^{2i-1}}<0$ it follows that 
\begin{equation} \phi(0) + (1-\epsilon)\gamma x \le  \phi(x)\le \phi(0)+ (1+\epsilon)\gamma x \le (1+\epsilon)\gamma x.\label{ineq:phi} \end{equation} 
Note that $|c_{2^{2i-1}}|\approx | b_{2i-1}^{\beta+1}|<<|b_{2i-1}|$ and therefore 
 $R_{2i-1}(b_{2i-1})=b_{2i-1}$ implies that $\gamma \approx b_{2i-1}^{1-\beta}$.

%Now let us show that in logarithmic coordinates $R_{2i-1}$ is expanding. 
%So define $y=l(x)=-\log(x)$ and $p(x)=x^\beta$. Then $l^{-1}(y)=e^{-y}$ and 
%$$D(l\circ R_{2i-1}\circ l^{-1})(y)=D(l\circ \phi\circ p \circ l^{-1})(y)=  \dfrac{\phi'(e^{-y\beta})}{\phi(e^{-y\beta})}\beta e^{-y\beta}. $$
%Taking $\hat x=l^{-1}\circ p (y)$ and taking $\gamma=\phi'(\hat x)$ the above expression is therefore, using (\ref{ineq:phi}),
%$$\dfrac{\phi'(e^{-y\beta})}{\phi(e^{-y\beta})}\beta e^{-y\beta} \ge \dfrac{\gamma}{(1+\epsilon)\gamma e^{-y\beta}} \beta e^{-y\beta} =\frac{\beta}{1+\epsilon}>1$$
%provided $\epsilon>0$ is small enough. As the length of the interval $[b_{2i},b_{2i-1}]$ is not bounded in logarithmic coordinates, 
%this is not enough. 

From (\ref{bklog}) we have 
$\log (1/b_{2i-1}) \approx 2^i$,  $\log (1/b_{2i}) \approx 2^{i+1}$, 
 and therefore  $\log(\log(1/b_{2i-1})) \approx i \log 2 + O(1)$, $\log(\log(1/b_{2i})) \approx (i+1) \log 2 + O(1)$
   and so  
the length of the intervals $[b_{2i},b_{2i-1}]$ is bounded in double logarithmic coordinates. 

Let us show that  $R_{2i-1}$ is expanding in double logarithmic coordinates. 
So define $l_2(x)=\log(\log (1/x))$ where we assume $x\in [b_{2i},\eta b_{2i-1}]$. Then 
$$Dl_2(x)=\dfrac{-1}{x\log(1/x)}\mbox{ and }x=l_2^{-1}(y)=e^{-e^y}.$$ 
Moreover,
$$D(l_2 \circ R_{2i-1}\circ l_2^{-1})(y)= D(l_2 \circ  \phi\circ f \circ l_2^{-1}  )(y)= \dfrac{\phi'(e^{-\beta e^y}) (\beta e^y) e^{-\beta e^y}}{\phi(e^{-\beta e^y})\log (1/\phi (e^{-\beta e^y}))}$$
Since $x=l_2^{-1}(y)=e^{-e^y}$, $\log x=-e^y$ and $\log (1/x^\beta)=\beta e^y$  this is equal to 
$$
\dfrac{\phi'(x^\beta) x^\beta \log (1/x^\beta)} {\phi(x^\beta)  \log (1/\phi (x^\beta)) }   \ge 
(1-\epsilon) \gamma  \dfrac{x^\beta \log (1/x^\beta)} {\phi(x^\beta)  \log (1/\phi (x^\beta)) }
$$
where in the inequality we used (\ref{ineq:phi'}). 
Since $t\mapsto t \log( 1/t)$ is increasing for $t>0$  small and because of (\ref{ineq:phi}) the latter expression
is bounded below by 
$$\ge (1-\epsilon)\gamma  \dfrac{x^\beta \log (1/x^\beta)}{(1+\epsilon)\gamma x^\beta \log(1/((1+\epsilon)\gamma x^\beta))} 
= \dfrac{ (1-\epsilon)}{(1+\epsilon)} \dfrac{\log (1/x^\beta)}{\log(1/((1+\epsilon)\gamma x^\beta))} 
.$$
Since 
$\gamma \approx  b_{2i-1}^{1-\beta}$,
there exists $C_0>0$ so that this is bounded below by 
$$ \ge \dfrac{1-\epsilon}{1+\epsilon} \,\, \dfrac{\log(1/x^\beta)}{\log(1/x^\beta) +  (1-\beta)\log(1/b_{2i-1})+ \log(C_0)}.
 $$
 Since the latter expression is increasing in $x$ for $x\in [0,b_{2i-1}]$ and since $x\in [b_{2i},b_{2i-1}]$
 this is bounded from below by 
$$ \dfrac{1-\epsilon}{1+\epsilon} \,\, \dfrac{\beta\log(1/b_{2i})}{\beta \log(1/b_{2i}) + (1-\beta)\log(1/b_{2i-1}) + \log(C_0)} .
$$
Since $b_{2i}\approx b_{2i-1}^2$ this is bounded from below by 
$$ \dfrac{1-\epsilon}{1+\epsilon} \,\, \dfrac{2\beta\log(1/b_{2i-1}) + \log(C_0'')}{2\beta \log(1/b_{2i-1}) + (1-\beta)\log(1/b_{2i-1}) + \log(C_0')} \ge \dfrac{2\beta}{1+\beta}- o(\epsilon)>1
$$
 provided $i$ is large and  $\epsilon>0$ is small. 
It follows that in double-logarithmic coordinates $R_{2i-1}$ is expanding on $[b_{2i},\eta b_{2i-1}]$. 

It follows that if $W$ is a wandering interval above, then in double-logarithmic coordinates
the iterates described in (\ref{eq:iterw1}) and (\ref{eq:iterw2}) increase each step 
in length by a factor $(\beta+1)/2$. So their length tends to infinity.  But this violates
that all iterates are contained in $\cup_{i\ge i_0}  [b_{2i},b_{2i-1}]$
because, as we saw,  in double-logarathmic coordinates the length of the intervals $[b_{2i},b_{2i-1}]$
is uniformly bounded from above. \end{proof}

\end{document}